\documentclass[11pt,letterpaper]{amsart}

\usepackage{palatino, amssymb, amsmath, amsthm, color, amsfonts, mathrsfs, enumerate, graphicx}

\usepackage{comment}

\usepackage[longnamesfirst]{natbib}
\bibpunct[ ]{(}{)}{,}{a}{}{,}

\newtheorem{theorem}{Theorem}
    \newtheorem{lemma}[theorem]{Lemma}
    \newtheorem{proposition}[theorem]{Proposition}
    \newtheorem{corollary}[theorem]{Corollary}

        \newtheorem{conjecture}[theorem]{Conjecture}

\theoremstyle{definition} 
    \newtheorem{definition}[theorem]{Definition}
    \newtheorem*{definition*}{Definition}
    \newtheorem{remark}[theorem]{Remark}
    \newtheorem*{remark*}{Remark}
    \newtheorem{example}[theorem]{Example}

\newcommand{\Qt}{{Q^{\otimes 2}}}
\newcommand{\eps}{\varepsilon}
\newcommand{\one}{{\mathbf 1}}
\newcommand{\cd}{\stackrel{d}{\longrightarrow}}

\newcommand{\diag}{\operatorname{diag}}
\newcommand{\basis}{{\rm e}}
\newcommand{\sint}{\int{\raisebox{-.82em}{\hspace{-1.18em}\scriptsize \bf o}}\hspace{.7em}}
\newcommand{\sintn}{\int_{[n]}{\raisebox{-.82em}{\hspace{-1.69em}\scriptsize \bf o}}\hspace{.7em}}
\newcommand{\ssint}{\int{\raisebox{-.33em}{\hspace{-.87em}\scriptsize \bf o}}\hspace{.5em}}
\newcommand{\ssintn}{\int_{[n]}{\raisebox{-.30em}{\hspace{-1.65em}\scriptsize \bf o}}\hspace{1.1em}}
\newcommand{\seq}{\mathbb R^{\mathbb N}}
\newcommand{\bx}{\operatorname{box}}

\definecolor{gray-asparagus}{rgb}{0.0, 0.6, 0.3}

\title{KPZ fluctuations in the planar stochastic heat equation}
\author{Jeremy Quastel \and Alejandro Ram\'\i rez \and B\'alint Vir\'ag}

\begin{document}
\begin{abstract}We use a version of the Skorokhod integral to give a simple and  rigorous formulation of the Wick-ordered  stochastic heat equation with planar white noise, representing the free energy of an undirected random polymer. The solution for all times is expressed as the $L^1$ limit of a martingale given by the Feyman-Kac formula and defines  a randomized shift, or Gaussian multiplicative chaos.   The 
fluctuations far from the centre are shown to be given by the one-dimensional KPZ equation.
\end{abstract}

\maketitle

\tableofcontents

\section{Introduction and main results}

Although the KPZ universality class should describe random planar geometry, until recently models shown to be universal have all been directed, missing  some crucial symmetries. 
In this article we study the planar Wick-ordered heat equation
	\begin{figure}
           \vspace{-4em}
		\centering
		\includegraphics[scale=0.8]{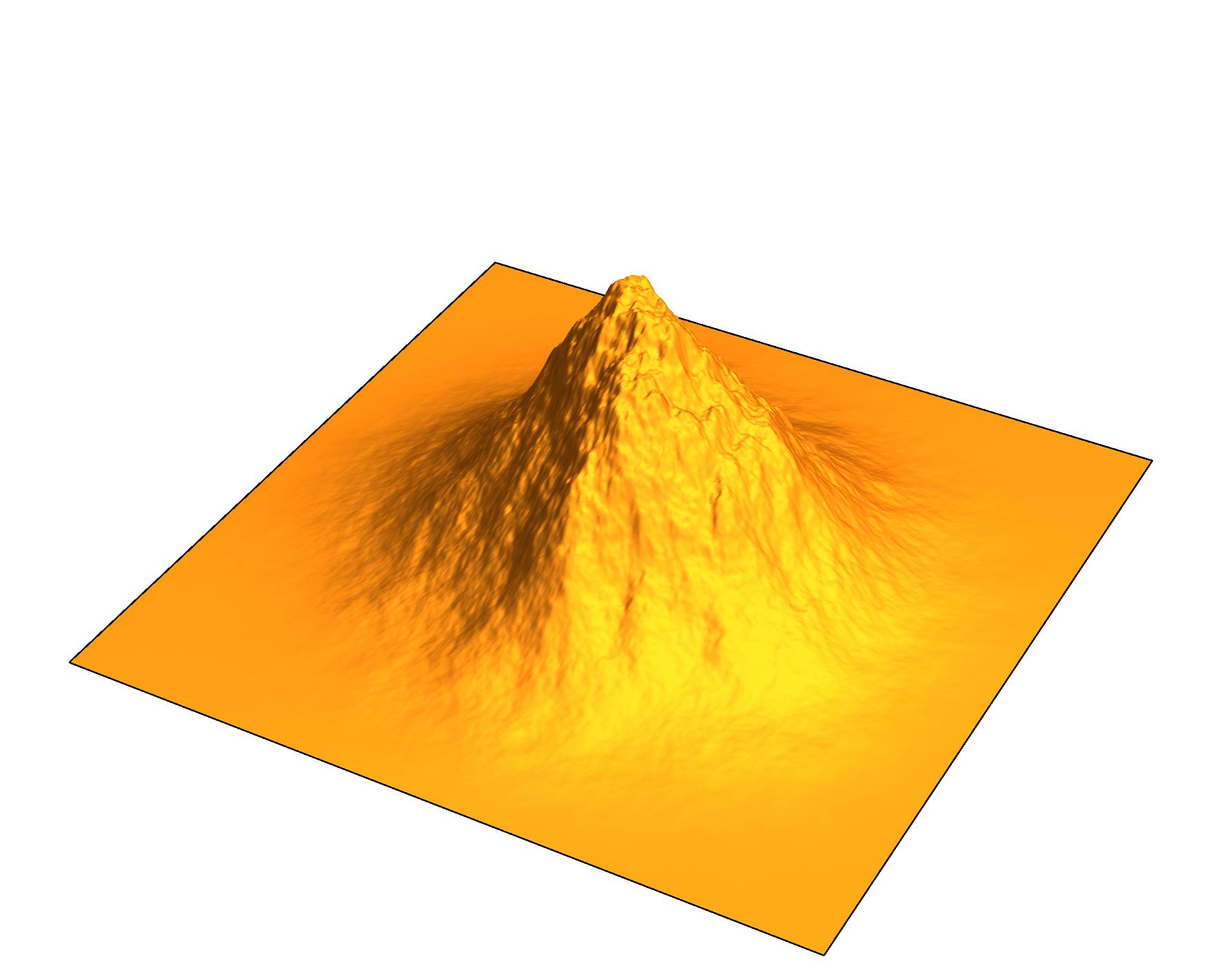}
		\caption{Solution of the planar Wick-ordered heat equation at fixed time started from $\delta_0$ }
			\label{f:wick}
	\end{figure}
\begin{equation}\label{2dshe}
\partial_t u = \tfrac12\Delta u  +   u \,\xi, \qquad u(\cdot,0)=\varsigma
\end{equation}
where $\xi$ is a planar white noise
that does not depend on time (Figure \ref{f:wick}). The precise definition uses the mild form  \eqref{inteq-intro} of the heat equation, in which the white noise term is given by a Skorokhod integral, a close relative of the It\^o integral that does not need a specific time direction. The term "Wick-ordered \!" refers to the fact that the solutions of such equations are given in terms of Gaussian exponentials with quadratic correction.

The solution $u(x,t)$ was known to exist and be expressed in terms of a chaos series, but only up to a critical time, see \cite{NualartZakai} and \cite{HuChaos}, after which the $L^2$ norm blows up, and the chaos expansion fails to converge. The blow-up time, $t_c$, which coincides with the largest time for which the mutual intersection local time of a pair of independent standard planar Brownian motions has a rate one exponential moment, is known exactly.  It is  half the optimal constant in the Gagliardo-Nirenberg-Sobolev inequality, see \cite{Chen} and \cite{BassChen}.

In this article we extend these results in several ways. First, we use an elementary version of the Skorokhod integral to define the solution \eqref{2dshe} for all times, including $t>t_c$. We give a construction of $u$  as a randomized shift, or as the free energy of an undirected polymer in a random environment. This holds for all times as an $L^1$ functional
of the white noise, and  coincides for $t<t_c$ with the chaos series.  Let $p$ denote the planar heat kernel,\begin{equation}\label{planarheatkernel} p(x,t)= \tfrac{1}{2\pi t}\exp\{- |x|^2/2t\}\end{equation} and $(\Xi,\mathcal F, P)$ be a probability space containing a Gaussian sequence that we will use to define our planar white noise and  Skorokhod integral, \emph{planar} in this article always referring to $\mathbb{R}^2$. We have the following definition.
\begin{definition}\label{d:Wick-ordered -she-intro}
Let $\varsigma$ be a finite  measure on $\mathbb R^2$.
Then 
$u\in L^1(\mathbb R^2\times(0,T]\times\Xi)$
is a solution of the planar {\bf Wick-ordered  heat equation \eqref{2dshe}} with initial condition $\varsigma$
if
for almost all $(x,t,\xi)\in \mathbb R^2\times (0,T]\times \Xi$ the random function of $y$,
\begin{equation}\label{dKay}
\mathcal K_{x,t}(u)(y) =  \int_0^t p(x-y,t-s)\, u(y,s) ds,
\end{equation}
is Skorokhod-integrable (Definition \ref{d:Skorokohod-functions}) and satisfies
\begin{equation}\label{inteq-intro}
u(x,t) = \int p(x-y,t)d\varsigma(y) +   \sint \mathcal K_{x,t} (u)~ \xi.
\end{equation}
\end{definition}
 By linearity, it suffices to understand Dirac $\delta$ initial conditions. The explicit solution is given as follows.
\begin{theorem}\label{t:solution-intro}
With $\varsigma=\delta_0$, the planar Wick-ordered  heat equation has the following unique solution. Let $\basis_j$ be bounded functions forming  an orthonormal basis of $L^2(\mathbb R^2)$, let $\xi_j=\langle \basis_j,\xi\rangle$, let $B$ be a planar Brownian bridge from $0$ to $x$ in time $t$ defined on a probability space $(\Omega,\mathcal G, Q)$ independent of $\xi$. Then
\begin{align}\label{approx2}
u(x,t) &=p(x,t)\lim_{n\to \infty} Z_n, \quad\! Z_n=E_{Q}\exp \bigg\{ \sum_{j=1}^n m_j\xi_j - \tfrac12 \sum_{j=1}^n m_j^2\bigg\}, \\ m_j&=\int_0^t \basis_j(B(s))\,ds.
\notag\end{align}
$Z_n$ is a uniformly integrable martingale, so it converges almost surely and in $L^1(P)$ to a limit $Z$. For $t<t_c$, it converges in $L^2(P)$. The solution does not depend on the choice of basis $\basis_j$.
\end{theorem}

The choice of the Skorokhod integral is motivated by the
  interpretation of \eqref{2dshe} as a polymer free energy.  By the Feynman-Kac formula one expects the solution of the stochastic heat equation \eqref{2dshe} to have a representation
\begin{equation}\label{e:Wick:}
u(x,t) =\lim_{n\to\infty} E_{Q} \exp \bigg\{\int_0^t \xi_{[n]}(B(s)) ds -\operatorname{norm}(n,\omega)\bigg\}p(x,t).
\end{equation}
where $\xi_{[n]}$ is a mollified version of the noise, and $\operatorname{norm}$ is some normalization.  
Our solution \eqref{approx2} is exactly in this form, with $\xi_{[n]}=\sum_{j=1}^n \xi_j \basis_j$, a  Gaussian noise with $n$ degrees of freedom. We will use this mollification throughout the paper. 

We expect the random polymer measure $M_\xi$ to be the 
a measure on continuous paths $[0,t]\to \mathbb{R}^2$   written as
\begin{equation}\label{e:Wick2}  M_\xi=\lim_{n\to\infty} \exp \bigg\{\int_0^t \xi_{[n]}(B(s)) ds -\operatorname{norm}(n,\omega) ds \bigg\}Q 
\end{equation}
where $Q$ is the Brownian  bridge law. We might write  the integral  as
$\langle \xi_{[n]}, X\rangle$ where 

$$
X(A) = \int_0^t\one(B(s)\in A)ds
$$
is the occupation measure of the Brownian bridge up to time $t$. Since $X$ is singular with respect to Lebesgue measure,  $\langle \xi_{[n]}, X\rangle$ will not have a limit, as the measure $X$  will not smooth out the limiting white noise. 

However, we expect that the $L^2$ norm $\langle X_n,X_n\rangle$ of the mollified version $X_n=\sum_{j=1}^n \basis_j \langle X,  \basis_j\rangle$  converges, after subtracting a normalizing constant, to twice the two-dimensional self-intersection local time $
    \langle X_n,X_n\rangle - E_{Q_{BM}} \langle X_n,X_n\rangle  \to 2\gamma_t $.
Here  $Q_{BM}$ signifies that the normalization uses Brownian motion from $0$ rather than Brownian bridge from $0$ to $x$. 
Such results are shown in \cite{Varadhan}, \cite{LeGall} for certain mollifications. 

If we take expectation of $u(x,t)$ over the randomness of the white noise, we get 
\begin{equation}
E_P  E_{Q_{BM}}e^{\langle \xi_{[n]}, X\rangle}=   E_{Q_{BM}}  E_P  e^{\langle \xi, X_n\rangle}= E_{Q_{BM}}    e^{\tfrac12\langle X_{n}, X_{n}\rangle}
\end{equation}
which suggests that we could use $\operatorname{norm}(n,\omega)=E_{Q_{BM}}\langle X_n,X_n\rangle  /2$. This breaks the semigroup property slightly, but the problem can be fixed by adding a $t$-dependent term to get $\operatorname{norm}(n,\omega)=E_{Q_{BM}}    \langle X_{n}, X_{n}\rangle/2-\frac{1}{2\pi}t\log t$, see Example \ref{PAM}.

This choice  corresponds to renormalizing the solution of the equation $\partial_t u_n=\tfrac12\Delta u_n +u_n\xi_{[n]}$ as $e^{-E_{Q_{BM}}\langle X_n,X_n\rangle  /2+\frac{1}{2\pi}t\log t}u_n$ to obtain a limit. The result is called PAM in the literature, for parabolic Anderson model, see, for example \cite{hairer2015simple}.  PAM satisfies the usual semigroup property. 
Note however, that from the above construction the expectation of the polymer measure over the noise $\xi$ is \emph{not} the original Brownian bridge $Q$, but instead the re-weighted bridge measure $e^{\gamma_t+\frac{1}{2\pi}t\log t}Q$. So $M_\xi$ is a randomized version of the  $e^{\gamma_t+\frac{1}{2\pi}t\log t}Q$, instead of $Q$.

Starting from the physical assumption that the polymer measure should be a randomized version of the original Brownian bridge, one is led to renormalize by the self-intersections themselves instead of just their expectations.  This corresponds to renormalizing Gaussian exponentials as
\begin{equation}\label{renormexp}
  \exp\{W-\tfrac12 E_P W^2 \} .
\end{equation}
Our renormalization is essentially that, up to the $t$-dependent factor: we use $\operatorname{norm}(n,\omega)=\langle X_n,X_n\rangle  /2$ instead of $E_{Q_{BM}}\langle X_n,X_n\rangle  /2-\frac{1}{2\pi}t\log t$. This leads to \eqref{approx2}, and, remarkably, to the Skorokhod interpretation of the planar SHE. The price is that the resulting solution of \eqref{2dshe} is no longer a semigroup, since the Brownian motions weighted by their renormalized self-intersection local time no longer have the Markov property.

In summary, one has to choose whether to  keep the expected path measure as Brownian bridge, or to keep the semigroup property.   As in the It\^o-Stratonovich dichotomy, the choice depends on which symmetries are more intrinsic to the scientific problem at hand.  One advantage of the Skorokhod approach is that it  leads to a  simple construction of PAM as a re-weighting of the paths in this measure by the exponential of self-intersection local time,  recovering the semigroup property,  Example \ref{PAM}. This will be used in a future article to extend the present results to PAM itself. 
  
  The cutoff \eqref{approx2} is not directly related to standard discrete polymer models. Instead, consider a random environment of i.i.d.\; random variables $\zeta_{i,j}$, $i,j\in \mathbb{Z}$ and a random walk $R$ on $\mathbb{Z}^2$.  The standard energy 
 is \begin{equation}\sum_{n=1}^N \zeta_{R_n}=\sum_{i,j} \zeta_{i,j} m_{i,j},\end{equation} where $m_{i,j}$ is the number of visits of the walk to site $i,j$ up to time $N$. To obtain a polymer model which is a discrete analogue of the 2d Wick ordered polymer, this energy should be balanced by subtracting $ \sum_{i,j} \log E_Pe^{m_{i,j}\xi_{i,j}}$ so that the expectation over the random environment of the exponential of the energy is unity.  Such {\bf balanced polymer models} can be  expected to converge to our $Z$ in the weak noise (intermediate disorder) limit under appropriate moment conditions, although chaos expansions can only work for small times, and the methods developed in this paper only apply directly in the Gaussian case. We leave this for future work.

The Skorokhod integral also has some  pleasant computational properties.
To solve the planar Wick-ordered  heat equation, and prove Theorem 2, we use the fact that projections of this equation to a finite-dimensional Gaussian space satisfy a version of \eqref{inteq-intro}, a key property of the Skorokhod integral. These projections are in fact deterministic finite-dimensional linear PDEs, that can be solved in a pathwise sense explicitly via the Feynman-Kac  formula.  The solutions are given by  $p(x,t)Z_n(x,t)$ as in Theorem \ref{t:solution-intro}, and they are real-analytic in the noise coordinates. 

To understand the convergence of the projections, we think of the sequence $m_j$ as a {\it randomized shift}. A general theory around this notion was developed by \cite{shamov} in studying Gaussian multiplicative chaos. We develop a simple version of this theory, reviewing and establishing existence and convergence theorems along the way.

The Gaussian multiplicative chaos in this setting is the random polymer measure, so we obtain a rigorous version of the polymer measure implied in \eqref{e:Wick:} for free.

\begin{theorem}[GMC representation]\label{t:2dpolymer-intro} With the notation of Theorem \ref{t:solution-intro}, the law of the marginal law of $(\xi_i+m_i, i\ge 0)$ , is absolutely continuous with respect to the law of $(\xi_i,i\ge 0)$ and the Radon-Nikodym derivative is given by the partition function $Z$. The unnormalized random polymer is the unique $P$-random measure $M_\xi$ on $\Omega$ satisfying
\begin{equation}\label{shdef}
E_PE_{M_\xi}F(\xi,\omega)=E_{P\times Q}F( \xi+m(\omega),\omega).
\end{equation}
for all bounded measurable functions $F:\Xi\times \Omega\to \mathbb R$. 
\end{theorem}

An equivalent formulation in the language of random distributions is that the the law of $\xi+X$ is absolutely continuous with respect the law of $\xi$ with density $Z$. The former law is a convolution in the space of distributions, since $\xi$ and $X$ are independent, see Remark \ref{r:rabbit}.

In fact, $M_\xi$ in Theorem \ref{t:2dpolymer-intro}
satisfies $$M_\xi  = \lim_{n\to \infty} e^{ \sum_{j=1}^n m_j(\omega)\xi_j - \tfrac12 \sum_{j=1}^n m_j(\omega)^2}Q$$ in $P$-probability with respect to the weak topology of measures on path space.
Hence,  the statement of equation (\ref{shdef}) can be seen as a version of Girsanov theorem, at least for bounded functions $F$ depending only on $\xi$.

The key technical input to Theorems \ref{t:solution-intro} and \ref{t:2dpolymer-intro} is to show uniform integrability of $Z_n$. For this we use that, even for $t>t_c$, off sets of arbitrarily small probability, the mutual intersection local time of planar Brownian motion has exponential moments.

The solution of the one-dimensional stochastic heat equation with space-time white noise is given by the same simple construction.  Consequences of the construction are a new criterion (Proposition  \ref{propmain}) for convergence in the class of such models and the apparently new observation \eqref{shiftkpz} that the shift representation holds for the KPZ equation and the continuum random directed polymer (CDRP).  Indeed, while the full polymer measures were studied before in a mollified setting (see, for example, \cite{mukherjee2016weak, broker2019localization}), only the partition function or the endpoint law  were considered without mollification. We will use the randomized shift representation of the full polymer measure in the proof of Theorems \ref{t:1dpolymer-intro} and \ref{t:process}.  Two dimensions are crucial here. Although the stochastic heat equation can be defined in 3 dimensions (\cite{hairer2018multiplicative}), the shift representation is not expected to hold. One essential difference is that intersection local time is not defined pathwise.

Recall  the Kardar-Parisi-Zhang equation,
\begin{equation}\label{kpz}
    \partial_t h =\tfrac12 (\partial_x h)^2  + \tfrac12\partial_x^2 h + \xi, \end{equation}
where $x\in\mathbb R$, $t\ge 0$ and $\xi$ is  white noise in one space and one time dimensions, in contrast to the white noise of (\ref{2dshe}), which is spatial in the plane,  but time-independent.
The solution of this equation is defined through the Hopf-Cole transformation $
h(x,t) = \log    z(x,t)$ that turns it into  the one-dimensional multiplicative stochastic heat equation

\begin{equation}\label{she1}
    \partial_t z = \tfrac12\partial_x^2 z + \xi z, \qquad z(0,\cdot) = \delta_0.
\end{equation}
Let 
\begin{equation}\label{onedheatkernel}
q(x,t) = \tfrac{1}{\sqrt{2\pi t} } e^{-x^2/2t}
\end{equation}
be the one-dimensional heat kernel.  

\begin{theorem}\label{t:1dpolymer-intro} The Wick-ordered version, Definition \ref{d:1dshe},  of the one-dimensional stochastic heat equation \eqref{she1}  with space-time noise has the following unique solution:
$$
z(x,t) =q(x,t)\lim_{n\to \infty} E_{Q}\exp \bigg\{ \sum_{j=1}^n m_j\xi_j - \tfrac12 \sum_{j=1}^n m_j^2\bigg\}, \quad m_j =\int_0^t\basis_j(b(s),s)s,
$$
where $(\basis_j)_{j\ge 1}$ are bounded functions forming an orthonormal basis of $L^2(\mathbb R^2)$.
Here $b$ is a one-dimensional Brownian bridge from $0$ to $x$ in time $t$. 
The limit exists in $L^2(P)$ and coincides with the It\^o solution of \eqref{she1}. The corresponding polymer measure $M_\xi$ coincides with the continuum directed polymer constructed in \cite{AKQ2}. It has the shift representation
\begin{equation}\label{shiftkpz}
    E_P E_{ M_\xi} F(\xi,\omega) = E_{P\times Q} F(\xi+m(\omega),\omega).
\end{equation}
\end{theorem}
Next, we show that on appropriate scales, the solution $u$ of the planar Wick-ordered  heat equation \eqref{2dshe}
has KPZ fluctuations. In particular, we prove the following.
\begin{theorem}\label{t:crossover} For any $t>0$ and $a\in \mathbb R$ as $N\to \infty$, with $\varsigma=\delta_0$ we have
\begin{equation}\label{e:crossover}
P(u ( (0,N^{3/2}t),Nt )\times Ne^{N^2t/2}\sqrt{2\pi t}\le a)  \to F_{\operatorname{\it KPZ}}(t,a),
\end{equation}
where 
$
F_{\operatorname{\it KPZ}}(t,a)=P(z(0,t)\le a)
$ are the KPZ crossover distributions,
which have a determinantal representation (see, for example, \cite{ACQ}).
\end{theorem}
We also have a process-level version of Theorem \ref{t:crossover}. To  understand the almost sure dependence on instances of white noise in the next theorem,   couple them as 
\begin{equation}\label{wnrescaled}
 \xi_N(\cdot, \cdot) =N^{-1}\xi(N^{-1/2} \cdot,N^{-3/2}\cdot).
\end{equation}
The law of $\xi_N$ does not depend on $N$, but now these white noises are defined on the same probability space $(\Xi,\mathcal F, P)$.
We introduce extra parameters to both equations,
\begin{equation}\label{e:uz}
\partial_t u = \tfrac 12(\nu \partial_x^2+ \partial_y^2) u  +  \beta u \,\xi_N, \qquad
    \partial_t z = \tfrac12\nu \partial_x^2 z + \beta z\,\xi ,
\end{equation}
and write $u_{\nu,\beta,N}(x_0,y_0;x,y;t)$ for the solution with initial condition $\delta_{(x_0,y_0)}$ Similarly, write $z_{\nu,\beta}(x_0,t_0;x,t)$ for the solution started from $\delta_{x_0}$ at time $t_0$.
Note that for $u$, the noise $\xi_N$ is used as  $\xi_N(x,y)$ and for $z$ the noise $\xi$ means $\xi(x,t)$.

\begin{theorem}\label{t:process} Let $\rho,\nu,\beta,s>0$ and $x,y,t\in \mathbb R$. Let ${\rm p}=(x,t)$ and ${\rm q}=(y,t+s)$. Let ${\rm p}_N=(N^{1/2}x,N^{3/2}t )$ and let ${\rm q}_N=(N^{1/2}y,N^{3/2}(t+s))$.
As $N\to \infty$, we have
\begin{equation}\label{processconv}
 u_{\nu,\beta,N}({\rm p}_N;{\rm q}_N;N\rho s) \times \sqrt{2\pi \rho s}Ne^{N^2s/(2\rho)} \to z_{\nu \rho,\beta}({\rm p};{\rm q})
\;\mbox{ in }L^1(\Xi,\mathcal F, P).
\end{equation}
Moreover, the polymer measure $M_N$ on paths from ${\rm p}_N$ to ${\rm q}_N$ in time $N\rho s$
corresponding to $Z_N$ converges to the 1+1-dimensional continuum directed random polymer measure (CDRP, the polymer measure corresponding to $z$) $M$ on paths $b$ between space-time points  ${\rm p}$ and ${\rm q}$ 
\begin{align*}
\Big((N^{-1/2}&B_N(N\rho r)_1,N^{-3/2}B_N(N\rho r)_2), \;r\in[0,s]\Big) \mbox{ under $M_N$} \\&\;\to\; \Big((b(t+r),t+r), \;r\in [0,s]\Big) \mbox{ under $M$,} \qquad  \mbox{in $P$-probability.}
\end{align*}
\end{theorem}

Since \eqref{e:uz}  is convergence in $L^1$, and not just in law, it holds jointly in all seven parameters  $\beta,\nu,\rho,x,y,s,t$. For example, with all parameters except $y$ fixed, Theorem \ref{t:process} implies  convergence of finite dimensional distributions to those of the top line of the KPZ line ensemble.

As a corollary of Theorem \ref{t:process} we conclude that in the double limit when  $t\to \infty$ after the first limit, the planar Wick-ordered  heat equation converges to the Airy process,  see \cite{QS}, \cite{virag2020heat}. By linearity, Theorem \ref{t:process}  implies convergence for general initial conditions often considered in KPZ. The KPZ fixed point appears in  the double  limit. 

The Feynman-Kac representation 
makes the proof of Theorem \ref{t:process} transparent and explains the $N^{3/2}$ scaling as well.

\begin{proof}[Proof of Theorem \ref{t:process}]
\ \\ Fix an orthonormal basis of $L^2(\mathbb 
R^2)$ consisting of bounded continuous functions  $\basis_j$, and let
$$\basis_{N,j}(x,y):=\frac{1}{N}\basis_j(N^{-1/2}x,N^{-3/2}y).$$ This, together with our  choice of the noise coupling \eqref{wnrescaled} is compatible with the space-time scaling. It is chosen so that the coordinate representation $\xi_j=\langle \basis_{N,j},\xi_N(\cdot) \rangle=\langle \basis_{j},\xi(\cdot) \rangle$ does not depend on $N$. 

We will solve for $u$ in the basis $\basis_{N,j}$. Theorem \ref{t:solution-intro} shows that the solution is basis independent, and, with parameters added in a straightforward way, we can write \begin{align*}
u({\rm p}_N,{\rm q}_N,N\rho s) &=p({\rm q}_N-{\rm p}_N,N\rho s)\lim_{n\to \infty} Z_{N,n}, 
\\Z_{N,n}&=E_{Q}\exp \bigg\{ \sum_{j=1}^n m_{N,j}\xi_{j} - \tfrac12 \sum_{j=1}^n m_{N,j}^2\bigg\}, \\m_{N,j}&=\beta \int_0^{N\rho s} \basis_{N,j}(B_N(r))\,dr,
\end{align*}
where $B_N$ is a two-dimensional Brownian bridge ${\rm p}_N\to {\rm q}_N$ in time $N\rho s$ with covariance matrix $\diag(\nu,1)$. We couple the bridges on a single probability space $Q$ by writing the limiting spatial and temporal directions separately:
$$
B_N(N\rho r)=(N^{1/2}b_{\rm sp}(t+r),N^{1/2}b_{\rm ti}(t+r)+(t+r)N^{3/2}), \qquad r\in [0,s],
$$
where $b_{\rm sp},b_{\rm ti}$ are independent 1-dimensional Brownian bridges run from time $t$ to $t+s$ from $x$ to $y$ and from $0$ to $0$, and having variances $\nu\rho$ and $\rho$ respectively.  
It is convenient to rewrite
$$
m_{N,j}= \beta  \int_0^s  \basis_j(b_{\rm sp}(t+r),N^{-1}b_{\rm ti}(t+r)+t+r)\,dr
$$
as the change-of-variable and scaling factors $N$ cancel (this is why the $N^{3/2}$ scaling is crucial). Then for each $j$,  noting that $N^{-1}b_{\rm ti}(t+r)\to 0$ pointwise in $r$, $Q$-a.s., we conclude that
$$
\lim_{N\to \infty} m_{N,j}=m_j=\beta  \int_0^s  \basis_j(b_{\rm sp}(t+r),t+r)\,dr, \qquad Q-\mbox{a.s.}
$$
and by Theorem \ref{t:1dpolymer-intro}, with parameters added in a straightforward way, we have 
\begin{equation}\notag
  z_{\nu\rho,\beta}({\rm p},{\rm q})=q({\rm q}-{\rm p})\lim_{n\to\infty} E_Q\exp\left(\sum_{j=1}^n m_j\xi_j-\frac{1}{2}m_j^2\right).
\end{equation}
This establishes what the limit should be, including matching the parameters.  We  justify the exchange of limits in $n$ and $N$ and complete the proof in Section \ref{ss:tconvergence}. \end{proof}

The key technical input for Theorem \ref{t:process} is again a mutual intersection local time estimate for a Brownian bridge. This time, we need a \emph{uniform}  bound on exponential moments off a small set for large $t$ and distant endpoint. Fortunately, the large drift in bridges to far away points makes them intersect less and allows us to establish such a bound.

As seen in  proof outline above, the scaling $N^{3/2}$ in Theorems \ref{t:crossover} and \ref{t:process}  preserves the interesting  interaction between $u$ and the white noise $\xi$. But what about the asymptotic behavior of $u((0,N^{\kappa}),N)$ for other values of $\kappa$? When $\kappa>3/2$, we expect that $u$ is close to the heat kernel $p$ and their difference should satisfy an ordinary central limit theorem. A more exciting  regime is $\kappa \in [0,3/2)$: 
\begin{conjecture}
There exists $\kappa_c\in [0,3/2)$ so that 
for all $\kappa \in [\kappa_c,3/2)$, $\kappa\neq \kappa_c$ 
$$
\exists\, a,s\in  \mathbb R^{\mathbb N} \text { s.t.  } \frac {\log u((0,N^{\kappa}),N)-a_N}{s_N}\cd TW  \;\;\; \Leftrightarrow \;\;\; \kappa>\kappa_c,
$$
where $\cd TW$ means convergence in law to the GUE Tracy-Widom distribution.
\end{conjecture}
We expect $\kappa_c$ to be above the diffusive scaling, that is $\kappa_c\ge 1/2$.
It would be interesting just to show strict inequlaty  $\kappa_c>1/2$, or to find any value of $\kappa$ so that the Tracy-Widom limit holds. We note in passing that the results in \cite{konig2020longtime}
seem to suggest  $\kappa_c\ge 1$.
\subsection{History}

Essentially all  work on exact KPZ fluctuations has  been on directed models.
An exception is the breakthrough work on the boundary fluctuations of lozenge tilings of polygons \cite{aggarwal2021edge}.  These models, however, are still non-intersecting line ensembles, and the challenge there is to localize previous methods.

In terms of underlying polymer paths, the technical difficulty is to show that self-intersections do not affect the large-scale behaviour.
Previously the only tool was to identify a regeneration structure (for example, \cite{MR3372855}).  Besides leading to non-explicit coefficients defined in terms of the regeneration times, such methods failed to identify universal fluctuations except in some trivial (Gaussian) situations.

The present paper treats the undirected two dimensional continuous case, represented by the stochastic heat equation and its underlying polymer.  There is a choice of interpretation of the stochastic integral in the Duhamel form of the equation, and here we study the model resulting from the Skorokhod interpretation.  The Skorokhod integral was introduced by \cite{MR0391258}.  Presentations in \cite{MR2200233},
\cite{MR1474726} and \cite{MartinsNotes} are highly recommended. \cite{MR660187} identified it as  the adjoint of the Malliavin derivative, and it is often defined that way. This makes sense for $L^p$, $p>1$, as the Malliavin derivative is closable on the duals $L^q$, $1<q<\infty$.  The technical challenge in our work is the necessity to extend to appropriate $L^1$ integrands so that the integration in the Duhamel form of the stochastic heat equation \eqref{2dshe} makes sense. Finally, the equation is  solved using  martingale theory and an appropriately checkable condition for uniform integrability.

The martingale approximations take the same form as approximations to the Gaussian multiplicative chaos studied by \cite{shamov}.   In particular,  this gives  a representation of the solution as the Radon-Nikodym derivative of a shift of the underlying white noise, in our case by the occupation measure of the underlying paths. Such representations were used in the mollified setting by, for example, \cite{mukherjee2016weak} and \cite{broker2019localization}.

The technical input for the uniform integrability of the approximating martingales are exponential moments of the mutual intersection local times off a small set.  \cite{geman1984local}, \cite{LeGall}, \cite{Chen} and \cite{BassChen} are recommended references for the mutual intersection local time and finiteness of rate one exponential moments without the cutoff.

\subsection{Structure of the article}

In Section \ref{skorokhod} we define an elementary version of the Skorokhod integral, give the main property of projections, give a number of illustrative examples and survey relations to the standard definition, in particular, the fact that in $1$ and $1+1$ dimensions, the integral coincides with the It\^o integral.  Thus the present work describes an alternate, equivalent construction of the KPZ equation with space-time white noise.  In Section \ref{heateq} we solve, explicitly, projected versions of the stochastic heat equation, in the planar as well as the $1+1$ dimensional case.  In Section \ref{randomizedshifts} criteria are introduced for the convergence of the resulting martingales. 
Section \ref{intersectionlocaltime} 
constructs the mutual intersection local times in terms of our coordinate representation and reviews the main scaling properties and bounds.  The key estimates needed to prove the $L^1$ convergence are contained in Section \ref{exponentialmoments}. The pieces are put together and the main theorems are proven in Section \ref{s:ikea}. 

\section{Skorokhod integral}\label{skorokhod}

\subsection{White noise}

Let $(R,\mathcal{B}, \mu)$ be a measure space.  We will always assume $L^2(R,\mathcal{B},\mu)$ is separable, with orthonormal basis $\{\basis_j\}_{j=1,2,\ldots}$.
White noise defined on a probability space $(\Xi,\mathcal{F}, P)$, is a linear isometry 

\begin{equation}
\label{linis}
\xi: f\mapsto \langle f,\xi\rangle
\end{equation}
from $L^2(R,\mathcal{B},\mu)$ to $L^2(\Xi,\mathcal{F}, P)$ such that
$\langle f,\xi\rangle$, $f\in L^2(R,\mathcal{B},\mu)
$
form a mean zero Gaussian family. The isometry property says that
\begin{equation} E_P[\langle f,\xi\rangle\langle g,\xi\rangle]= \langle f,g\rangle_{L^2(R,\mathcal{B}, \mu)}
\end{equation}
for all $f,g\in L^2(R,\mathcal{B},\mu)$. The bracket notation $\langle f,\xi\rangle$ suggests but does not rigorously mean an $L^2(R,\mathcal{R}, \mu)$ inner product, since $\xi$ cannot be realized in $L^2(R,\mathcal{B}, \mu)$.

The white noise can be constructed in the following elementary way, which we use repeatedly.  Define  \begin{equation} \xi_j= \langle \basis_j,\xi\rangle.
\end{equation}From the above definition, these are independent Gaussians with mean $0$ and variance $1$ and for $f\in L^2(R,\mathcal{B}, \mu)$ with $f=\sum_{j=1}^\infty f_j\basis_j$
with $f_j=\langle f,\basis_j\rangle $, we can realize \begin{equation}\langle f,\xi\rangle = \sum_{j=1}^\infty f_j \xi_j\end{equation} which converges in $L^2(\Xi,\mathcal{F},P)$.

We will often think of functions, measures and noises in terms of their representations $f_j=\langle f,\basis_j\rangle$, $\mu_j=\langle \mu,\basis_j\rangle$ and $\xi_j=\langle\xi,\basis_j\rangle$ in the orthonormal basis, and sometimes refer to these as \emph{coordinate representation}.

Therefore in this paper we will take $(\Xi,\mathcal F, P)$ be a probability space given by an independent standard Gaussian sequence: $\Xi=\mathbb R^{\mathbb N}$,  $\mathcal F$ is the Borel $\sigma$-algebra on product space, and $P$ is independent standard Gaussian product measure. It comes equipped with a natural filtration
\begin{equation}
\mathcal{F}_n=\sigma(\xi_1,\ldots,\xi_n).
\label{effen}
\end{equation}

We will mostly be considering \emph{planar white noise}, which is the case $R=\mathbb{R}^2$,
$\mathcal{B}=$ Borel, $\mu=$ Lebesgue measure.

\subsection{Definition of the integral}
The Skorokhod integral is a notion of integral against white noise that does not need a time ordering.  So it is particularly suited to problems involving higher dimensional time-independent white noise.   The definition is elegant, and the concept is powerful, but it is hard to quickly see the motivation behind it. We first give a definition, and then provide  motivation.

\begin{definition}[Skorokhod integral of random sequences]\label{d:skorokhod} \ 

Let $D=\{1,\ldots, \max(D)\}$ be finite or let $D=\mathbb N$. 

Let $\xi_n,n\in D$ be independent standard 
Gaussian random variables defined on some probability space $(\Xi, \mathcal F, P)$

Let $G=(G_n\in  L^1(\Xi,\mathcal F, P), n\in D)$ be a random sequence measurable with respect to  $(\xi_n:n\in D)$.

We say that $G$ is {\bf Skorokhod integrable} if there exists a random scalar
$$
S=\sint G \;\xi   \qquad \mbox{ in } L^1
$$ called the {\bf Skorokhod integral} of $G$ such that for every $n \in D$ and every bounded differentiable  $F:\mathbb R^n \to \mathbb R$ with bounded gradient $\nabla F$ we have
\begin{equation}\label{defbygrad}
ESF(\xi_1,\ldots \xi_n)=E G\cdot \nabla F(\xi_1,\ldots, \xi_n).
\end{equation}
\end{definition}

\begin{remark} Since $S$ and $G$ are assumed to be in $L^1$, it is enough (by approximation) to take smooth functions with compact support in \eqref{defbygrad}. 
\end{remark}

\begin{definition}[Skorokhod integral of random measures or functions]\label{d:Skorokohod-functions} \

Let $(R,\mu)$ be a measure space.

Consider an orthonormal basis of $L^2(R,\mu)$ consisting of bounded functions $(\basis_n, n\in D)$ where $D=\{1,\ldots, \max(D)\}$ is finite or $D=\mathbb N$.

Let $\xi_n,n\in D$ be independent standard 
Gaussian random variables defined on some probability space $(\Xi, \mathcal F, P)$

Let $\hat G$ be a random variable on $(\Xi,\sigma(\xi_n:n\in D), P)$ taking values in some space of signed measures on $R$. 

We say that $\hat G$ is {\bf Skorokhod integrable} with {\bf Skorokhod integral} $\ssint \hat G \xi =S$ if the total absolute mass of $\hat G$ is finite a.s., and 
$G=(G_n=\int  \basis_n \hat G,n\in D)$ is Skorokhod integrable with integral $S$ in the sense of Definition \ref{d:skorokhod}. 

When  $\hat G$ is a random measurable function, we apply this definition to the random signed measure $\hat G\mu$. 
\end{definition}

This is somewhat closer to the standard usage.  What we are doing is thinking of the Skorokhod integral as a renormalized extension of an $L^2$ inner product and working directly in the coordinate representation.

\subsection{Projections}

For infinite dimensional examples, it helps to define  versions of the Skorokhod integral where only finitely many variables are integrated over. A natural projection property then determines the integral, see Proposition \ref{p:projection}.

\begin{definition}[Projections]\label{proj} Let $\mathcal P_nG$ denote the vector $G$ with all coordinates beyond $n$ set to zero and
\begin{equation}\label{peeenn}
\sintn G \,\xi: =\sint \mathcal (P_n G)\,\xi.
\end{equation}
\end{definition}
The following proposition can be thought of as a constructive definition of the Skorokhod integral in nice cases.

\begin{proposition}
\label{fdc}
\begin{enumerate}
\item If $G$ is Skorokhod integrable, then its integral is unique a.s.\

\item
Let $G_j$, $\xi_j$ be the $j$th coordinate of $G,\xi$, respectively. Assume that for all $j$, and almost all $\xi_1,\ldots,\xi_{j-1},\xi_{j+1},\ldots$, the function
$G_j$ is absolutely continuous
in $\xi_j$. If  for each $i\ge 1$ we have, $E|G_i|<\infty$ and with
\begin{equation}\label{e:simple-skorokhod}
S_n=\sum_{j=1}^n G_j\xi_j-\partial_{\xi_j}G_j, \qquad E|S_n|<\infty
\end{equation}
then
$$
\sintn G \,\xi=S_n.
$$
If $S_n$ converges in $L^1$ to a limit $S$ as $n\to\infty$, then $G$ is Skorokhod integrable and
$$
\sint G \,\xi = S.$$
\end{enumerate}
\end{proposition}

\begin{proof} (1)
If $S,S'$ are two Skorokhod integrals, then by the definition $E[(S-S')F]=0$ for all bounded differentiable $F$ with bounded derivatives. Such functions $F:{\mathbb R}^n\to \mathbb R$ separate $L^1$ and hence $E[ S-S'~\mid~\mathcal F_n]=0$ for each $n\in \mathbb N$, and so  $S-S'=0$ a.s.

(2) let $F(\xi_1,\ldots,\xi_n)$ be smooth, compact support. Condition on the $\sigma$-field $\mathcal J_{\hat j}$ generated by $\xi_1,\ldots,\xi_{j-1},\xi_{j+1},\ldots$.  $FG_j$ is absolutely continuous with compact support so  integrating by parts with respect to the Gaussian measure, 
$$
E[FG_j\xi_j|\mathcal J_{\hat j}]=E[\partial_{\xi_j}(FG_j)|\mathcal J_{\hat j}].
$$
Taking expectations, 
$$
E[FG_j\xi_j-F\partial_{\xi_j}G_j]=E[\partial_{\xi_j}(FG_j)-F\partial_{\xi_j}G_j]=E[(\partial_{\xi_j}F)G_j].
$$
Note that $\partial_{\xi_j} F=0$ for $j>n$. Summing over $j\le m$ this gives that for $m\ge n$,
$$
E[S_mF]=E[G\cdot\nabla F].
$$
Since $S$ converges in $L^1$ and $F$ is bounded, 
$$
E[SF]=\lim_{m\to\infty}E[S_mF]=E[G\cdot \nabla F]
$$
as required.
\end{proof}

The Skorokhod integral is linear, and has $ES=0$, as we can see by taking $F=1$ in the definition. It also behaves nicely under conditional expectations.

\begin{proposition}\label{p:projection}
  If $G$ is Skorokhod integrable, then  $E[\mathcal P_n G|\mathcal F_n]$ is Skorokhod integrable, and
$$
E\left[\sint G\,\xi \,\Big|\, \mathcal F_n \right]=\sint E[\mathcal P_n \,G \,|\, \mathcal F_n ]\, \xi = \sintn E[G|\mathcal F_n] \,\xi
$$
\end{proposition}
For $n=0$ we get $E\ssint G\,\xi =0$.
\begin{proof}
We need to check that for all test functions $F$ as in the definition, we have
$$
E\Big[F\,E[S|\mathcal F_n ]\Big]=E\Big[\nabla F \cdot E[\mathcal P_n \,G \,|\, \mathcal F_n ]\Big]
$$
By properties of the conditional expectation, the definition of $S$, we have
$$
E\Big[F\,E[S|\mathcal F_n ]\Big]=E\Big[E[F|\mathcal F_n ]\,S\Big]
=
E\Big[\nabla E[F|\mathcal F_n ] \cdot G\Big].
$$
Now
$$
\partial_j E[F|\mathcal F_n ]=\begin{cases}E[\partial_jF|\mathcal F_n ], \qquad &j\le n \\ 0 & j>n
\end{cases}
$$
so we have
\begin{align*}
E\Big[\nabla E[F|\mathcal F_n ] \cdot G\Big]&=\sum_{j=1}^n E\Big[ E[\partial_j F|\mathcal F_n ] G_j\Big] =\sum_{j=1}^n E\Big[ \partial_j F  E[G_j|\mathcal F_n ]\Big]\\&=E\Big[\nabla F \cdot E[\mathcal P_n \,G \,|\, \mathcal F_n ]\Big]
\end{align*}
as required, showing the first equality. The last inequality follows by definition of the finite-dimensional version of the Skorokhod integral.
\end{proof}
\begin{corollary}\label{c:Sk-martingale}  Let $G$ be a $\mathbb R^{\mathbb N}$-valued random variable such that $E|G_i|<\infty$ for each $i\in\mathbb N$.
Let $H_n=E[\mathcal P_nG|\mathcal F_n]$.

The following are equivalent.
\begin{enumerate}
    \item For every $n$, $S_n=\ssint H_n \,\xi$ exists
        and the $S_n$ are uniformly integrable.
    \item $G$ is Skorokhod integrable.
\end{enumerate}
If these conditions hold, then $S_n$ is a martingale with limit $\ssint G\,\xi$ almost surely and in $L^1$.
\end{corollary}
\begin{proof}
If (2) holds, by Proposition \ref{p:projection}, $H_n$ is Skorokhod integrable and

$$S_n=E[\sint G\xi|\mathcal F_n].
$$Since these are conditional expectations of a fixed integrable random variable, we see that $(S_n)$  are uniformly integrable, so (1) and the final conclusion hold.

If (1) holds, then by Proposition \ref{p:projection} applied to $G'=H_{n+1}$, we have  $$E[S_{n+1}|\mathcal F_n]=E\left[\sint H_{n+1}\,\xi \,\Big|\, \mathcal F_n \right]=\sint H_n\, \xi = S_n.$$
So $S_n$ is a uniformly integrable martingale and converges to some limit $S$ almost surely and in $L_1$.  Then, for $F$ bounded differentiable with bounded gradient (as in Definition \ref{d:skorokhod}), we have
$$
E[FS]=\lim E[FS_n] = \lim E[\nabla F\cdot H_n] = E[\nabla F \cdot G].
$$
In the last two equalities, we used that by definition, the coordinate $(H_n)_j$ for $n\ge j$ is a uniformly integrable martingale with limit $G_j$. Thus $G$ is Skorokhod integrable with integral $S$.
\end{proof}

We call a subspace $S\subset \seq$ finitary if there exists an $n$ such that for all vectors in $S$, their coordinates are zero beyond  $n$. The projection to a finitary subspace is well-defined: simply restrict a vector to the first $n$ coordinates and then apply Euclidean projection there. Orthogonal invariance implies the following corollary to Proposition \ref{p:projection}.

\begin{corollary}
Let $\mathcal P$ be projection to a finitary subspace of $\seq$. Let $\mathcal S$ be $\sigma$-field generated by the random variable $\mathcal P\xi$.
If $G$ is Skorokhod integrable, then  $E[\mathcal P G|\mathcal S]$
is Skorokhod integrable, and
$$
E\left[\sint G\,\xi \,\Big|\, \mathcal S\right]=\sint E[\mathcal P\,G \,|\, \mathcal S]\, \xi.
$$
\end{corollary}
\subsection{Examples of Skorokhod integrals}
\ \\
We first consider some one-dimensional problems ($D=\{1\}$ in Definition (\ref{d:skorokhod})), as these are as simple as possible.

\begin{example}\label{ex:recursion}
Consider the case when $\Xi=\mathbb R$ houses a  one-dimensional Gaussian $\xi$. Let's solve the recursion
$$
X_0=1, \qquad X_{k}=\sint X_{k-1} \,\xi, \qquad k\ge 1.
$$
Then the solution is $X_n=H_n(\xi)$, where $H_n$ are the Hermite polynomials. This can be seen by the recursion $H_{k+1}(x)=xH_k(x)-H_k'(x)$ and the formula \eqref{e:simple-skorokhod}.  See Example \ref{ex15} for further connection between the Skorokhod integral and Hermite polynomials.
\end{example}

The following example is good to keep in mind when trying to construct discrete-time Wick-ordered  polymers.

\begin{example}
Consider a one-dimensional Gaussian $\xi$, and the slightly more complex linear recursion
$$
X_0=1, \qquad X_{k}=X_{k-1}+\beta \sint X_{k-1} \,\xi, \qquad k\ge 1.
$$
Write the solution in the form $X_k=\beta^k p_k(\xi)$. Then $p_k$ satisfies
$$
p_{k}=\beta^{-1} p_{k-1}+p_{k-1}\xi - p'_{k-1}, \qquad p_0=1
$$
This is just the Hermite recursion with $x=\beta^{-1} +\xi$, so
the solution is $X_k=\beta^kH_k(\beta^{-1}+\xi)$, where $H_k$ are the Hermite polynomials.
\end{example}

Can basic functions be non-integrable?

\begin{example}
Still in one-dimension, let $G=\one(\xi\ge 0)$. The Skorokhod integral tries to be $\xi^+-\delta_0(\xi)$, but of course this is not a random variable, as $\delta_0$ is not an honest function. It is a simple exercise to show that $\ssint G \, \xi$ does not exist. Perhaps it could be defined by extending the probability space, but we do not pursue this direction.
\end{example}

Next, we study the continuous version of Example \ref{ex:recursion}.

\begin{example}[A simple Skorokhod stochastic differential equation] We still work on the one-dimensional Gaussian space, and consider the ``stochastic'' differential equation $du = \ssint u\xi$ written in an integral form
\begin{equation}\label{e:simple-Sk}
u(t)=\int_0^t\left(\sint u(s)\,\xi\right) ds, \qquad u(0)=1.
\end{equation}
This is not a usual SDE, as the entire randomness is based on  a single Gaussian variable $\xi$!
To solve this, we introduce the test function $F=e^{i\lambda\xi}$ satisfying  $\nabla F=i\lambda F$, and let $\hat u(t)=E[u(t)F]$. We multiply \eqref{e:simple-Sk} by $F$, take expectations and use the definition of the Skorokhod integral to get
$$
\hat u(t) = i\lambda \int_0^t \hat u(s)ds, \qquad \hat u(0)=e^{-\lambda^2/2}.
$$
this is just the mild version of the equation $\hat u'=i\lambda \hat u$ solved uniquely by
$\hat u=e^{i\lambda t-\lambda^2/2}$.
This holds for all $\lambda$, so inverting, we get $u(t)=e^{t\xi-t^2/2}$.

Remarkably, this method will also solve the Wick-ordered  planar stochastic heat equation \eqref{2dshe}.
\end{example}

The planar Wick-ordered  heat equation will be a continuum version of the following simple example.

\begin{example}\label{ex:MC}(Continuous time, finite state-space Wick-ordered  polymer)
Next, consider the evolution of a finite state-space $S$ and continuous time Markov chain  moving on these $d=|S|$ states (for example a continuous time symmetric random walk on the $d$-cycle). The generator $K$ is a $d\times d$ matrix, and the transition probability vector $p$ from a given site $x$  satisfies the forward equation
$$
\partial_t p(y,t) = (p(\cdot,t)K)(y,t) \qquad p(\cdot,0)=\one_x.
$$
We add a fixed  potential $V:S\to \mathbb R$, and then the partition function for a  non-random  ``polymer''  (i.e. in a non-random potential) should look like
$$
\partial_t u(y,t) = (u(\cdot,t) K)(y) + u(y,t)V(y) \qquad u(\cdot,0)=\one_x.
$$
Now, if the potential is now random and Gaussian, the Skorokhod version of this equation is written as
$$
\partial_t u(y,t) = (u(\cdot,t) K)(y) + \sint u(y,t)\one_y \mathbb \,\xi \qquad u(\cdot,0)=\one_x.
$$
Here $\xi$ is a standard $d$-dimensional Gaussian vector defined on $\Xi=\mathbb R^d$.
The vector $\one_y$ is there so that the noise corresponds to a diagonal potential.

Multiplying by  $F=e^{i\sum_{y\in S}\xi(y)\lambda(y)}$.
setting $\hat u =E[uF]$ and taking expectations we
get
$$
\partial_t \hat u(x,t) = \hat u(\cdot,t) K (y) + \hat u(y,t)i\lambda(y) , \qquad u(\cdot,0)=\one_x e^{-\sum_i \lambda_i^2/2}.
$$
At this point, we leave it to the reader to use the discrete version of the Feynman-Kac formula and then Fourier inversion to conclude that
$$
u(y,t)=E_Q \bigg[\exp\Big\{\sum_z X(z)\xi_z-X(z)^2/2\Big\}\one(X_t=y)\bigg].
$$
The expectation is under a measure $Q$ independent of the Gaussian space $\Xi$. Under the measure $Q$, the variable $X_t$ is the position of the original Markov chain at time $t$ when started from $x$, and $X(z)=\int_0^t \one (X_s=z)ds$ is the time  spent at site $z$.

The reader will recognise the formula above, without the quadratic terms in the exponential, as the definition of the random polymer with Gaussian weights. The quadratic correction makes $Eu(y,t)=p(y,t)$.
\end{example}

Next, we provide another example of a Skorokhod integral problem that will be a key ingredient to the proof of Proposition \ref{p:fshe-solve}. It can be thought of as a pathwise version of the Duhamel formulation \eqref{e:finite-she} of the projected planar Wick-ordered SHE.

\begin{example}\label{ex:sk-transport} 
Let $\basis_1, \ldots, \basis_n$ be bounded  functions that are orthonormal in   $L^2(\mathbb R^2)$.  
Let $\omega:(0,T]\to \mathbb R^2$ be a fixed measurable function.  We derive a solution to the  differential equation 
\begin{equation}\label{e:sk-transport}
\partial_t u(t)=\sintn u(t)\delta_{\omega(t)}\, \xi, \quad t\in (0,T], \qquad u(0)=1,
\end{equation}
over $u\in \sigma(\xi_1,\ldots,\xi_n)$, with $u=u(t)=u(t,\xi_1,\ldots,\xi_n)>0$ differentiable in $t,\xi_1,\ldots \xi_n$.  Here the Skorokhod integral is of a randomly weighted delta mass at a deterministic location, see Definition \ref{d:Skorokohod-functions}. Integration by parts in  the test function definition of the Skorokhod integral 
shows that 
$
\ssintn u(t)\delta_{\omega(t)}\, \xi =\sum_{i=1}^n (f_i) (\xi_i-\partial_{\xi_i})u
$
and so assuming the above regularity of $u$, \eqref{e:sk-transport} is equivalent to 
$$
\partial_t u=\sum_{i=1}^n (f_i) (\xi_i-\partial_{\xi_i})u,  \qquad u(0,\cdot)=1, \qquad f_i(t)=\basis_i(\omega(t)).
$$
We divide by $u$ and rearrange to get
$$
(\partial_t +\sum_{i=1}^n f_i \partial_{\xi_i})\log u=\sum_{i=1}^n f_i\xi_i,
$$
an inhomogeneous transport equation. Using the method of  characteristics,  set $\xi(t)=\xi_0+\int_0^t f(s)\,ds$, so that $\partial_t \log u(t,\xi(t))=\sum_{i=1}^n f_i(t)\xi_i(t)$. So at least along the curves $\xi(t)$ we have 
\begin{align}\label{e:transport-solution}
u(t,\xi)&=\exp \sum_{i=1}^n \int_0^t f_i(t)\int_0^sf_i(s)\,dr\,ds
\\ \notag
&=\exp \sum_{i=1}^n\xi_i m_i(t)-\tfrac12m_i(t)^2, \qquad m_i(t)=\int_0^t \basis_i(\omega(s))ds.
\end{align}
Differentiating directly shows that \eqref{e:transport-solution} indeed solves the problem \eqref{e:sk-transport}. Integrating \eqref{e:sk-transport} in $t$ we also have 
\begin{equation}\label{e:sk-transport-int}
\partial_t u(t)=1+\int_0^t\sintn u(s)\delta_{\omega(s)}\,\xi\,ds, \quad t\in (0,T].
\end{equation}
\end{example}

\subsection{Classical Skorokhod integral}
We will show that Definition \ref{d:skorokhod} of the Skorokhod integral given in this article  is an $L^1$ extension of the standard or classical definition. The precise statement will be given in Lemma \ref{l1extension} below. Before we do this, we provide some motivation for the Skorokhod integral (our definition, as well as the classical one),  looking at the finite dimensional case. A natural requirement is that for deterministic functions (in this case, vectors $G$ in $\mathbb R^n$) the Skorokhod integral should be the inner product, \begin{equation}\label{ppp} \sint G \,\xi = G_1\xi_1+\ldots +G_n\xi_n.
\end{equation}
Now let's say we are in $n=1$, and ask what the integral of a polynomial $G$ in $\xi$ should be. One option might be $G\xi $, but it does not have mean zero, for example when $G=\xi$. If we want it to have zero expectation,
we need to introduce a correction, $\ssint G \xi =G \xi  ~+$ correction. It is reasonable  that the correction would be a polynomial in $\xi$ of lower degree. Moreover, to be reminiscent of the It\^o integral, we may want iterated integrals of 1 to be orthogonal. For polynomials $G$, these two requirements define a unique integral
\begin{equation}
\label{recur}
\sint G \xi = G\xi  -\partial_\xi G, \qquad n=1.
\end{equation}
 Indeed, the iterated integrals of $1$ with the requirement of orthogonality, turn out to be given by the Hermite polynomials, which satisfy the recurrence relation (\ref{recur}).
By the completeness of these polynomials
there is a unique continuous extension for general integrands.

To treat the general finite dimensional case, it helps to first recall a very general construction.

Let  $H=L^2(R)$.  The $n$-dimensional case corresponds to $R=\{1,\ldots ,n\}$ with counting measure.
       For each $k\ge 1$, define recursively
    $\mathcal{H}_{k}$ as  the set
    of symmetric polynomials in $\xi_1,\xi_2,\ldots$ of degree $k$, orthogonal to $\mathcal{H}_{k-1}$, under the inner
    product
    $E[FG]$.  Then (see, for example, \cite{Janson})
    \begin{equation}\label{iso}
        \oplus_{k=0}^\infty \mathcal{H}_{k}=L^2(\Xi,\mathcal F, P).
    \end{equation}
     Let  
     \begin{equation}
     \label{pikk}
     \pi_k:\oplus_{k'=0}^\infty \mathcal{H}_{k'}\to\mathcal{H}_k,
     \end{equation}
      denote the orthogonal projection to $\mathcal{H}_{k}$.

$\mathcal{H}_{0}$ consists of non-random functions on $R$, and can be identified with $H$.  For $k\ge 1$, one needs some symmetrization (after all, the polynomials $\xi_1\xi_2$ and $\xi_2\xi_1$ mean the same thing.)  Let $H^{\otimes_s k}=L^2\left(R^k, \mathcal R^k_s, \frac{1}{k!}\mu^k\right)$, where $\mathcal R^k_s$
is the $\sigma$-algebra of symmetric $\mathcal R^k$-measurable functions
on $R^k$.  We now explain how to represent $\mathcal{H}_{k}$ by
$H^{\otimes_s k}$. The {\bf symmetric Fock space} $\oplus_{k=0}^\infty H^{\otimes_s k}$ is thus identified with $\oplus_{k=0}^\infty \mathcal{H}_{k}$, and therefore by \eqref{iso}, with the Hilbert space  $L^2(\Xi,\mathcal F, P)$
of all  functions of $\xi_1,\xi_2,\ldots$ with finite second moment.
Given $g\in  H^{\otimes k}$, call
$$
{\rm Sym}(g)=\sum_{\sigma\in S_k}g\circ\sigma
$$
where  $S_k$ is the symmetric group.
In particular,
for $g_1,\ldots,g_k\in H$, $g_1\otimes\cdots\otimes g_k$ is the element of $L^2(R^k)$ which takes $(x_1,\ldots,x_k)$ to $\prod_{i=1}^k g_{i}(x_i)$ and  $
{\rm Sym}(g_1\otimes\cdots\otimes g_k)$ is the symmetric function that takes $(x_1,\ldots,x_k)$ to $
\sum_{\sigma\in S_k}\prod_{i=1}^k g_i(x_{\sigma_i})$.  Linear combinations of such are dense in $H^{\otimes_s k}$.

For $g\in H$ define $I_1(g)=\langle g,\xi\rangle$, as in the definition of the white noise  in (\ref{linis}).
For $k\ge 2$, and $g_1,\ldots,g_k\in H$ let
\begin{equation}
     \label{isometry-multi}
  I_k({\rm Sym}(g_1\otimes\cdots\otimes g_k))=\pi_k(I_1(g_1)\cdots I_1(g_k)),
\end{equation}
 where $\pi_k$, defined in (\ref{pikk}), is the orthogonal projection to $\mathcal H_k$.
It can then be extended to
$H^{\otimes_s k}$ by linearity and density.
For $g\in H^{\otimes_s k}$, $
\int g\xi=I_k(g)
$ is its {\bf multiple stochastic integral}.

\begin{definition}
Let $G: R\to L^2(\Xi,\mathcal F, P)$.   It follows  from \eqref{iso} that $G(x)= \sum_{k=1}^\infty \pi_k(G(x))$ in $L^2$.
 For each $x$, $\pi_k(G(x))=I_k(g_x)$, for some symmetric function $g_x\in H^{\otimes_s k}$. Now, since for each $x$,  $g_x(x_1,\ldots, x_k)$ is a function of $k$ variables, we can think of $\bar g (x_1,\ldots,x_k,x):=g_x(x_1,\ldots,x_k)$ as function of $k+1$ variables and set
\begin{equation}\label{iterativeclassic}
\int \pi_k(G)\xi=I_{k+1}({\rm Sym}(\bar g))).
\end{equation}
If
$
\int G\xi:=\sum_{k=0}^\infty \int \pi_k(G)\xi
$
 converges in $L^2(\Xi,\mathcal{F}, P)$, we call it {\bf the classical Skorokhod integral} of $G$ and call $G$ classically Skorokhod integrable. Note that we use the notation $\int G\xi$ to distinguish it from our previous definition of the Skorokhod integral $\ssint G\xi$.
\end{definition}

\begin{example}\label{ex15}\label{ex:multi-hermite} 
\ \\ Let $R=\{1,\ldots ,n\}$ with counting measure. Then $H=L^2(M,\mathcal M,\mu)$ is   $\mathbb R^n$ with the standard
inner product.
 $H^{\otimes k}$ is 
the set of real functions defined on $\{1,\ldots,n\}^k$, or, equivalently the positions in  $\{1,\ldots,n\}$ of $k$ \emph{distinguishable} particles.
$H^{\otimes_s k}$ then
corresponds to 
the set of real functions defined on the occupation measures of $k$  \emph{indistinguishable}
 particles on $\{1,\ldots,n\}$ represented by $\eta=(\eta_1,\ldots,\eta_n)\in\mathbb N^n$ with
$|\eta|:=\sum_i \eta_i = k$,   and inner
product  $
\sum_{|\eta|=k}g(\eta)f(\eta).
$
for  $g,f\in H^{\otimes_s k}$.
Let $H_k$, $k=0,1,\ldots$ be the Hermite polynomials. They can be defined through their generating function
\begin{equation}\label{hermitegenfn}
    e^{\lambda x - \tfrac12 \lambda^2}= \sum_{k=0}^\infty
    H_k(x) \frac{\lambda^k}{k!}.
\end{equation} One checks directly that
  $I_k: H^{\otimes_s k}\to\mathcal{H}_k$
  is given in this representation by
\begin{equation}
I_k(g)= \sum_{|\eta| =k}g(\eta) H_{\eta_1}(\xi_1)\cdots H_{\eta_n}(\xi_n),
\end{equation}
so the classical Skorokhod integral for $(g_x)_{x=1,\ldots,n}$ with $g_x\in H^{\otimes_s k}$, can just be written, by linearity, as a sum over $|\eta|=k$ of
\begin{align}
    \int &g_x(\eta) H_{\eta_1}(\xi_1)\cdots H_{\eta_n}(\xi_n)\xi(x)\\&\notag=\sum_{{ x\in\{1,\ldots,n\}}}{\rm Sym}(\bar g)(\eta+{1_x}) { H_{\eta_x+1}(\xi_x)\prod_{y\in \{1,\ldots,n\}-\{x\}}} H_{\eta_y}(\xi_y),
    \end{align}
 where  we define $\bar g(\eta+1_x):=g_x(\eta)$, and  $(\eta+{1}_x)_i=\eta_i+\delta_{i,x}$. 
  Using the Hermite polynomial identity
$$
H_{k+1}(x)=xH_k(x)-\partial_x H_k(x),
$$
we obtain that for every set of random variables
$ G(x)$, with $x\in R$,
$$
\int  G\xi=\sum_{i=1}^n  G_{i}\xi_i-\partial_{\xi_i} G_{i}
 =\sint  G\xi
$$
where in the last equality we have used \eqref{e:simple-skorokhod}.
\end{example}

We can now deduce the following lemma.

\begin{lemma}
\label{finited}Let $R=\{1,\ldots,n\}$ and let $G:R\to \oplus_{k=0}^\infty\mathcal H_k$, where $\mathcal H_{k}$   is the set of symmetric
polynomials in $(\xi_i)_{i\in R}$ of degree $k$. \\
 Assume that
$\sum_{k=0}^\infty\sum_{x\in R} |E[\pi_k(G(x)]|<\infty$. Then,  $G$ if has a classical
Skorokhod integral,  it is Skorokhod integrable. Furthermore, in this case, these integrals are equal.
\end{lemma}
\begin{proof} From the previous example, we know that the classical Skorokhod integral and the Skorokhod integral coincide for each $k$ in the space $H^{\otimes k}$. It follows that for each $K$, if $G_K=\sum_{k=0}^K\pi_k(G)$, then 
$$
S_K:=\int G_K\xi=\sint G_K\xi.
$$
Therefore, for all $F:\mathbb R^n\to\mathbb R$ we have that
$
E[S_KF]=E[G_K\cdot\nabla F]$.
Taking the limit $K\to\infty$ in the above inequality, we
conclude that $S=\sum_{k=0}^\infty\int\pi_k(G)\xi$ satisfies
$$
E[SF]=E[G\cdot\nabla F],
$$
so that $G$ is Skorokhod integrable and its Skorokhod integral is
equal to its classical Skorokhod integral.
\end{proof}

\begin{lemma} 
\label{l1extension}Let $G: R\to L^2(\Xi,\mathcal F, P)$ be a random function such that
$$\int_{\mathbb R} E[G(x)^2]d\mu<\infty.$$  Then, if $G$
  has a classical
  Skorokhod integral, it is
  Skorokhod integrable.  Furthermore, in this case, these integrals are equal.
\end{lemma}
\begin{proof}
Lemma \ref{finited} says that the statement  holds in the finite dimensional case.
For the general case,  we use the fact that integrands $G$ with
\begin{equation}\label{jbd}\sum_{k=0}^\infty (k+1) \int_{R} E[\pi_k(G)^2]d\mu<\infty,
\end{equation}
form  the domain of the classical Skorokhod integral (see, for example, Theorem 7.39 of \cite{Janson}).  Take such a $G$ and let $G_n$ be the truncation of its coordinate representation.

Let $S_n$ be its classical Skorokhod integral. We have shown in Example \ref{ex15}  that this coincides with
our Skorokhod
integral, so that
for all $F:\mathbb R^n\to \mathbb{R}$ which are bounded with bounded
gradient,
\begin{equation}\label{seqx}
E[S_n F]=E[G_n\cdot\nabla F].
\end{equation}
If $S$ is the classical Skorokhod integral of $G$ then
\begin{equation}
E\left|S_n-S\right|^2=
\sum_{k=0}^\infty E\left[\left|\int \pi_k(G_n)-\pi_k(G))\xi\right|^2\right].
\end{equation} From the definition \eqref{iterativeclassic} of the classical Skorohod integral as a simple raising operator on the $\mathcal{H}_k$-s, we have for any $G$, $$
\sum_{k=0}^\infty E\left[\left|\int \pi_k(G)\xi\right|^2\right]= \sum_{k=0}^\infty (k+1)\int_{ R} E[\pi_k(G)^2]d\mu.$$ Since $G_n\to G$ in this space
from \eqref{jbd},
$S_n$
converges to $S$ in $L^2(\Xi,\mathcal F,P)$.
Now we can take limits on both sides of \eqref{seqx} to conclude that $E[S F]=E[G\cdot\nabla F]$.
  \end{proof}

Once one has identified the Skorokhod integral on the homogeneous chaoses, it is not a big step to identify the Skorokhod integral with the Ito integral when the integrands are appropriately non-anticipating.   The proof can be found in  \cite{BarlowNualart}, p. 138.

\begin{remark}[It\^o integral in 1 dimension] \label{r:Ito1}Let $\xi$ be a $d$-dimensional white noise on $\mathbb{R}$ and $B(t)=<\xi,\mathbf{1}_{[0,t]}>$, $d$-dimensional Brownian motion defined on $(\Xi, \mathcal{F}, \mathbf{P})$.  Let $\mathcal{F}_t=\sigma( B(s), s\in [0,t])$. 

The space of $d$-dimensional, square integrable adapted processes $G(t)$
    is the closure in $L^2(\Xi, \mathcal{F}, \mathbf{P};\mathbb{R}^d )$ of those which can be written as   $G(t) = \sum_{i=0}^{n-1} g_i \mathbf{1}_{(t_i,t_{i+1}]} (t)$
    where $g_i$ are bounded and measurable with respect to   $\mathcal{F}_{[0,t_i]}$.
    
      For any such  $G(t)$,  $\ssint G \xi$ coincides with the It\^o integral $\int G(t) dB(t)$.
\end{remark}

\begin{remark}[It\^o integral in 1+1 dimensions]\label{r:Ito1+1} Let $\xi$ be white noise on $\mathbb{R}_+\times \mathbb{R} $ and
    $B(x,t)= <\xi,\mathbf{1}_{[0,t]\times [0,x]}>$, $x\ge 0$ and $B(x,t)= -<\xi,\mathbf{1}_{[0,t]\times [x,0]}>$, $x< 0$ be the Brownian sheet.

     Let $\mathcal{F}_{[0,t]}=\sigma(B(s,x), s\in[0,t], x\in \mathbb{R})$.   The space of square integrable, time-adapted processes $G(x,t)$ is the closure in $L^2(\Xi, \mathcal{F}, \mathbf{P})$ of those which can be written as
    $$G(x,t) = \sum_{i,j=0}^{n-1} g_{i,j} \mathbf{1}_{(t_i,t_{i+1}], (x_i,x_{i+1}]} (x,t),$$ where  $g_{i,j}$ are bounded and measurable with respect to $ \mathcal{F}_{[0,t_i]}$.

For such $G(x,t)$, $\ssint G \xi$ coincides with the It\^o integral
    $\int G(x,t) B(dt, dx)$.
\end{remark}
Crucially for us, the second part of the theorem tells us that the standard $1+1$ dimensional multiplicative stochastic heat equation coincides with its Skorokhod version.  This can also be seen through the chaos representation.

\begin{example}\label{example25}
We try to solve the integral equation with initial data $\mu= \delta_0$ in the Skorokhod sense 
\begin{equation}
    Z(x,t) = p(x,t) + \int \int_0^t p(x-y,t-s) Z(y,s) ds \xi(y) dy.
\end{equation}
We look for a solution $Z(x,t)\in L^2(P)$. As above
we have $Z(x,t) = \sum_{k=0}^\infty I_k(g_k)$ for some $g_k(\cdot;t,x)$ in  symmetric $L^2((\mathbb R^2)^k)$.  Using \eqref{iterativeclassic} this means that $I_0(g_0) = p(x,t)$ and
\begin{equation}\label{38}
    I_{k+1}(g_{k+1}(x,t)) = \int \int_0^t p(x-y,t-s) I_k( g_k(y,s)) ds \xi(y) dy
\end{equation}
Here one should think of  $g_k\in H^{\otimes_s k}$ as depending on the extra parameters $t,x$.  It is not hard to see  that the solution is
\begin{equation}\label{39}
g_k(x_1,\ldots,x_k;t,x)= E_Q[m(x_1)\cdots m(x_k)]
\end{equation}
where $m$ is the occupation measure of the Brownian bridge from $0$ at time $0$ to $x$ at time $t$. Of course, this can also be written explicitly as
\begin{align}\label{40}
\int_{0<s_1<\cdots<s_k<t} \prod_{j=0}^k p(s_{j+1}-s_j, x_{j+1}-x_j)
    \end{align}
where $p(x,t)$ is the planar heat kernel \eqref{planarheatkernel} and $y_{k+1}=x$, $y_0=0$.

By comparison, in PAM, the $p$'s in \eqref{40} are weighted by the exponential of the renormalized self-intersection local time (see \cite{GuHuang}).

In $1+1$ dimensions, the analogue of \eqref{39}  uses the occupation measure of $(b(s),s)$ where $b$ is the one-dimensional bridge instead of the occupation measure of a planar Brownian bridge.  The analogue of \eqref{40} can be found in \cite{ACQ}.
One checks directly that the $L^2$ norm of $Z$, which is the sum over $k$ of the $L^2$ norms of the $g$'s, is finite in $1+1$ dimensions, for all $t$, and only finite up to $t_c<\infty$ in the planar case.
\end{example}

\section{Solving  Wick-ordered   heat equations}
\label{heateq}

\subsection{Projections of the equation and their solution}

Perhaps the most useful property on the Skorokhod integral is that the conditional expectation of solutions of \eqref{2dshe} with respect to $\mathcal F_n$ (defined in \eqref{effen}) are themselves solutions of a stochastic heat equation, but now with respect to finite dimensional noise. Moreover, we can solve such equations in an explicit way.

\begin{proposition}\label{p:spde-projection}

Let $u$ be a solution of the Wick-ordered heat equation, Definition \ref{d:Wick-ordered -she-intro}. Then  $u_n(x,t):=E[u(x,t)|\mathcal F_n]$ solves the following problem. We have $u_n\in L^1(\mathbb R^2\times (0,T]\times \Xi)$, and  for almost all $x,t$ the random variable $\mathcal K_{x,t}(u_n)$ is $\mathcal F_n$-measurable, Skorokhod integrable, and
\begin{equation}\label{e:finite-she}
u_n(x,t) = \int p(x-y,t)d\varsigma(y) +   \sintn \mathcal  K_{x,t} (u_n)~ \xi.
\end{equation}
\end{proposition}
\begin{proof}
Since conditional expectation is an $L^1$ contraction for every $x,t$, we see that $u_n\in L^1(\mathbb R^2\times[0,t]\times \Xi)$. 

Next, we take conditional expectation of \eqref{inteq-intro}. Only the last term on the right needs explanation.

By Proposition \ref{p:projection},
$$
E[\sint \mathcal K_{x,t}(u) \,\xi |\mathcal F_n]=\sintn E[\mathcal K_{x,t}(u)|\mathcal F_n]\, \xi.
$$
in particular,  $E[\mathcal K_{x,t}(u)|\mathcal F_n]$ is Skorokohod integrable for almost all $x,t$.

Since the Skorokhod integral is defined in terms of the coordinates $\basis_i$, it suffices to check \begin{align}\label{e:Knorm}
    \int \mathcal K_{x,t}(|v|)(y)dy&=\|v\|_{L^1(\mathbb R^2,[0,t])},\qquad \text{for all measurable } v,\\ \label{e:cord-equal}
\int \basis_i(y)E[\mathcal K_{x,t}(u)(y)|\mathcal F_n]\,dy&=\int \basis_i(y)\mathcal K_{x,t}(u_n)(y)\,dy, \qquad \text{ for all }i \text{ a.s.}
\end{align}
Direct computation gives  \eqref{e:Knorm}. For \eqref{e:cord-equal} we multiply both sides by an $\mathcal F_n$-measurable test function $F$ and take expectations. The claim follows by Fubini, which is justified since \eqref{e:Knorm} for $u$ and $u_n$ has finite expectation. 
\end{proof}

\begin{proposition}[Existence, uniqueness and explicit solution for finite $n$] \label{p3}\label{p:fshe-solve}
There is a unique  solution  to  \eqref{e:finite-she}. It has the following representation.
Let $Q_{x,t,y}$ be the law of a planar Brownian bridge $B$ from $y$ at time $0$ to $x$ at time $t$, and let 
\begin{align}\label{e:mj}
&m_j(t)= \int_0^t\basis_j(B(s))ds, \quad M(t)=  \exp\left(\sum_{j=1}^n m_j(t) \xi_j - \tfrac12 m_j^2(t) \right), \\
&Z_n(0,y;t,x)= EQ_{x,t,y} M(t),\notag
\end{align}
and
\begin{equation}\label{e:proposed-un}
 u_n(x,t)=\int
 Z_n(y,0;x,t) p(x-y,t)\, d\varsigma(y).
\end{equation}

\end{proposition}

\begin{proof} It is natural to expect that a version of equation \eqref{e:finite-she} holds pathwise for $B$.  More precisely, by Example \ref{ex:sk-transport}, for every $B:(0,T]\to \mathbb R^2$,  $M(t)=M(B,t)$ of \eqref{e:mj} solves
\begin{equation}\label{e:pathwise-D}
M(t)= 1+\int_0^t\sintn M(s)\delta_{B(s)}\,\xi\,ds, \quad t\in (0,T].
\end{equation}
Here the Skorokhod integral is of a randomly weighted $\delta$-mass at a fixed location, but this situation is fully covered by  Definition \ref{d:Skorokohod-functions}. Consider the finite measure $Q'_{x,t}=p(x-y,t)Q_{y,x,t}d\varsigma(y)$. Taking $Q'_{x,t}$-expectations of \eqref{e:pathwise-D} yields  
\begin{equation}\label{e:E-pathwise}
u_n(x,t)=:E_{Q'_{x,t}}M(t)=E_{Q'_{x,t}}1+S(x,t),
\end{equation}
and $u_n$ is our proposed solution  \eqref{e:proposed-un}
of the problem \eqref{e:finite-she}. 

To establish that $u_n$ is indeed a solution, it remains to show that $u_n\in L^1((0,T]\times \mathbb R^2\times \Xi)$, which follows from  $E_PM(s)=1$, 
and that $\mathcal K_{x,t}(u_n)$ is Skorokhod-integrable with integral $S(x,t)$. 

By Example \ref{e:sk-transport}, the Skorokhod integral in \eqref{e:pathwise-D} exists and
$$\sintn M(s)\delta_{B(s)}\,\xi=(\xi-m(s))\cdot \basis(B(s)) M(s).$$
This is bounded as a function of $B$ and $s$. By Fubini,
\begin{align*}
S(x,t)=E_{Q'_{x,t}}&\int_0^t (\xi-m(s))\cdot \basis(B(s)) M(B,s)\,ds \\&= 
\int_0^t E_{Q'_{x,t}}[(\xi-m(s))\cdot \basis(B(s)) M(B,s)]\,ds.
\end{align*}
By the Markov property of the Brownian bridge at time $s$ and Fubini again,
\begin{align*}S(x,t)=\int \int_0^t E_{Q'_{y,s}}[(\xi-m(s))\cdot \basis(y)M(B,s)]p(x-y,t-s)\,ds\,dy
\end{align*}

For Skorokhod integrability,  note first that by \eqref{e:Knorm}  $K_{x,t}(u_n)\in L^1(\mathbb R^2)$ for almost all $x,t,\xi$. 
Let $F:\mathbb R^n\to\mathbb R$ be a bounded function with bounded gradient. 
Since $\basis$ and $m(s)$ are bounded, the argument of $E_{Q'_{s,t}}$  is also absolutely integrable over $(\Xi,P)$, and the integral is bounded over $s,y,B$. So by Fubini, $E_P[FS(x,t)]$ is given by  
\begin{align*}
\int \int_0^t E_{Q'_{y,s}}E_P[(\xi-m(s))\cdot \basis(y)M(s)F]]p(x-y,t-s)\,ds\,dy.
\end{align*}
Next, we use that $\partial_{\xi_i}(M(s)e^{-|\xi|^2/2})=(m_i(s)-\xi_i)M(s)e^{-|\xi|^2/2}$ and integrate by parts in $\xi$ to get 
\begin{align*}
\int \int_0^t& E_{Q'_{y,s}}E_P[\nabla F\cdot \basis(y)M(s)]]p(x-y,t-s)\,ds\,dy
\\&=
\int E_P[\nabla F \cdot \basis(y)\int_0^t E_{Q'_{y,s}}M(s)]p(x-y,t-s)\,dy\,ds.
\end{align*}
By definition, this equals $
\int E_P[\nabla F \cdot \basis(y)\mathcal K_{y,t}(u_n)(y)\,dy\,ds].$
Thus $S(x,t)=\ssintn \mathcal K_{x,t}(u_n)\xi$. By \eqref{e:E-pathwise}  $u_n$ solves \eqref{e:finite-she}.

To show  uniqueness, let $u,v$ be two solutions of \eqref{e:finite-she}. Consider Fourier type test functions,  $F(\xi)=\exp(i\lambda \cdot \xi)$, since in this case  $\nabla F=i\lambda F$. Multiplying \eqref{e:finite-she} by $F$, taking expectations, and using the test function definition of the Skorokhod integral, we see that 
$w=E[(u-v)F]$ satisfies
\begin{align}\notag
w(x,t)&= 
  E \int_0^t \int  h(y) (u-v)(y,s) p(x-y,t-s)dyds\\ &=
  \int_0^t \int  h(y) w(y,s) p(x-y,t-s)dyds, \qquad h(y)=i\lambda \cdot \basis_i. \label{e:w-Fub}
\end{align}
Note that $\|w\|_{L^1((0,T]\times \mathbb R^2)}\le \|u-v\|_{L^1((0,T]\times \mathbb R^2\times \Xi)}$. 
The argument at the end of Remark \ref{r:weak-pde-version} below allows the change of order of integration in \eqref{e:w-Fub} for almost all $x$. 
Integrating \eqref{e:w-Fub}
    in $x$ using the fact that $ \int p(x-y,t-s) |w(y,s) |dy\le \| w(\cdot,s)\|_{L^1(\mathbb R^2)}$, since the heat kernel is an $L^1$-contraction, we find that 
$$A(t)= \int |w(x,t)| dx \text{\;\;\;satisfies\;\;\;} A(t) \le \|h\|_{\infty} \int_0^t A(s) ds.$$
Thus  $A(t)=0$ and $w(x,t)=0$ for almost all $x,t$.
Now $w(x,t)$ is the Fourier transform of $(u-v)(x,t,\xi) e^{-|\xi|^2/2}/(2\pi)^{n/2}$ in $\xi$ at $\lambda$. As $L^1(\mathbb R^n)$  functions are a.e.\ determined by their Fourier transform, it follows that $(u-v)(x,t,\xi)=0$ for almost all $x,t,\xi$, as required.
\end{proof}

\begin{remark}\label{r:weak-pde-version}
    A formulation equivalent to \eqref{e:finite-she} is that  $u_n=u_n(x,t,\xi)\in L^1(\mathbb R^2\times (0,T]\times \mathbb R^n)$ satisfies the following.  
For every bounded test function $F:\mathbb R^n\to \mathbb R$ with bounded gradient, for almost all $(x,t)\in \mathbb R^2\times (0,T]$ we have
\begin{align}\notag
E_\xi[F(\xi)u_n(x,t,\xi)]&=E[F]\int p(x-y,t)d\varsigma(y) \\
\label{e:finite-she2a}
&\;\;+ \int\int_0^t p(x-y,t-s)E_\xi[\nabla F(\xi) \cdot \basis (y)u(y,s)]\,dy\,ds.
\end{align}
where $E_\xi f(\xi)=(2 \pi)^{-n/2}\int e^{-|\xi|^2/2}f(\xi)d\xi_1\dots d\xi_n$ is simply the integral with respect to standard Gaussian measure on $\mathbb R^n$, and $\basis(y)=(\basis_1(y),\ldots,\basis_n(y))$ as in Definition \ref{d:Skorokohod-functions}.

To see the equivalence, 
Definition \ref{d:Skorokohod-functions} of the Skorokhod integral requires  \eqref{e:finite-she2a}
with the last term $a$ replaced by 
$$E\left[\int \nabla F(\xi) \cdot \basis(y)\int_0^t p(x-y,t-s)u_n(y,s)\right]\,ds\,dy,
$$
which is just a change of the order of integration. In addition, it requires  $\mathcal \int |\mathcal K_{x,t}(u_n)|(y) dy$ to be finite for almost all $x,t$, but this automatically holds by \eqref{e:Knorm}.

To  justify the change of order of integration, multiply both  by a bounded measurable   $\varphi(x):\mathbb R^2\to \mathbb R$ and integrate in $x$. We can conclude by Fubini if the quadruple integral is absolutely convergent. For this, we can fist drop the bounded terms $\varphi(x), \nabla F(\xi) \cdot \basis(y)$. But now  the integrand is positive and  the integral equals $\|u_n\|
_{L^1(\mathbb R^2\times (0,T]\times \Xi)}<\infty$.
\end{remark}

\begin{remark}The problem \eqref{e:finite-she2a} for $u_n$ is  on an $n$-dimensional Gaussian space and therefore much simpler than a true stochastic PDE.
It is  an integral   version of the  linear PDE for $u_n(x,t,\xi_1\ldots \xi_n)$ given by
\begin{equation}\label{e:finite-she2}
\partial_t u_n = \left(\tfrac12\Delta_{x} +  \sum_{j=1}^n \basis_j( \xi_j - \partial_{\xi_j}) \right)u_n,\qquad u_n(x,0,\xi) = \varsigma(x)
\end{equation}
 where $\basis_j$ defined in Theorem 2 are bounded functions forming an orthonormal basis of $L^2(\mathbb R^2)$.  The initial condition $\varsigma$ is only a function of $x$ but this is a PDE in all variables $t,x_1,x_2$ and $\xi_1,\ldots,\xi_n$.  As we have seen, the solution is even real analytic in $\xi$. Example \ref{ex:sk-transport} gives some intuition about how the transport part of this equation emerges.
\end{remark}

{\noindent \bf Remarks.}
\begin{enumerate}
\item $Z_n$ is the Radon-Nikodym derivative of the distribution of $(\xi_1+m_1,\ldots, \xi_n+m_n)$ under the product measure of the laws of the white noise $P $ and the Brownian bridge $Q$, with respect to that of the white noise
$P $.
    \item $Z_n$ is a non-negative martingale with respect to the filtration $\mathcal F_n=\sigma(\xi_1,\ldots,\xi_n)$ and hence converges almost surely to a random variable $Z$.

    \item If $E_{P }[Z]=1$ we have a candidate for the solution of the Wick-ordered  heat equation \eqref{2dshe}.
    \item If $E_{P }[Z]=1$, $Z$ has the interpretation as the Radon-Nikodym derivative of the distribution $\hat P$ of the shifted white noise $\xi + X$ with respect to the distribution $P$ of $\xi$, where $X$ is the occupation measure of the Brownian bridge up to time $t$. The expected free energy is equal to the relative entropy of $\hat P$ with respect to $P$, \begin{equation}\notag
        E_{\hat{P }}\log Z = H(\hat{P}\mid P).
    \end{equation}

    \item  For a.e. realization of the Brownian bridge, its occupation measure $X$ is not in the Cameron-Martin space; if it were, we should have 
    $$\int \left(\frac{ d X}{dx}\right)^2 dx= \sum_{j=1}^\infty m_i^2 = 2\int_{0\le s_1<s_2\le t} \delta_0( B(s_2)-B(s_1))ds_1ds_2,$$ 
    the $2$-dimensional self-intersection local time.  This sum needs a diverging shift  renormalization to exist.  $Z\in L^1(P )$ because of the expectation over $Q$.  We call $Z$  a \emph{randomized shift}, see, for example, \cite{shamov}. We will formally introduce randomized shifts in Section \ref{randomizedshifts}.
    \item The solution of \eqref{inteq-intro} does \emph{not} satisfy the  semigroup property, \begin{equation}\notag u(x,t+s)\neq \int u(y,s)u(x-y,t)dy.\end{equation}
    This is due to the fact that the renormalization in the exponential is done by substracting the self-intersection local time $\sum_{j=1}^\infty m_j^2/2$, which does not grow linearly in time.
     On the other hand, for every $n, \lambda$, $\hat u_n(x,t; \vec{\lambda})=E_P[e^{i\lambda \cdot \xi}u_n(x,t)]$ satisfies an ordinary heat equation (we used this in \eqref{e:w-Fub}), and hence the semigroup property.
\end{enumerate}

\subsection{Solving the planar Wick-ordered  heat equation}
\label{ss:2dshe-solve}
\ \\
Proposition \ref{p:spde-projection} shows that for almost all $x,t$ the solution $u_n(x,t)$ is a martingale in $n$.

\begin{proposition}[Solving the Wick-ordered  heat equation]\label{p:solving2d}
Suppose that for almost all $(x,t)$ the martingale $u_n(x,t)$ converges in $L^1(P)$. Then the limit $u(x,t)$ is a solution of the Wick-ordered  heat equation \eqref{2dshe}. The solution is unique $x,t,\xi$-almost everywhere.
\end{proposition}

Almost everywhere uniqueness here means that any two solutions agree $x,t,\xi$ a.e.
Note that there is considerable work involved in showing $L^1$ convergence (the assumption of Proposition \ref{p:solving2d}). This will be addressed in the later sections of the article.

\begin{proof}
First we have to show that for almost all $x,t$ if we write
$$
u(x,t)=\lim_{n\to \infty} u_n(x,t)
$$
then we also have
\begin{equation}\label{e:pu-limit}
\sint {\mathcal K}_{x,t} u ~\xi = \lim_{n\to \infty} \sintn {\mathcal K}_{x,t} u_n ~\xi.
\end{equation}
The finite-dimensional version of the Wick-ordered  heat equation \eqref{e:finite-she} says
\begin{equation}\label{42a}
    \sintn {\mathcal K}_{x,t} u_n \xi = u_n(x,t)- \int p(x-y,t)d\varsigma(y)
\end{equation}
since $u_n(x,t)$ is uniformly integrable, so is $\ssintn {\mathcal K}_{x,t} u_n \,\xi$.
Thus by part $(2)$ of Corollary \ref{c:Sk-martingale},  ${\mathcal K}_{x,t} u$ is Skorokhod integrable, and by the last statement of the corollary, \eqref{e:pu-limit} holds almost surely and in $L^1$, and so $u$ satisfies the Wick-ordered  heat equation \eqref{inteq-intro}.

The solution is unique, since it is determined by its conditional expectations, which are unique by Propositions \ref{p:spde-projection} and \ref{p:fshe-solve}.
\end{proof}

\subsection{The SHE in one dimension with space-time noise}

In this section, we outline how the $1+1$ dimensional stochastic heat equation, usually defined through a chaos expansion, can be rigorously defined and solved using the Skorokhod integral.

The method we describe works for a general stochastic PDE of the form $\partial_t v=\mathcal L v + \xi v$, where $\mathcal L$ is a generator for a stochastic process. \emph{Algebraically}, such equations are equivalent to Example \ref{ex:MC}, although in many cases there are technical challenges or even analytic obstacles to a solution.

There is no significant  \emph{algebraic} difference between the pure space and space-time version of the processes: the time coordinate can just be added as an extra  spatial coordinate.

\begin{definition}\label{d:1dshe}
Let $\varsigma$ be a finite  measure on $\mathbb R$.
Then $z\in L^1((0,T]\times\mathbb R\times \Xi)$ is a solution of the {\bf one-dimensional   stochastic heat equation with space-time white noise} and initial condition $\varsigma$
if
for almost all $(x,t)\in (0,T]\times \mathbb R$, the random function of $(y,s)$
\begin{equation} q(x-y,t-s)\, z(y,s)
\end{equation}
where $q(x,t)$ is the one-dimensional heat kernel \eqref{onedheatkernel},
is Skorokhod integrable, and satisfies
\begin{equation}\label{inteq'}
z(x,t) = \int q(x-y,t)d\varsigma(y) +   \sint q(t-\cdot,x-\cdot)z~ \xi.
\end{equation}
\end{definition}

Although we are working in $[0,\infty)\times \mathbb R$ we can just think of it as a subset of $\mathbb R^2$ and use the same orthonormal basis $\basis_j$ of $L^2(\mathbb R^2)$ of bounded functions.

 \begin{proposition}\label{p:1dshe-project} Let $z$ be a solution of the one-dimensional
   stochastic heat equation, Definition \ref{d:1dshe}. Then
   $z_n(x,t):=E[z(x,t)|\mathcal F_n]$ solves the following problem:
   For almost all $x\in\mathbb R$ and $t>0$ 
   \begin{equation}
       \label{11deq}
z_n(x,t)=\int q(x-y,t)d\varsigma(y)+\sintn z_n(y,s)q(x-y,t-s)\xi.
   \end{equation}

 \end{proposition}
 \begin{proof} As in the proof of Proposition \ref{p:spde-projection}, it is enough to
   show that
$$
E[\sint z(y,s) q(x-y,t-s)\xi|\mathcal F_n]=\sintn z_n(y,s)q(x-y,t-s)\xi.
$$
This follows as before by Fubini and the fact that  $z$ is Skorokhod-integrable.
\end{proof}

\begin{remark}In this case, \eqref{11deq} is an integral version of the  partial differential equation
$$
\partial_t z_n = \tfrac12\partial_{x}^2z_n +  \sum_{j=1}^n \basis_j\, ( \xi_j - \partial_{\xi_j})
z_n
$$
with $\basis_j=\basis_j(x,t)$.
\end{remark}

\begin{proposition}\label{p:1dshesolve} (Existence, uniqueness and explicit solution for
  finite $n$) The solution $z_n$ of equation (\ref{11deq}) exists and is unique in the following sense. For almost
  all $x,t,\xi$ we have

  $$
  z_n(x,t)=\int Z_n(y,0;x,t)q(x-y,t)d\varsigma(y)
  $$
  where $Q_{y,x,t}$ is the law of one-dimensional Brownian bridge $b$ from $y$ at time $0$ to $x$ at time $t$ and $Z_n(y,0;x,t)=E_{Q_{y,x,t}}M(t)$ with 
  \begin{equation}\label{onedpolymer}
  M(t)=\exp\left(\sum_{j=1}^n m_j\xi_j-\frac{1}{2}m_j^2\right),
  \qquad m_j(t)=\int_0^t e_j(b(s),s)ds.
  \end{equation}
  \end{proposition}

\begin{proof} The proof is really the same as Proposition \ref{p:fshe-solve}.  
More precisely, by Example \ref{ex:sk-transport}, for every $B:(0,T]\to \mathbb R^2$,  $M(t)=M(b,t)$ of \eqref{e:mj} solves
\begin{equation}\label{e:pathwise-D2}
M(t)= 1+\int_0^t\sintn M(s)\delta_{(b(s),s)}\,\xi\,ds, \quad t\in (0,T].
\end{equation}
Consider the finite measure $Q'_{x,t}=q(x-y,t)Q_{y,x,t}d\varsigma(y)$. Taking $Q'_{x,t}$-expectations of \eqref{e:pathwise-D2} yields \eqref{11deq}. Uniqueness is shown the same way as in the proof of Proposition \ref{p:fshe-solve}.
\end{proof}

To take the limit as $n\to\infty$, we will need the moment bounds established in the next section. The last step of the solution is given in Section \ref{ss:1dshe-solve}.

\section{Randomized shifts}\label{randomizedshifts}

In this section, we consider the solution of the Wick-ordered  heat equation from another angle, that of randomized shifts, or, using different terminology, Gaussian multiplicative chaos. The definition we use is closely related to \cite{shamov}. We will first define the notion and give motivation later.


As before, let $(\Xi,\mathcal F, P)$ be a probability space given by an independent standard Gaussian sequence:  where $\Xi=\mathbb R^{\mathbb N}$, and $\mathcal F$ is the Borel $\sigma$-algebra on product space, and $P$ is standard Gaussian product measure.

Let $ m=(m_1,m_2,\ldots)$  be a sequence of real valued random variables on another probability space $(\Omega,\mathcal G, Q)$.

\begin{definition}\label{d:Z}
The random variable $Z$ defined on $\Xi$ is called the {\bf partition function}, or {\bf total mass} of the {\bf randomized shift}  $m$ if for all bounded measurable functions $F:\Xi \to \mathbb R$, we have
\begin{equation}\label{aaa}
        E_P[ZF] = E_{P \times Q} F(\xi+m).
\end{equation}
\end{definition}

Note that  $Z$ depends only on the law of $m$, not $m$ as a random variable; nonetheless for simplicity we will sometimes indicate the dependence by writing $Z=Z(m)$. Definition   \ref{d:Z} can be interpreted as follows: the marginal law of $\xi+m$ is absolutely continuous with respect to the law of $\xi$ with Radon-Nykodim derivative $Z$.

The definition is much more concrete when we consider $\mathcal P_n m$, that is, $m$ with all coordinates beyond the $n$th set to zero. In this case, one checks directly that
\begin{equation}\label{zedenn}
Z_n(m):=Z(\mathcal P_n m)  = E_Q \exp \sum_{j=1}^n m_j \xi_j -m_j^2/2
\end{equation}
satisfies the definition.
Note that the right hand side is well-defined and finite for arbitrary $m$, moreover, it is a non-negative  $\mathcal F_n$-martingale and therefore converges almost surely. We will need to know when the convergence is in $L^1$. This  requires  uniform integrability.
In what follows, given an event $A\in\mathcal G$, we adopt 
the notation 
\begin{equation}
\label{aem}
1_A m=(1_Am_1,1_Am_2,\ldots).
\end{equation}

We have the following representation.

\begin{proposition}  \label{shift1}
For any $F\in \mathcal F$ for which $E_P|Z_nF|<\infty$ or $F\ge 0$, with $Z_n$ given by (\ref{zedenn}) with $\mathcal P_n$ the projection to the first $n$ coordinates, defined in the preamble to \eqref{peeenn},
\begin{equation}\label{e:EZn}
   E_{P}\left[Z_nF\right] = E_{Q\times P} \,F(\xi + \mathcal P_nm).
\end{equation}
 In particular,
\begin{enumerate}[(a)]
\item  $E_{ P} Z_n =1$;
\item
$
    E_{ P}\left[Z_n^2\right]=E_{ Q^{\otimes 2}} \exp \sum_{i=1}^n m_i m'_i$.
The second expectation is over the product of $ Q$ with itself and $m$, $m'$ are two independent copies. Note that this may be infinite, even for finite $n$.
\item  $$E_{ P} [Z^k_n]
=
E_{ Q^{\otimes k}}
\exp\left\{\sum_{j,\ell=1, j\ne \ell}^k\sum_{i=1}^n m_i^{(j)} m^{(\ell)}_i\right\}
$$ where $m^{(1)},\ldots, m^{(k)}$ are independent copies with the same caveat as in (b).
\item For any $A\in \mathcal G$, $E_P|Z_n(m)-Z_n(1_Am)|\le 2Q(A^c)$.
\end{enumerate}
\end{proposition}

\begin{proof}Note that
\begin{align*}
E_{ P}[ Z_{n} \mid \mathcal F_{n-1}] &= E_{ Q} \Big[ \exp\{\sum_{i=1}^{n-1} m_i \xi_i- \tfrac12 \sum_{i=1}^n m_i^2 \} E_{ P}[  \exp\{m_{n}\xi_{n} - \tfrac12 m_{n}^2]\}\big]\\&=Z_{n-1}.\end{align*}   
\eqref{e:EZn} is just the definition of the Gaussian measure together with Fubini's theorem, which applies because the integrand is non-negative.  To prove (a-b-c) by induction take $F= Z_n^{k-1}$ in \eqref{e:EZn} and expand the exponential. For $d$, note that if $m_c$ is the conditional law of $m$ given $A^c$, then 
\[
|Z_n(m)-Z_n(1_Am)|=Q(A^c)|Z_n(m_c)-1|.
\qedhere\]
\end{proof}

We already have examples from the previous section.

\begin{example}
In Example \ref{ex:MC} of a continuous time, finite state-space Wick-ordered  polymer we found the following. For the polymer ending at time $t$ and position $y$  the partition function can be written as
$$
p(y,t)=E_Q \bigg[\exp\Big\{\sum_z X(z)\xi_z-X(z)^2/2\Big\}\one(X_t=y)\bigg].
$$
The expectation is under a measure $Q$ independent of the Gaussian space $\Xi$. Under the measure $Q$, the variable $X_t$ is the position of the original Markov chain at time $t$ when started from $x$, and $X(z)=\int_0^t \one (X_s=z)ds$ is the time  spent at site $z$.

This is exactly in the form of Definition \ref{d:Z} with $m_j=X(s_j)$ for $j\le |S|$ and $m_j=0$ otherwise.
\end{example}

\begin{example}[Bayesian statistics]\label{e:Bayes1} Consider trying to guess the mean of a Gaussian random variable with variance $1$. The mean, unknown to us, is zero. Our prior distribution on $\Omega=\mathbb R$ is denoted $Q$. Let $n$ be fixed and let $$m_1(\omega)=\ldots= m_n(\omega)=\omega.$$ 

If the first $n$ samples are  $\xi_1,\ldots \xi_n$, and our opinion is that the mean is $\omega$, the likelihood (the relative odds, expressed as a density, of getting this outcome given our opinion) equals $\exp\sum_{i=1}^n (\xi_i-\omega)^2(2\pi)^{-n/2}$.
The evidence for our opinion is defined as the marginal density  of the data given our opinion, is given by $E_Q\exp\sum_{i=1}^n (\xi_i-\omega)^2=Z_n\exp^{-|\xi|^2/2}(2\pi)^{-n/2}$. The ``polymer'' interpretation of the posterior distribution is given in Example \ref{e:Bayes2}.
\end{example}

\subsection{Convergence in $L^1$}

Since $Z_n$ given in \eqref{zedenn} is a non-negative martingale, we have $Z_n\to Z$,  $P$-almost surely.  In order to extend \eqref{e:EZn} to $Z$, it is important that mass not be lost in the limit.

\begin{lemma}\label{ec}
Suppose that  either of the following equivalent conditions is satisfied: there exists a $Z$ such that $Z_n\to Z$  in $L^1(\Xi, \mathcal{F}, P)$, or $Z_n$ is uniformly integrable.
Then the marginal law
\begin{equation}\label{hatp}
    \hat{P}(A) : = (P\times Q)(\xi+m\in A)
\end{equation}
of $\xi+m$ on $(\Xi,\mathcal F)$ is absolutely continuous with respect to $P$. Also,  $P$-almost surely we have that $Z:=\lim_{n\to \infty} Z_n$, and $Z$ is the Radon-Nikodym derivative \begin{equation}
    Z=\frac{d\hat{P}}{dP}.\end{equation}  
\end{lemma}

 \begin{proof}  For $A\in \mathcal F$ define a probability measure $\tilde P(A) = E_P[ Z\mathbf{1}_A]$.  We want to show that $\tilde P = \hat P$.  For $A\in \mathcal F_n$, $E_P[ Z\mathbf{1}_A]= E_P[ Z_n\mathbf{1}_A]$ by the martingale property, and $$E_P[ Z_n\mathbf{1}_A]= (P\times Q)(\xi+m\in A)$$ by the  fact that $A\in\mathcal A_n$ and Proposition \ref{shift1}.  So
 $\tilde P = \hat P$ on $\mathcal S=\bigcup_{n=1}^\infty\mathcal F_n$. Since $\mathcal S$ forms a $\pi$-system, by the $\pi-\lambda$ theorem $\tilde P=\hat P$ on $\sigma(\mathcal S)=\mathcal F$. 
\end{proof}

\subsection{Properties of randomized shifts}

\begin{lemma}[0-1 law]\label{01law}
If each  $|m_i|\le b_i$ for some deterministic $b_i<\infty$, then $P (Z=0)\in \{0,1\}$.
\end{lemma}

A slightly different 0-1 law for GMC is shown in \cite{mukherjee2016weak}.

\begin{proof}
Fix $k\ge 1$. We write
$
Z=\lim_{n\to\infty}E_Q[XY_n]
$
where
$$
X=\exp \sum_{i=1}^k\xi_i m_i-m_i^2/2, \qquad Y_n=\exp \sum_{i={k+1}}^n \xi_im_i-m_i^2/2
$$
Note that for each fixed $\xi$, $X$ is bounded from above, so  $\lim_{n\to\infty}E_Q[Y_n]=0 \;\Rightarrow \;\lim_{n\to\infty}E_Q[XY_n]=0$ The assumption also implies that $X$ is bounded from below by a number  $\kappa>0$, so that $$\lim_{n\to\infty}E_Q[XY_n]\ge \limsup_{n\to\infty}\kappa E_Q[Y_n]\ge 0.$$ Hence 
$$
\lim_{n\to\infty}E_Q[XY_n]=0 \;\Leftrightarrow \;\lim_{n\to\infty}E_Q[Y_n]=0.
$$
Since the second event is in $\sigma(\xi_{k+1},\xi_{k+2},\ldots )$, so is the first, and since this holds for all $k$,  $\{Z=0\}$ is a tail event for the $\xi$.  Kolmogorov's 0-1 law implies the claim.
\end{proof}

Let $W_n=\exp\{-\sum_{i=1}^n m_i^2/2\}$, and let $w_n=E_QW_n$. Let $\tilde Q_n$ be the distribution of $m_1,\ldots, m_n$, biased by $W_n$, that is, for any test function $F$
\begin{equation}\label{e:Qtdef}
w_n E_{\tilde Q_n}[F]= E_Q[W_nF].
\end{equation}
Clearly, the distributions $\tilde Q_n$ determine the law of $m$.  The following proposition tells us that $Z$ determines $m$.
\begin{proposition} Under $\tilde Q_n$,  $(m_1,\ldots,m_n)$ has all exponential moments. \\ If $E_PZ=1$, then  the Laplace transforms satisfy
\begin{equation}\label{e:abovedisplay}
w_nE_{\tilde Q_n}[\exp \sum_{i=1}^n \xi_i m_i]=E_P[Z\,|\,\xi_1,\ldots,\xi_n].
\end{equation}
\end{proposition}
Note that the left hand side of \eqref{e:abovedisplay} is analytic in $\xi_1,\ldots \xi_n$, so even though the right is only defined almost everywhere, it has a unique continuous version.
\begin{proof} For the first claim, let $1\le k\le n$. We have
for all $t\in\mathbb R$,

$$
w_nE_{\tilde Q_n} \exp \{t m_k\}=E_Q \exp \Big\{t m_k -\sum_{i=1}^n m_i^2/2 \Big\}<\infty.
$$
since the expression in the expectation is bounded.

For the second, note that by definition \eqref{e:Qtdef} the left hand side equals
$$
Z_n=E_Q[\exp \sum_{i=1}^n \xi_i m_i-m_i^2/2]
$$
our usual martingale approximation of $Z$. So we need to prove
$
Z_n=E_P[Z\,|\,\xi_1,\ldots,\xi_n].
$
This is equivalent  to uniform integrability, or $E_PZ=1$.
\end{proof}

\subsection{Existence of the randomized shift}

As usual, the trouble is that none of the equivalent conditions of Lemma \ref{ec}
 (convergence in $L^1$ or uniform integrability of the martingale sequence) are practical to check.  Not surprisingly, there is a simple $L^2$ condition, and a restricted variant.

Define the {\bf intersection exponential} as
\begin{equation}\label{eqthing}
    \mathfrak e(m) = \lim_{n\to\infty} E_P[Z_n^2]=\lim_{n\to \infty}E_{ Q^{\otimes 2}}\exp \sum_{i=1}^n m_i m'_i.
\end{equation}
Note that $\mathfrak e(m)$ depends only on the law of $m$.
Here $m$ and $m'$ are independent copies, and $Q^{\otimes 2}$ denotes the product measure.
Since $E_{ P}\left[Z_n^2\right]$ is a martingale variance, \eqref{eqthing} is a limit of an increasing sequence. Hence $\mathfrak e(m)$
is always defined, but it may be infinite. 

\begin{proposition}\label{p:shifts-existence}
\begin{enumerate}
    \item
If $\mathfrak e(m)$ is finite, then the  partition function of the randomized shift exists,
$$
Z =\lim_{n\to\infty} Z_n  \qquad  P\mbox{-a.s. \;\; and \;\;in  }L^2(\Xi,\mathcal F, P),
$$
and $E_{ P}[Z^2]=\mathfrak e(m)$.
\item \label{p:shifts-existence-hard}If for every $p<1$ there is a $\mathcal G$-measurable set $A$ with $Q(A)>p$ so that $\mathfrak e(\one_A m)<\infty$  (where the notation $1_Am$ is defined in (\ref{aem})), then $Z$ exists and
$$
Z =\lim_{n\to\infty} Z_n  \qquad  P\mbox{-a.s. \;\; and \;\;in  }L^1(\Xi,\mathcal F, P).
$$
\end{enumerate}
\end{proposition}

\begin{proof}[Proof of Proposition \ref{p:shifts-existence}] For 1, $Z_n$ is an $L^2(P)$-bounded martingale, so it converges in $L^2(P)$ and almost surely.  Lemma \ref{ec} identifies the limit as $Z$.

For 2,  briefly call $Z(m)$  the partition function corresponding to $m$. By Proposition \ref{shift1} (d)  we have that for all $n$,
$$
E_P|Z_n(1_A m)-Z_n(m)| \le 2Q(A^c)<2(1-p).
$$
By the triangle inequality this implies that
$$
\sup_{n'\ge n} E_P|Z_n(m)-Z_{n'}(m)| \le 4(1-p)+\sup_{n'\ge n}E_P|Z_n(1_{A}m)-Z_{n'}(1_{A}m)|.
$$
Given $\epsilon>0$, from here we see that we can choose first $p$ so that the first term of the above inequality is smaller than $\epsilon/2$, which fixes $A$, and then $n$ large enough so that the second term is also smaller than $\epsilon/2$, which proves that $(Z_n(m))_{n\ge 1}$  is a Cauchy sequence.
Hence $Z_n(m)$ converges in $L^1(P)$. The rest follows from Lemma \ref{ec}.
\end{proof}

\begin{remark}[Is there a rabbit?]\label{r:rabbit}
A rabbit possibly ran through a field and left its tracks.
Can the fox tell if the rabbit has been there?

In our setting, the  field is planar white noise, and the tracks of the rabbit are the occupation measure of Brownian motion.

The information of the rabbitless field is $\xi$, and with the rabbit's tracks, $\xi+m$. The fact that $\xi+m$ is absolutely continuous with respect to $\xi$   means that the rabbit is lucky: if it passed through the field, the fox still cannot tell this with probability one.

The fox's observation $Y$ is a sample either from the law of $\xi+m$ or from the law of $\xi$. The fox may use the likelihood ratio test to decide which one.  In this test, if the two laws are absolutely continuous with respect to a common measure, the choice depends on whether the density ratio passes a certain threshold $a$. In fact, the law of $\xi$ can be the common measure, and then the likelihood ratio is $Z(Y)$ (here we use the definition of the random variable $Z$ as a function $\Xi\to \mathbb R$).

So the fox decides to keep pursuing when $Z(Y)>a$ for some $a$ depending on how hungry it is.
 The Neyman-Pearson Lemma says that this is the most powerful among tests of the same level.

Now consider a discrete model of a simple random walk as in Example \ref{ex:MC}. One may add a parameter $\beta$ in front of white noise, representing the density of the brush. The  regime will be decided based on the rabbit problem. If the fox has asymptotically no chance, then $Z$ converges to 1 and we expect an additive heat equation (and a Gaussian regime). On the boundary, we expect a multiplicative heat equation (crossover regime), and the asymptotic KPZ fluctuations -- the Tracy-Widom fluctuations -- probe the regime of the lucky fox.

See \cite{berger2013detecting} for a related problem.
\end{remark}

\begin{example}[The Gaussian multiplicative chaos on the circle]
Heuristically, for $\gamma\in [0,2)$ the Gaussian multiplicative chaos on the circle is a random finite measure given by
\begin{equation}
    M(A) = \int_A e^{\gamma \zeta (\omega) - \tfrac{\gamma^2}2 E[\zeta ^2(\omega)]} d\omega, \qquad A\subset [0,2\pi), \quad \mbox{(heuristic)}
\end{equation}
where $\zeta $ is  Gaussian free field, a generalized Gaussian process   with covariance
$$K(\omega,\omega')=2\log\frac{1}{|e^{i\omega}-e^{i\omega'}|}.
$$
There are many ways to make this rigorous. It fits in the present setting as follows.
Let  $\Omega=[0,2\pi]$ and let  $Q$ be uniform measure. Let $\gamma\ge 0$, and define the random variables
$$
m_{2k-1}(\omega)=\frac{ \cos(k\omega)}{\sqrt{k/2}}, \qquad m_{2k}(z)=\frac{\sin(k\omega)}{\sqrt{k/2}},\qquad  k=1,2,\ldots
$$
The random variable $Z(\gamma m)$ is the total mass of the Gaussian multiplicative chaos (GMC) on the circle with coupling constant $\gamma$. The random measure $M_\xi$ on $[0,2\pi]$ is the GMC itself. A quick computation gives
$$
\sum_{m=1}^\infty m_i(\omega)m_i'(\omega')= K(\omega,\omega'),
$$
which blows up at the diagonal, $\sum m_i^2=\infty$, so the partition function  has to be defined as a limit. The intersection exponential is given by
$$
\mathfrak e (\gamma m) = \frac{1}{2\pi}\int_0^{2\pi} \frac{1}{|e^{is}-1|^{2\gamma^2}}\, ds= \begin{cases}
\frac{2 \Gamma (-2 \gamma^2 )}{\Gamma (1-\gamma^2
   ) \Gamma (-\gamma^2 )}, \qquad &\gamma<1/\sqrt{2},
   \\
   \infty  & \gamma\ge 1/\sqrt{2}.
   \end{cases}$$
The Gaussian sequence $\xi_0,\xi_1,\xi_2,\ldots$ can be identified with white noise on the circle $\sum \basis_i\xi_i$ using the basis elements $\basis_0=\frac{1}{\sqrt{2\pi}}$, $\basis_k=\frac{1}{\sqrt{2\pi}}\sqrt{k} m_k$, $k \ge 1$. The analogue of the path measure in this case has state space given by the  circle and is the random function 
$$
X=\sum_{k=1}^\infty m_k(\omega)\basis_k(x)=2\Re \operatorname{Li}_{1/2}(e^{i(x-\omega)})
$$
which has a square-root singularity, and is not in  $L^2[0,2\pi]$ for any $\omega$. The rabbit problem, Remark \ref{r:rabbit} in this setting says that the marginal law of $\xi+\gamma X$ is absolutely continuous with respect to white noise $\xi$ exactly when $\gamma<2$. An alternative formulation is in terms of the absolute continuity of  $\zeta+\gamma K(\omega,\cdot)$ with respect to the GFF $\zeta$.
\end{example}

\subsection{Properties of $\sum m_jm_j'$.}
\label{ss:alpha-new}
\ \\
In preparation for our main convergence theorem, we need to establish properties of the sequence $\alpha_n(m,m')=\sum_{j=1}^nm_jm_j'$, or, more generally a sequence 
\begin{equation}\label{e:alpha-new}
\alpha_n(m,w')=\sum_{j=1}^nm_jw_j'
\end{equation} 
The prime notation here means the following. Consider the product  of two copies of $(\Omega,\mathcal G, Q)$. For a random variable $U$ defined on $\Omega$ and $(\omega,\omega')\in \Omega^2$ we write $U=U(\omega)$ to define $U$ on $\Omega^2$, and write $U'=U(\omega')$ to define an independent copy. 

It helps to think of the sequence $m_j$ as a random signed measure on $\mathbb N$, defined by  $m(\{j\})=m_j$. Let $m^{\otimes k}$ denote the random product measure on $\mathbb N^k$, and let $\rho_{k,m}=Em^{\otimes k}$. 
We first establish some facts about $\alpha_n$ and its moments. 
\begin{lemma}Assume $E|m_j|^k,E|w_j|^k<\infty$ for all $j$.  Then
\label{l:rhokn-m}
  \begin{enumerate}
    \item\label{l:rhokn-moments} 
    $E\alpha_n(m,w')^k= 
    \langle\one_{\{1,\ldots,n\}^k}\rho_{k,m},\rho_{k,w}\rangle_{\mathbb N^{\otimes k}}$
  
\item For  $k\in 2\mathbb{N}$,~~ $n\ge \ell$ and $\|\cdot\|_k=\|\cdot\|_{L^k(\Qt)}$ we have 
      $$\|\alpha_n-\alpha_\ell\|^{2k}_k\le\Big( \|\alpha_n(m,m')\|^k_k-\|\alpha_\ell(m,m')\|^k_k\Big)\Big(\|\alpha_n(w,w')\|^k_k-\|\alpha_\ell(w,w')\|^k_k \Big).$$
       \end{enumerate}
\end{lemma}

\begin{proof}
{\sl 1} follows by expanding the two sums. To prove {\sl 2}, consider the signed measure $\nu=\nu_{m,w}$ on $\mathbb N^k$ with $\nu(x)=\rho_{k,m}(x)\rho_{k,w}(x)$. When $m=w$, we have  $\nu\ge 0$ and  
\begin{align}\label{e:tri}
E\alpha_n^k &=\nu(\{1,\ldots,n\}^k)\ge \nu(\{1,\ldots,\ell\}^k)+\nu(\{\ell+1,\ldots,n\}^{k})\\ 
& =E\alpha_\ell^k+ E(\alpha_n-\alpha_\ell)^k \notag
\end{align}
since the last two sets are disjoint  subsets of the first one. The inequality holds for all $k\in \mathbb N$, but  without absolute values \eqref{e:tri} describes $L^k$ norms only for even $k$.

When $m$ and $w$ may be different, by Cauchy-Schwarz, with $A=\{\ell+1,\ldots,n\}^{k}$ we have
$$(E(\alpha_n-\alpha_\ell)^k)^2=(\nu_{X,Y}(A))^2\le \nu_{X,X}(A)\nu_{Y,Y}(A)
$$
and the claim follows as before.
\end{proof}
\begin{proposition}\label{p:rhok-m}
  Let $k_0\in 2\mathbb{N}$, and assume that 
  $\rho_{{k_0},m},\rho_{{k_0},w}\in \ell^2(\mathbb N^k)$.
  Then for integers $1\le k\le k_0$, 
  \begin{enumerate}
  \item  $\alpha_n(m,w')$ has a limit $\alpha(m,w')$ in $L^k(\Qt)$. 
  \item
  $
      E\alpha(m,w')^k=\langle \rho_{k,m},\rho_{k,w}\rangle  $.
  \item 
     $E[\alpha(m,w')^k]^2\le E[\alpha(m,m')^k]E[\alpha(w,w')^k]$ 
 \end{enumerate}
\end{proposition}
\begin{proof}  
 By Lemma \ref{l:rhokn-m} and the assumption, $\alpha_n$ is Cauchy in
$L^{k_0}$, so also in $L^k$.
Taking limits of {\sl 1} of the previous lemma gives {\sl 2}. 
{\sl 3} follows from the $\ell^2$ representation in {\sl 2}. 
\end{proof}

The following lemma will be useful to understand the intersection exponential. 

\begin{lemma}\label{l:i-exp bound}
If $\rho_{k,m}\in \ell^2(\mathbb N^k)$ for all $k$, then $\alpha_n=\alpha_n(m,m')$ has a limit $\alpha$ in $L^k$ for every $k$. If  $E\cosh(\alpha)<\infty$, then   
$$\mathfrak e(m):=\lim_{n\to\infty} Ee^{\alpha_n}= Ee^{\alpha}<\infty.
$$
\end{lemma}
\begin{proof} 
The first claim is a consequence of 
Proposition \ref{p:rhok-m}. Using Lemma \ref{l:rhokn-m} as well, we get 
$$E\alpha_n^{k}=\|\one_{\{1,\ldots,n\}^k}\rho_{k,m}\|^2_{\ell^2(\mathbb N^k)}\le
\|\rho_{k,m}\|^2_{\ell^2(\mathbb N^k)}=E\alpha^{k}.
$$
Summing with coefficients, using positivity (the reason for using $\cosh$ instead of $\exp$) Fubini twice, and dominated convergence,  we get 
\begin{align*}
E\cosh \alpha_n & = E \sum_{k=0 }^\infty \frac{\alpha_n^{2k}}{(2k)!}=\sum_{k=0}^\infty \frac{E\alpha_n^{2k}}{(2k)!}\to 
 \sum_{k=0 }^\infty \frac{E\alpha^{2k}}{(2k)!}\\&=E\cosh \alpha<\infty,
 \end{align*}
so  $\cosh(\alpha_n)\to\cosh(\alpha)$ in $L^1$. 
But $f(x)=e^{x}/\cosh(x)$ is a bounded continuous function, so $f(\alpha_n)\cosh(\alpha_n)\to f(\alpha)\cosh(\alpha)$ in $L^1$. Thus $Ee^{\alpha_n}\to Ee^{\alpha}$.
\end{proof}

Next, we consider a setting where sequences ${\mathbf m}_\ell=(m_{l,1}, m_{l,2},\ldots)$ converge to a limiting sequence $\mathbf m=(m_1,m_2,\ldots)$.
\begin{lemma}[Convergence of  $\alpha$]
\label{l:alpha-c-law}
Assume that
\begin{enumerate}[(i)]
\item $\lim_{\ell\to\infty} m_{\ell ,j}=m_j$ in probability for each $j$, 
\item
for each $\ell$, as $n\to\infty$, 
$$
S_{\ell,n}=\sum_{j=1}^n m_{\ell,j}m'_{\ell,j}
\to\alpha_\ell, \qquad 
S_n=\sum_{j=1}^n m_{j}m'_{j}
\to\alpha\quad \mbox{ in }L^2(\Qt),$$
\item
$\lim_{\ell \to \infty}E\alpha_\ell^2 =E \alpha^2$.
\end{enumerate}

Then as $\ell\to\infty$, 
$\alpha_\ell\to \alpha$ in probability. 
\end{lemma}
\begin{proof}
Let $\varphi(x)=|x|\wedge 1$, so that $E\varphi(X)$ metrizes convergence in probability.
Note that $E\varphi(X)\le E|X|\le (EX^2)^{1/2}$. Let $q_{\ell,n}=E S_{\ell,n}^2$ and $q_{n}=ES_n^2$.
By the $L^2$ moment formulas of Lemma \ref{l:rhokn-m} and Proposition \ref{p:rhok-m}, we have 
$$
E(\alpha-S_n)^2\le E\alpha^2-q_n, \qquad 
E(\alpha_\ell -S_{\ell,n})^2\le E\alpha_\ell ^2-q_{\ell,n}.
$$
Given $\eps>0$, let $n$ be so that
$$
E \alpha^2 -q_n <\eps.
$$
Let $\ell_0$ be so that for all $\ell\ge \ell_0$
$$
q_n< q_{\ell,n}+\eps,\qquad  E\alpha_\ell^2< E\alpha^2+\eps, \qquad E\varphi( S_{\ell,n}-S_n)<\eps.
$$
In the first inequality we used Fatou's Lemma.  Then
\begin{align*}
E\varphi(\alpha_\ell-\alpha)&\le (E(\alpha_\ell-S_{\ell,n})^2)^{1/2}+ E\varphi(S_{\ell,n}-S_n) +(E(\alpha-S_n)^2)^{1/2}\\
&\le (E\alpha_\ell^2-q_{\ell,n})^{1/2}+\eps+ (E\alpha^2-q_n)^{1/2} < \sqrt{3\eps}+\eps+\sqrt{\eps}.
\qedhere\end{align*}
\end{proof}

\subsection{Convergence of  randomized shifts partition functions}

One of the advantages of the randomized shifts approach is that it is very effective for proving convergence of models.
We will have a sequence of random sequences $\mathbf m_\ell$ converging to a random sequence $\mathbf m$ in probability.  Since  $Z(\mathbf m_\ell)$ only depends on the law of $\mathbf m_\ell$, we can couple the $\mathbf m_\ell$ in any way we like to get to the present setting. 

\begin{proposition}\label{propmain}

Let $\mathbf m_\ell = (m_{\ell ,1},m_{\ell ,2}, \ldots )$, $\ell  = 1, 2, \ldots$ be a random sequence.  Suppose that
    $m_{\ell,j}\to m_j$ in probability for each $j$. Let ${\mathbf m}=(m_1,m_2,\ldots)$. 
\begin{enumerate}
\item  Assume  $\mathfrak e(\mathbf m_\ell ) \to \mathfrak e(\mathbf m)<\infty.$
Then $Z(\mathbf m_\ell)\to Z(\mathbf m)$ in $L^2(\Xi, \mathcal{F}, P )$.
\item \label{propmain-hard}Assume that  for all $k\ge 1$ and each $\ell$, as $n\to\infty $,
$$
S_{\ell,n}=\sum_{j=1}^n m_{\ell,j}m'_{\ell,j}
\to\alpha_\ell, \qquad 
S_n=\sum_{j=1}^n m_{j}m'_{j}
\to\alpha \quad \mbox{ in }L^k(\Qt).$$
Suppose that  $E\alpha_\ell^2  \to E\alpha^2$
and that for some  $\gamma>1$ the following holds.  For each $p<1$, there exists a sequence of $\mathcal G$-measurable sets $ (A_{\ell})_{\ell\ge 1}$  with
$ Q(A_{\ell})> p$ and
\begin{equation}\label{cutoff}
     \limsup_{\ell\to\infty} E\cosh(\gamma \one_{A_{\ell}}\mathbf \alpha_\ell) <\infty.
\end{equation}
Then $Z(\mathbf m_\ell)\to Z(\mathbf m)$ in $L^1(\Xi, \mathcal{F}, P )$.
\end{enumerate}
\end{proposition}
Here, for clarity, we use $Z(\bf m)$ to denote the dependence on $\bf m$ even though it really only depends on the law of $\bf m$.  

\begin{proof}
For  1, we can bound $\tfrac13 E_{P }[ (Z(\mathbf m_\ell)- Z(\mathbf m)^2]$
by
\begin{align}\label{split1}
     E_{P }[ (Z({\mathbf m_\ell})- Z_n({\mathbf m}_\ell))^2] &+
    E_{P }[ (Z_n({\mathbf m _\ell})- Z_n({\mathbf m}))^2]\\& +E_P[ (Z_n({\mathbf m})- Z({\mathbf m}))^2].
    \notag
\end{align}
Let
$$q_{n,\ell}= E_{P }[ (Z_{n+1}({\mathbf m_\ell})- Z_n({\mathbf m}_\ell))^2], \qquad q_{n}= E_{P }[ (Z_{n+1}({\mathbf m})- Z_n(\mathbf m))^2]. $$
The weak convergence of $\mathbf m_\ell$ to $\mathbf m$ and the fact that $e^{\lambda x - \tfrac12 x^2}$ is bounded and continuous in $x$, means that for each $n$,
\begin{equation}\label{Zas}
Z_n(\mathbf m_\ell)\to Z_n(\mathbf m)
\end{equation}
$P$-almost surely. Let $\|\cdot\|_1$ denote the sequence $\ell^1$- norm. By the martingale property,
$$\mathfrak e(\mathbf m_\ell)=\|q_{\cdot,\ell}\|_1,\qquad \mathfrak e(\mathbf m)= \|q\|_1$$
Let $\hat q_n =\liminf_{\ell\to \infty} q_{n,\ell}$. By Fatou's lemma applied to sequences, 

\begin{equation}
\nonumber
\lim_{\ell\to\infty} \|q_{\cdot,\ell}\|_1 \ge \|\hat q\|_1.
\end{equation}
  By \eqref{Zas} and Fatou's lemma applied in the probability space $\Xi$ to $|Z_{n+1}(\mathbf m_\ell)-Z_{n}(\mathbf m_\ell)|^2$, we have  $q_n\le\hat q_n$.  After summing over $n$ we get
$$
\mathfrak e(\mathbf m)
=\|q\|_1\le \|\hat q\|_1
\le 
\lim_{\ell\to\infty}\|q_{\cdot,\ell}\|_1
=
\lim_{\ell \to\infty} \mathfrak e(\mathbf m_\ell)
=
\mathfrak e(\mathbf m), \qquad 0\le  q_n\le \hat q_n,
$$
where the last equality is our assumption. This implies $\hat q=q$. It follows that $q_{n,\ell}\to q_n$ coordinate-wise  (since then $\lim_{l\to\infty}||q_{\cdot,\ell}||_1=||\hat q||_1$  and by definition $\liminf_{\ell\to\infty}q_{n,\ell}=\hat q_n$) and also in $\ell^1$.

From the martingale property, the first and third terms of \eqref{split1} are given by $\sum_{i=n}^\infty q_{i,\ell}$ and $\sum_{i=n}^\infty q_{i}$. The convergence $\|q_{\cdot,\ell}-q\|_1 \to 0$ implies that for large enough $n_0$ the first term (and the easier third term) in  \eqref{split1} can be made uniformly small for all $\ell\ge 1, n\ge n_0$. Since the $L^2(P)$-norm of the left hand side converges to that of the right in \eqref{Zas}, the claimed convergence also holds in $L^2(P)$.  Thus the middle term in \eqref{split1} vanishes as $\ell\to\infty$. This proves Case 1.

For  2, for every $\kappa$, we can pick a sequence of events $(A^\kappa_\ell, \ell \ge 1)$ satisfying \eqref{cutoff} and $Q(A^\kappa_{\ell})\ge 1-1/\kappa$. To use part 1, we need to create convergence in probability artificially. Consider the random variable 
$
\Upsilon_\ell=(\mathbf m_\ell, (\one_{A^\kappa_\ell}, \kappa\ge 1))
$
taking values in the space of pairs of real-valued sequences. The two coordinates of $\Upsilon_\ell$ are tight as $\ell$ varies. Therefore for any subsequence of $\ell$ there is a further subsequence of $\ell_j$ so that $\Upsilon_{\ell_j}\to \Upsilon$ in law. It suffices to show the conclusion along this subsequence. By relabeling, we can drop the subscript $j$ and assume the convergence $\Upsilon_{\ell}\to \Upsilon$ in law. By the Skorokhod representation, this can be realized as almost sure convergence on some probability space, so we may assume $\Upsilon_{\ell}\to \Upsilon$ almost surely. We may write $\Upsilon=(\mathbf m, (\one_{A^\kappa}, \kappa\ge 1))$ where $A^\kappa$ is some event with $QA^\kappa\ge 1-1/\kappa$.

Next, fix $\kappa$, and use bar to denote $\kappa$-truncated quantities, such as $\mathbf{\bar m}_\ell=\one_{A^\kappa_\ell} \mathbf m_\ell$, $\mathbf{\bar m}=\one_{A^\kappa} \mathbf m$. 

By Lemma \ref{l:alpha-c-law}, we have $\alpha_\ell\to\alpha$ in probability, which implies $\bar \alpha_\ell\to\bar \alpha$ and $\cosh \bar \alpha_\ell \to \cosh\bar \alpha$ in probability. By our assumption, $E\cosh(\gamma \bar\alpha_\ell)$ is bounded for some $\gamma>1$, and therefore $\cosh \bar\alpha_\ell$ is uniformly integrable, so  $E\cosh \bar \alpha_\ell\to E\cosh \bar\alpha<\infty$. Similarly, $Ee^{\bar \alpha_\ell} \to Ee^{\bar \alpha}<\infty$.

By dominated convergence, 
we have $\bar \alpha_\ell =\lim_{n\to\infty} \bar S_{\ell,n}$ in every $L^k(\Qt)$, so by Lemma \ref{l:i-exp bound}, we have $\mathfrak e(\mathbf{\bar m}_\ell)=Ee^{ \bar \alpha_\ell}$, and similarly,  $\mathfrak e(\mathbf{\bar m})=Ee^{\bar \alpha}$. Thus
$$
\mathfrak e(\mathbf{\bar m}_\ell)\to \mathfrak e( \mathbf{\bar m})<\infty.
$$
Now Case 1 applies, and we get $L^2(P)$ convergence, and hence also $L^1(P)$ convergence
\begin{equation}\label{e:akl1}
\lim_{\ell\to \infty} E_P|Z(\mathbf \one_{A^\kappa_\ell}\mathbf m_\ell)-Z(\one_{A^\kappa} \mathbf m)|=0.
\end{equation}
Finally, by taking limits of Proposition \ref{shift1} (d), for all $\ell,\kappa$ we have
\begin{equation}\label{e:akl2}
E_P |Z({\mathbf m_\ell})- Z({\one_{A^\kappa_\ell}\mathbf m_\ell })|\le 2Q((A^\kappa_\ell)^c)\le 2/\kappa,
\end{equation}
similarly, 
\begin{equation}\label{e:akl3}
 E_P |Z({\mathbf m})- Z({\one_{A^\kappa}\mathbf m })|\le 2Q((A^\kappa)^c)\le 2/\kappa.
\end{equation}
Using \eqref{e:akl1}-\eqref{e:akl3} we make $\kappa$ large and then make $\ell$ large  to get $E_P |Z({\mathbf m})- Z(\mathbf m_\ell)|\to 0$.
\end{proof}
\begin{remark}\label{rem:unique} $Z(\bf m)$ does not depend on our choice of basis $\basis_j$. This is because it 
  can be arbitrarily approximated in $L^1(\Xi, \mathcal{F}, P )$ by $Z(\one_{A}\mathbf m)$  for some $A$,  as the above proof shows. The latter is basis-independent since it has
a chaos series in $L^2(\Xi, \mathcal{F}, P )$  as in Example \ref{example25} with \eqref{39} replaced by \begin{equation}\notag 
g_{k,A}(x_1,\ldots,x_k;t,x)= E_{Q\one_{A}}[m(x_1)\cdots m(x_k)].
\end{equation} \end{remark}
The results in this section, as well as the next,  parallel  those of Theorem 25 of \cite{shamov}.  In our setting,  this theorem reads as follows:
\begin{proposition}[Shamov]\label{shamov}
Let $\mathbf m=(m_1,m_2,\ldots)$, $\mathbf m_\ell = (m_{\ell ,1},m_{\ell ,2}, \ldots )$, $\ell  = 1, 2, \ldots$ be random sequences defined on the same probability space.   Suppose that
\begin{enumerate}
\item $Z(\mathbf m_\ell)$ is uniformly integrable
\item  For all $v\in \ell^2$, we have  $\langle \mathbf m_\ell,v\rangle \to \langle \mathbf m,v\rangle$ in $Q$-probability.
\item  $\alpha( \mathbf m_\ell,\mathbf m_\ell' )\to \alpha(\mathbf m,\mathbf m')$ in $Q^{\otimes 2}$-probability for some $\alpha(\mathbf m,\mathbf m')$.
\end{enumerate}
Then $Z(\mathbf m_\ell)\to Z(\mathbf m)$ in $L^1(\Xi, \mathcal{F}, P )$. \\
  Moreover, the random measures $M_\ell$, $M$ defined below in \eqref{e:polydef} exist and 
$M_\ell \to M$ in $L^1$ in the following sense. For any test function $F\in L^1(\Omega, \mathcal G, Q)$, we have $$E_{M_\ell} F \to E_M F\qquad \mbox{ in }L^1(\Xi,\mathcal F, P).$$
 \end{proposition} 
Let us comment on the main differences. 
In Proposition \ref{shamov},  uniform integrability is \emph{assumed}, while we are instead giving explicit conditions which imply the uniform integrability, conditions we  verify in later sections. Also, Assumption 2 of Proposition \ref{shamov} seems harder to verify than convergence in probability of the entries of $\mathbf m$.

\subsection{The polymer measure}\label{ss:polymer}

Up to this point, our main focus was random heat flow in one and two dimensions. Random paths came up as a tool for solving heat equations through the Feynman-Kac formula. In fact, one motivation for heat equations, random or deterministic, is that they describe the partition functions of polymers.

As an example, assign an i.i.d.\ mean one  positive random variables to lattice site in spacetime $\mathbb N\times \mathbb Z$, and let the weight of a simple random walk path be given by the product of the weight of sites its graph visits. When the law of each weight is $e^{\beta G-\beta^2/2}$ for a standard Gaussian $G$ and $\beta\in \mathbb R$, this construction can be described through  a randomized shift.

The remarkable advantage of the random shift formalism is that it  works effortlessly in technically difficult settings like Brownian paths weighted by white noise.  The price we have to pay is a bit of abstraction.

In short, the polymer measure is given by the Gaussian multiplicative chaos on path space. Recall our setting: a probability space $(\Omega, \mathcal G, Q)$ (the path space in the example above) on which there is a random sequence $m$. A Gaussian i.i.d.\ sequence $\xi$ defined on $(\Xi,\mathcal F, P)$. In the special case when the $m_i$ are bounded random variables and $\sum_{i=1}^\infty m_i^2<\infty$, then 
\begin{equation}\label{e:gmc-weight}
\sum_{i=1}^\infty m_i(\omega)\xi_i-m_i(\omega)^2/2
\end{equation}
exists as a limit of partial sums $Q$-almost surely, we can define the new measure
\begin{equation}\label{e:gmc-weighting}
M_\xi =  e^{\sum_{i=1}^\infty m_i(\omega)\xi_i-m_i(\omega)^2/2}Q
\end{equation}
on $\Omega$. The strength of the GMC theory is in the fact that such measures behave much better under limiting operations than the weight functions \eqref{e:gmc-weight}. 

The direct description \eqref{e:gmc-weighting} fails since in many cases the sum is not defined and $M_\xi$ is in fact singular with respect to $Q$. But it is an exercise to check that the formula \eqref{e:polydef} holds in the special case above. It turns out that it captures  enough information in it to be used as a definition. 

\begin{definition} The {\bf polymer measure} or {\bf Gaussian multiplicative chaos} of  the randomized shift $m$ is a random measure $M$ on $\Omega$ parametrized by $\xi$. It is defined by  the property that for all bounded measurable functions $F$ on $\Xi\times \Omega$
\begin{equation}\label{e:polydef}
E_PE_{M_\xi}F(\xi,\omega)=E_{P\times Q}F(\xi+m(\omega),\omega).
\end{equation}
\end{definition}
Here $\xi,m(\omega)$ are real sequences and $\xi+m(\omega)$ is just their termwise sum. This formula, simple as it looks,  contains the definition of the Wick-ordered  polymers. In that case, $\Omega$ is the space of paths, and $m_i(\omega)$ is the integral of the basis element $\basis_i$ with respect to the occupation measure of $\omega$.

Note that $Z(\xi)=M_\xi(\Omega)$ is the partition function, or total mass, of the random measure $M$, and that the $P$-expectation of $M$ is exactly $Q$. The normalized version  $$M_\xi(\cdot)/M_{\xi}(\Omega),$$ is a random probability measure.

For the definition to make sense, we need the following, one of the main results of \cite{shamov}.

\begin{theorem}[\cite{shamov}, Theorem 17 and Corollary 18]\label{t:shamov-polymer}
If the marginal law of $\xi+m(\omega)$ is absolutely continuous with respect to $\xi$, then the polymer measure/Gaussian multiplicative chaos exists and is unique. \end{theorem}

\begin{example}[Discrete lattice polymer in 1+1 dimension]
Let $n(i)$ be an enumeration of the even lattice points in $\mathbb N\times \mathbb Z$, and consider discrete time simple random walk $X$ with $X_0=0$ run until time $t_0$.  $$\mbox{ For } \quad n(i)=(x,t)\qquad \mbox{ let }\quad m_i=\begin{cases}\one(X_t=x)&\mbox{ if }t\le t_0\\ 0 & \mbox{otherwise,}\end{cases}$$ in other words, $m$ encodes the occupation measure of the graph of $X$. In the polymer measure, each path $\omega$ has weight  weight $2^{-t_0}\exp \sum_i \xi_im_i-m_i^2/2$. So for any test function $F$, we have
$$
E_{M_\xi}F(\xi,\omega)=2^{-{t_0}}\sum_\omega \Big( \exp \sum_i \xi_im_i(\omega)-\frac{m_i(\omega)^2}{2}\big) F(\xi,\omega).
$$
using the Radon-Nikodym derivative  for the shifted mean Gaussian, the expression $E_PE_{M_\xi}F(\xi,\omega)$ can be written as
\begin{align*}
2^{-{t_0}}\sum_\omega  &E_P(\exp \sum_i \xi_im_i(\omega)-\frac{m_i(\omega)^2}{2})F(\xi,\omega)=2^{-{ t_0}}\sum_\omega E_PF(\xi+m,\omega)\\&=E_{P\times Q}F(\xi+m,\omega)
\end{align*}
so it satisfies \eqref{e:polydef}.
\end{example}

\begin{example}[Bayesian statistics, continued]\label{e:Bayes2} \ \\Continuing Example \ref{e:Bayes1}, recall that $m_i(\omega)=\omega$. After $n$ observations, the Gaussian multiplicative chaos is given by the unnormalized posterior measure
$\exp\left\{-n\omega^2/2+\omega\sum_{i=1}^n \xi_i \right\}Q$.
The normalized posterior distribution is  $\exp\left\{-n\omega^2/2+\omega\sum_{i=1}^n \xi_i \right\}Q/Z_n$.

This example holds in full generality: one can think of any  Gaussian multiplicative chaos as the posterior distribution of $Q$ after observing the (usually different) statistics $m_1,m_2\ldots$.
\end{example}

\begin{example}[Definition of PAM]\label{PAM}
The planar parabolic Anderson model, PAM, is closely related to the planar Wick-ordered  polymer, whose partition function is 
\begin{equation}\label{e:fornorm}
Z= \lim_{n\to\infty} Z_n, \qquad Z_n=E_Q\Big\{\exp \sum_{j=1}^n m_j\xi_j -m_j^2/2\Big\}
\end{equation}
where $m_j=\int_0^t \basis_j(B(s))\,ds$ and $\basis_j$ is an orthonormal basis of $L^2(\mathbb R^2)$ consisting of bounded functions. For this example, we take $Q$ to be standard Brownian motion measure on some interval $[0,t]$.

Our definition corresponds to smoothing out the noise by replacing  the finite dimensional Gaussian process with $\sum_{j=1}^n\xi_j\basis_j$, using it to define a Wick-ordered  polymer, and then taking a limit. 

PAM should  correspond to the partition function 
\begin{equation}\notag
Z_{PAM}= \lim_{n\to\infty} E_Q\Big\{\exp \sum_{j=1}^n m_j\xi_j -c_{n,t}\Big\} \qquad \mbox{(heuristic)}.
\end{equation}
for some constants $c_{n,j}$.
This means we are solving the deterministic heat equation with respect to the smoothed noise $\sum_{j=1}^n\xi_j\basis_j$, and renormalizing by $e^{-c_{n,t}}$. Proving that this limit exists can be cumbersome, as, unlike in the Wick-ordered  case, there is no natural martingale to consider. A precise definition can be made in a different way. While we don't do it here, it can be shown that 
\begin{equation}\label{e:noWick}
 \gamma(B)=\lim_{n\to\infty} \gamma_n(B)-E_Q\gamma_n(B), \qquad \gamma_n(B) =\frac{1}{2}\sum_{j=1}^nm_j^2, \qquad \mbox{(heuristic)}
 \end{equation}
exists. 
The random variable $\gamma(B)$ is called the self-intersection local time. Heuristically, up to centering, $2\gamma$  wants to be the squared $L^2(\mathbb R^2)$ norm of the occupation measure of $B$. Indeed, $2\gamma_n$ is the squared $L^2(\mathbb R^2)$-norm of the occupation measure projected onto the span of $\basis_1,\ldots, \basis_n$.  Instead of proving \eqref{e:noWick}, we use the definition  given in \cite{LeGall} through mutual intersection local times:
$$
\gamma(B)=\sum_{n=1}^{\infty}\sum_{k=1}^{2^{n-1}} A_{n,k}-E_QA_{n,k}, \qquad A_{n,k}=\alpha(B_{[\frac{2k-2}{2^n}t,\frac{2k-1}{2^n}t]},B_{[\frac{2k-1}{2^n}t,\frac{2k}{2^n}t]}).
$$
The arguments of $\alpha$ in $A_{n,k}$ are conditionally independent given $B(\tfrac{2k-1}{2^n})$, and so each term is defined as in Section \ref{intersectionlocaltime}.

Considering the normalization \eqref{e:fornorm}  one could first try to take $c_{n,t}$ to equal 
$E_{Q}\gamma_n$. But there is a problem:
For the PAM to inherit the semigroup property from the polymer given by Brownian paths weighted by $e^{  \sum_{j=1}^n m_j\xi_j}$, for $0<s<t$ the normalization should be asymptotically additive: 
\begin{equation}\label{e:additive}
\lim_{n\to\infty} c_{n,t}-c_{n,s}-c_{n,t-s}=0.
\end{equation}
For $0<s<t$ let $B_1$ and $B_2$ denote the restrictions of $B$ to $[0,s]$ and $[s,t]$, respectively. Then by bilinearity, we have 
$$
E_Q\gamma_{n}(B)-E_Q\gamma_{n}(B_1)-E_Q\gamma_{n}(B_2)=E_Q\alpha_n(B_1,B_2)
$$
which converges to 
\begin{align*}
E\alpha(B_1,B_2)&=\int_{\mathbb R^2}\int_0^s\int_0^{t-s}p(x,t_1)p(x,t_2)\,dt_2\,dt_1 \,dx
\\ &
= \frac{t \log t -s \log s - (t-s) \log (t-s)}{2\pi}.
\end{align*}
This suggest the time-dependent correction 
$$
c_{n,t}=E\gamma_n(B)-\frac{1}{2\pi}t \log t
$$
and this is indeed asymptotically additive \eqref{e:additive}.

To get a quick definition of PAM, we use $\gamma$ to re-adjust the weight of the Wick-ordered  polymer paths:
$$M_{PAM,\xi}=e^{\gamma(\omega)+\frac{1}{2\pi}t \log t
}M_\xi, \qquad Z_{PAM}=E_{ M_{PAM,\xi}}[\Omega]=E_{M_\xi}e^{\gamma+\frac{1}{2\pi}t \log t
}.
$$
Note that $\gamma$ is a random variable defined on $(\Omega,\mathcal G,Q)$. For $P$-almost all $\xi$, $\gamma$ is defined $M_{\xi}$-almost everywhere. To see this,  let $A$ be a set of full $Q$ measure  so that  $\gamma$ is defined on $A$. Then $E_PE_{M_\xi}\one_{A^c}=E_Q\one_{A^c}=0$ by \eqref{e:polydef}.

We also have

$$E_PZ_{PAM}=E_PE_{ M_\xi}e^{\gamma+\frac{1}{2\pi}t \log t
}=e^{\frac{1}{2\pi}t \log t}
E_Qe^{\gamma}.$$
\cite{LeGall} shows that the right hand side is finite for small $t$, but is infinite for large $t$. It can be shown that in the latter case, $Z_{PAM}$ is still finite $P$-almost surely, but $E_PZ_{PAM}=\infty.$
\end{example}

The following proposition will be used to show that the planar  Wick-ordered   polymer converges to the continuum random directed polymer.  

\begin{proposition}\label{p:polymer-limit} 
Assume that condition 1 or 2 of Proposition \ref{propmain} holds.
Then the polymer measures converge in the following sense. For any bounded test function $F:\Omega\to \mathbb R$
\begin{equation}\label{e:ZF}
E_{M_\ell} F\to E_{M}F \qquad  \mbox{in }L^1(P).
\end{equation}
\end{proposition}
By \eqref{e:ZF}, $M_\ell\to M$ in $P$-probability with respect to weak topology of measures. 
The normalized measures $M_\ell/Z_\ell\to M/Z$ as well, as long as $P(Z>0)=1$.
This holds in our cases of interest by Lemma \ref{01law}.

\begin{proof}
By linearity, we may assume $E_QF=1$, $F\ge 0$.
By the  assumption that $F$ is bounded, the assumptions of  Proposition \ref{propmain} hold for the sequences  $\mathbf m_\ell$ and $\mathbf m$ under the modified probability measure $FQ$. Let $Z_F(\cdot)$ denote the corresponding partition functions. Then
\[
E_{M_\ell}F= Z_F(\mathbf m_\ell) \to Z_F(\mathbf m)=E_MF \qquad \mbox{ in }L^1(P). \qedhere\]
\end{proof}

\subsection{Historical comment} Skorokhod solutions of the 2d multiplicative stochastic heat equation were studied previously in several works starting with  \cite{NualartZakai} as {\em generalized Wiener functionals}, and continuing under the name {\em white noise analysis}, where the non-linearities in stochastic partial differential equations are interpreted via a Wick product $f\diamond g$.  The equations are then solved as formal chaos series, analogous to formal power series solutions of differential equations.  In some cases, for example if used to define the non-linearity in the KPZ equation, this leads to spurious results; \cite{MR1743612} showed the resulting KPZ object has 1:4:7 scaling instead of the physical 1:2:3 scaling.
In addition, unless such series converge in $L^2$ - which they generally do not - they cannot even be identified with random variables.  This unfortunately led to the impression that all such  formal series solutions were  non-physical.  

However, the Wick product $f\diamond \xi$ does correspond to the Skorokhod integral when $\xi$ is white noise; the trouble is when it is used in other nonlinearities, such as $\partial_x h\diamond \partial_x h$ in the KPZ equation.  The Skorokhod integral is a remarkable discovery.  Although it is in principle non-local, it reproduces the It\^o integral in $1$ and $1+1$ dimensions (see Remarks \ref{r:Ito1} and \ref{r:Ito1+1}). In $2d$
the only manifestation of the non-locality is a reweighting by the renormalized self-intersection local time (as demonstrated in this article), making it a natural polymer model.  

The 2d Skorokhod multiplicative stochastic heat equation was solved in the  formal chaos series sense in \cite{MR1257000}, \cite{MR1451997} see also \cite{MR2571742}.  The first reference gives
a formal Feynman-Kac formula for the solution (see also \cite{MR2778803}, \cite{MR2962086}  for concrete result with other Gaussian noises).  Ex post facto, we may identify the formal expressions with concrete ones obtained in this article only for small times $t<t_c$ when the chaos series converges in $L^2$; otherwise, their meaning is unclear.  As far as we are aware, the connection to Gaussian multiplicative chaos, the representation as a Feynman-Kac formula weighted by the renormalized self-intersection local time, and the $L^1$ theory were not anticipated.

\section{Intersection local time }\label{intersectionlocaltime}

Our first goal is to give a novel short definition of mutual intersection local time for independent random measures.

\subsection{Intersection local time for random measures}
\label{ss:intersectionlocaltime}
Let $X,Y$ be random variables defined on a probability space $(\Omega, \mathcal G, Q)$ taking values on the space of finite measures on a measure space $(R,\mathcal B, \mu)$. In general, we will write $\rho_X$ for the density of the deterministic measure $EX$ with respect to $\mu$ if it exists. 

Consider the product  of two copies of $(\Omega,\mathcal G, Q)$. For a random variable $U$ defined on $\Omega$ and $(\omega,\omega')\in \Omega^2$ we write $U=U(\omega)$ to define $U$ on $\Omega^2$, and write $U'=U(\omega')$ to define an independent copy. 

We would like to define the mutual intersection local time $\alpha=\alpha(X,Y')$  as the random inner product 
$$
\alpha{=}\int \frac{dX}{d\mu} \frac{dY'}{d\mu} d\mu, \qquad \mbox{(heuristic)}
$$
a random variable on $\Omega^2$, even when $X,Y'$ do not have densities. When they do, for any nonnegative bounded random variables $F,G$ on $\Omega$ we have
$$
E_\Qt[FG'\alpha]{=}\int E_{\Qt}[F\frac{dX}{d\mu} G'\frac{dY'}{d\mu}] d\mu=\langle \rho_{FX},\rho_{GY}\rangle_{\mu},$$
where for $f,g\in L^2(\mu)$ as $\langle f,g\rangle_\mu$ is the inner product. 

This gives rise to a definition.

\begin{definition}\label{d:milt}
\ \\ Assume that $\rho_X$ and $\rho_Y$ exists. If there exists $\alpha\in L^1(\Omega^2,\sigma(X,Y'), \Qt)$  so that for all  bounded measurable  $F,G\ge 0$ on $\Omega$ we have
$$
E_{\Qt}[FG'\alpha]=\langle \rho_{FX},\rho_{GY}\rangle_\mu <\infty,
$$
then we call $\alpha=\alpha(X,Y')$ the {\bf mutual intersection local time} of $X,Y'$. Recall that given two measures $Z$ and $W$, $Z\ge W$ iff $Z-W\ge 0$.
\end{definition}

\begin{proposition}\label{p:milt-unique}
\begin{enumerate}
    \item If the intersection local time exists, then
it is nonnegative and unique.
\item If $X\le\bar X$, $Y\le \bar Y$ and $\bar \alpha=\alpha(\bar X,\bar Y')$ exists, then $\alpha=\alpha(X,Y')$ exists and $\alpha\le\bar \alpha$.
\end{enumerate}
\end{proposition}
\begin{proof}
Let $\alpha, \hat \alpha$ be two versions of the mutual intersection local time. Consider the subset $\mathcal S\subset\sigma(\mathcal G^2)$ on which both random variables have the same expectation, and this expectation is nonnegative. Then $\mathcal S$ is a  $\pi$-system  containing a $\lambda$ system  of  $\sigma(\mathcal G^2)$ given by the product sets, since 
$$
E[\one_{A\times B}\,\alpha]=E[\one_{A\times B}\,\hat \alpha]=\int \rho_{\one_AX}\rho_{\one_BY}d\mu \ge 0.
$$
Thus by the $\pi-\lambda$ theorem,  $\sigma(\mathcal G^2)=\mathcal S$. Then   $\{\alpha<0\},  \{\alpha<\hat \alpha\} \in \mathcal S$, so 
$$
E[\one_{\alpha<0}\,\alpha]\ge 0, \qquad E[\one_{\alpha<\hat \alpha}\alpha]=E[\one_{\alpha<\hat \alpha}\hat \alpha]
$$
so $E\alpha^-=0$, and since  $\alpha\in L^1$, we get $E(\hat \alpha-\alpha)^+=0$. Similarly 
$E(\hat \alpha-\alpha)^-=0$, and so $\hat \alpha=\alpha\ge 0$ a.s. 

For 2, note that by the Caratheodory extension theorem the equality $\mathcal A(A\times B):=\int \rho_{\one_AX}\rho_{\one_BY}d\mu$ extends from the algebra generated by the measurable rectangles to to a measure $\mathcal A$ on $\Omega^2$. Then $\bar \alpha d\mu^{\otimes 2}$ dominates $\mathcal A$ on product sets. So  $\mathcal A$ has a density in $L^1$, which fits the definition of $\alpha$.
\end{proof}

Definition \ref{d:milt} does not give a construction of intersection local time, nor an easily verifiable condition for when it should exist. 

Recall the notation $X^{\otimes 2}=X\otimes X$ is the $Q$-random product measure on $R\times R$; it is the product of the same random measure, not independent copies. As a special case of Proposition \ref{p:rhok}, we show the following.

\begin{corollary}\label{c:alpha-exists}
If the deterministic measure $E_Q[X^{\otimes 2}]$ has a square integrable density with respect to $\mu^{\otimes2}$, then the mutual intersection local time $\alpha(X,X')$ exists. 
\end{corollary}

Random measures defined on $\Omega$ form a cone in a vector space through $(bX+Y)(A) = b X(A) + Y(A)$ for $b\ge 0$.  When they are occupation measures of random functions,
this has nothing to do with adding function values $bX(t)+Y(t)$ in $\mathbb R^d$. The following are immediate from the definition. 

\begin{lemma}\label{l:bilinear} If  $d\ge 0$ is a constant and $\alpha(X,Y')$, $\alpha(X,Z')$ exist, then the following quantities exist and satisfy 
$$
\alpha(cX,Y')=c\alpha(X,Y'),
\qquad \alpha(X,Y'+Z')=\alpha(X,Y')+\alpha(X,Z')
$$
and $\alpha(X,Y')(\omega_1,\omega_2)=\alpha(Y,X')(\omega_2,\omega_1).$ 
For all $Q$-random bounded scalars $C, D\ge 0$, the following random variable on $\Omega^2$  exists and satisfies
\begin{equation}\label{e:bilinearCD}
\alpha(CX,D'Y')=CD'\alpha(X,Y').
\end{equation}
\end{lemma}

Corollary \ref{c:alpha-exists} is based on a coordinate representation of the random measure $X$. Consider  an orthonormal basis $\basis_i$, of $L^2(R,\mathcal B, \mu)$ consisting of bounded functions. Then with 
$$
m_j=m_{X,j}=\int \basis_jdX, \qquad X_n=\sum_{j=1}^n m_j \basis_j,
$$
we have a signed approximation $X_n$ of the random measure $X$. Similarly, let $w_j$ be the coordinates of $Y$, and let  
\begin{equation}\label{ltdef}
\alpha_n(X,Y')=\int X_nY'_n d\mu=\sum_{j=1}^n m_{j}w_{j
}'.
\end{equation}

We are now in the setting of Section \ref{ss:alpha-new}. Indeed, $\alpha_n(X,Y')$ as defined here, equals $\alpha_n(m,w')$ of \eqref{e:alpha-new}. So in order for $\alpha_n$ to have a limit in $L^k(\Qt)$, $k\in 2\mathbb N$, we just need $\rho_{k,m},\rho_{k,w}\in L^2(\mathbb N^k)$. But $$
\rho_{k,m}(i_1,\ldots,i_k)= E_Q[ m_{i_1}\cdots m_{i_k}],$$ the coefficient of the basis vector $\basis_{i_1}\otimes\cdots \otimes \basis_{i_k}$ in the orthonormal basis representation of $\rho_{k,X}$, the density of $E_QX^{\otimes k}$. Thus we have
\begin{equation}\label{e:rhom-rhox}
\langle \rho_{k,m},\rho_{k,w}\rangle_{\mathbb N^k} = \langle \rho_{k,X},\rho_{k,Y}\rangle_{\mu^{\otimes k}}.
\end{equation}
We will use the  results of Section \ref{ss:alpha-new} to show that $\alpha_n$ converges, and we will use the following general lemma to identify the limit. 

\begin{lemma}\label{l:alpha-L1}
If $\rho_X$, $\rho_Y\in L^2(R)$ and $\alpha_n(X,Y')$ is precompact in $L^1(\Qt)$, then the mutual intersection local time $\alpha(X,Y')$ exists and $\alpha_n(X,Y')\to \alpha(X,Y')$ in $L^1(\Qt)$.
\end{lemma}

\begin{proof}
Let $\alpha_\infty$ be a limit point of $\alpha_n$ along a subsequence $\mathcal N$.
Let $F,G\ge 0$ be bounded random variables on $\Omega$. 
By a direct computation, we have 
$$E_{\Qt}[FG'\alpha_n]=\sum_{i=1}^n  \langle \rho_{FX}, e_i\rangle_{\mu} \langle \rho_{GY}, e_i\rangle_{\mu} 
$$
Since $\rho_{FX},\rho_{GY}\in L^2(R)$, and  $FG'\alpha_n\to FG'\alpha_{\infty}$ in $L^1$ along $\mathcal N$, taking limits we conclude   
\[E_{\Qt}[FG'\alpha_\infty]=\langle \rho_{FX}, \rho_{G Y}\rangle_\mu. 
\qedhere \]
\end{proof}

\begin{proposition}\label{p:rhok} Let $k_0\in 2\mathbb N$. If $\rho_{k_0,X},\rho_{k_0,Y}\in L^2(\mu^{\otimes k_0})$, then for all integers $1\le k \le k_0$,
\begin{enumerate}  \item  \label{p:rhok-limit} The intersection local time $\alpha(X,Y')$ exists and
  $$\alpha_n\to \alpha(X,Y')\quad \mbox{in}\quad  L^k(\Qt).$$ 
  \item 
  $
      \|\alpha\|^k_{L^k(\Qt)}=\langle \rho_{k,X},\rho_{k,Y}\rangle_{\mu^{\otimes k}}
  $.
  \item 
     $E[\alpha(X,Y')^k]^2\le E[\alpha(X,X')^k]E[\alpha(Y,Y')^k]$ \label{p:alpha-misc.CS}.
 \end{enumerate}
\end{proposition}
\begin{proof}  This is a direct consequence of \eqref{e:rhom-rhox}, Proposition \ref{p:rhok-m} and Lemma \ref{l:alpha-L1}.
\end{proof}

In classical settings, with $R=\mathbb R^d$ and Lebesgue measure $\mu$, intersection local time is often defined as the value  at zero of the continuous density of the 
convolution $X*\tilde Y'$ where $\tilde Y=Y(-\cdot)$. The two notions are essentially the same, as the next lemma shows. 
\begin{lemma}\label{l:classical-milt}
Let $X,Y$ be random measures on the torus $R=[-\ell,\ell]^d$ or on $R=\mathbb R^d$. If $X*\tilde Y'$ has continuous density $\alpha_*(X,Y')$ at zero a.s.\;and $\rho_{2,X},\rho_{2,Y}\in L^2$ then $\alpha(X,Y')$ exists and $\Qt$-a.s. equals $\alpha_*(X,Y')$. 
\end{lemma}
\begin{proof}
For the torus, we may assume $\ell=\pi$. Let $\basis_j$ be the Fourier modes ordered by frequency. For $n=n_j=(2j+1)^d$, the projection to the span of $\basis_1, \ldots, \basis_n$ commutes with translations, so it is a convolution by a function $\varphi_n$. More explicitly,  
$$
\varphi_n= D_n^{\otimes d}, \qquad \mbox{and} \quad D_n(x)=\frac{1}{2\pi}\sum_{k=-j}^j e^{ik  x}=\frac{\sin((k+1/2)x)}{2\pi \sin(x/2)}
$$
is the Dirichlet kernel. Hence $
 X_n\mu =X*\varphi_n
$, and
$$
\alpha_n(X, Y')=\int (X * \varphi_n)(Y' * \varphi_n)=(X* \varphi_n * \tilde  \varphi_n *\tilde Y')(0)= (\varphi_n*X*\tilde Y')(0)
$$
since $\varphi_n$ is symmetric and a projection, so $\varphi_n*\tilde \varphi_n=\varphi_n$. Since $\varphi_n$ approximates  $\delta_0$,  we see that $\alpha_{n_i} \to \alpha_*$, which then equals $\alpha$ by Proposition \ref{p:rhok}.

For the case of $\mathbb R^d$, apply the  torus argument on  $[-2\ell,2\ell]^d$ to $X\one_{[-\ell,\ell]^d}$,  $Y\one_{[-\ell,\ell]^d}$ and use dominated convergence to get the result. 
\end{proof}

\subsection{Moment bounds}\label{ss:moments}

The solution and convergence of the planar Wick-ordered  heat equation requires estimates of the exponential moments of the mutual intersection local time of planar Brownian motion.

 Proposition  \ref{p:rhok} shows that for random measures $X$, the mutual intersection local time $\alpha(X,X')$ exists if for $k=2$ the measure $E_Q[X^{\otimes k}]$ has a density $\rho_k$ in $L^2.$ When $\alpha(X,X')$ exists, its $k$th moment is given by $\|\rho_k\|^2$. 

The following bound on the moments holds for any continuous time process with independent stationary increments (L\'evy process)
$B$ with joint densities with respect to some fixed measure.  In particular, it holds for Brownian motion with drift on $R=\mathbb R^d$.  For Brownian motion without drift, related bounds can be found in \cite{LeGall}. But the method there does not readily generalize to Brownian motion with drift.
The statement involves a parameter $r$, which will be chosen carefully later.

\begin{proposition}\label{p:moments}
Let $k$ be a positive integer, $r>0$ and let $X(A)=\int_0^1  \one_A(B(t))\,dt $ be the occupation measure of a stochastic process $B$ in $R=\mathbb R^d$ up to time $1$. For a time vector $\mathbf s=(s_1,\ldots, s_k)$ let $p(\mathbf x, \mathbf s)$ be the  density of the random vector $(B(s_1),\ldots, B(s_k))$ at $\mathbf x=(x_1,\ldots,x_k)$; assume this density exists for almost all vectors $\mathbf s$. Then the random $k$-fold product $X^{\otimes k}$ has density
\begin{equation}\label{e:rho-p}
\rho_k(\mathbf x)=
\int_{\mathbf s\in [0,1]^k}p(\mathbf x,\mathbf s)d\mathbf s
\end{equation}
 on $R^k$.
Assume further that $B$ is a Levy process. 
Let $\tau$ be an independent exponential rate $r$ random variable, and let $\varphi_r(\cdot)$ be the density of $B(\tau)$ on $R$. Then 
\begin{equation}\label{miltmom}
\|\rho_k\|^2_2\le k!^2e^{2r}\|\varphi_r/r\|^{2k}_2.
\end{equation}
\end{proposition}

\begin{proof}
For any measurable set $A\subset  R^k$ we have
\begin{align*}
E[X^{\otimes k}](A)&=E\int_{\mathbf s\in [0,1]^k}\one((B(s_1),\ldots B(s_k)) \in A)\\&=\int_{\mathbf s\in [0,1]^k}P((B(s_1),\ldots, B(s_k) \in A))
\end{align*}
Expressing the probability as an integral and using Fubini we get
$$
E[X^{\otimes k}](A)=\int_{\mathbf x\in A} \int_{\mathbf s\in [0,1]^k}p(\mathbf x,\mathbf s),
$$
which implies the density formula \eqref{e:rho-p}.
Let $[0,1]^{\uparrow k}$ denote the subset of vectors in $\mathbf s\in [0,1]^k$ with $s_1<\cdots< s_k$.
\begin{align*}
\|\rho_k\|^2_2&=\int_{R^k}\int_{\mathbf s,\mathbf t\in [0,1]^k}p(\mathbf x,\mathbf s)p(\mathbf x,\mathbf t)\\&=\sum_{\sigma,\eta} \int_{ R^k}\int_{\mathbf s\in [0,1]^{\uparrow k}}p(\mathbf x,\sigma\mathbf s)\int_{\mathbf t\in [0,1]^{\uparrow k}}p(\mathbf x,\eta \mathbf t)
\end{align*}
where the sum is over all permutations $\sigma$ and $\eta$.
 By Cauchy-Schwarz applied to the $\mathbf x$-integral, the summand is at most
\begin{equation}\label{e:CS}
\int_{R^k}\left(\int_{\mathbf s\in [0,1]^{\uparrow k}}p(\mathbf x,\mathbf \sigma \mathbf s)\right)^2
\end{equation}
since $p(\sigma \mathbf s,\mathbf x)=p(\mathbf s,\sigma^{-1}\mathbf x)$, changing variables $\mathbf x\mapsto\sigma \mathbf x$  shows that
we can drop the $\sigma$ from \eqref{e:CS}. We now bound each factor the same way. Rewrite \eqref{e:CS} as 
\begin{equation}\label{e:k!2}
\int_{ R^k}\int_{\mathbf s,\mathbf t\in [0,1]^{\uparrow k}}p(\mathbf x,\mathbf s)p(\mathbf x,\mathbf t).
\end{equation}
Let
\begin{equation}\notag
a(s,t)=\int_{ R^k}\int_{\substack{(s_1,\ldots,{s_{k-1}})\in [0,s]^{\uparrow (k-1)}\\(t_1,\ldots,{t_{k-1}})\in [0,t]^{\uparrow (k-1)}}}p(\mathbf x,\mathbf s)p(\mathbf x,\mathbf t),
\end{equation}
i.e.\ the integral on the right hand side of \eqref{e:k!2} without integrating over the last two variables $s,t$. Since
$1\le e^{2r}e^{-rs-rt}$  for $s,t \in [0,1]$,
\begin{align*}
\int_{\mathbb R^k}\int_{\mathbf s,\mathbf t\in [0,1]^{\uparrow k}}p(\mathbf x,\mathbf s)p(\mathbf x,\mathbf t)&=\int_0^1\int_0^1 a(s,t)dsdt\\&\le e^{2r}\int_0^\infty\int_0^\infty e^{-rs-rt}a(s,t)dsdt.
\end{align*}
Change variables to the increments $\mathbf u$, $\mathbf v$ of $\mathbf s,\mathbf t$ and the increments $\mathbf y$ of $\mathbf x$ to get
\begin{align*}
\int_0^\infty&\int_0^\infty r^{2k}e^{-rs-rt}a(s,t)dsdt
\\&=\int_{ R^k}\int_{\mathbf u,\mathbf v\in [0,\infty)^{k}} r^{2k}e^{-r\sum_{i=1}^k u_i+v_i}\prod_{i=1}^k q(y_i,u_i)q(y_i,v_i)= \Big(\int_{R} \varphi_r(x)^2\Big)^k
\end{align*}
where we used the independent increment property of $B$, and  $q(x,u)$ is the density of $B(u)-B(0)$ at $x$ and that $r e^{-ru}$ is the density of the exponential random variable $\tau$.
\end{proof}
Next, we apply this to planar Brownian motion, with and without drift. We need the following fact. 
\begin{lemma}\label{l:Brown-comp}
Let $B$ be planar Brownian motion with covariance $I$ and drift $(0,N)$, and let $\varphi_{N,r}$
be the density
 of $B$ at   independent rate $r$ exponential
 random time. 
Then we have
\begin{equation}\label{85}\|\varphi_{N,r}\|^2\;\begin{cases}=\frac{r}{2\pi}, & \; N=0;\\ \le \frac{r}{2\pi}\wedge \frac{r^{3/2}}{\sqrt{8}N}, &  \;N \ge 0\end{cases}\end{equation}
 \end{lemma}
 \begin{proof}
Let $\tau$ be the exponential variable, and let prime denote independent copies. Then $\|\varphi_{N,r}\|^2$ is the density of $B_{\tau}-B'_{\tau'}$ at zero. Conditionally on $\tau$ and $\tau'$, $B_{\tau}-B'_{\tau'}$ is a Gaussian with mean $(0,(\tau-\tau')N)$ and  covariance matrix $(\tau+\tau')I$. Therefore

\begin{equation}\notag
\|\varphi_{N,r}\|^2=\int_0^\infty
\int_0^\infty
\frac{r^2e^{-r(t+t')}}{2\pi (t +t')} \exp \Big\{ -\frac{(t-t')^2 N^2}{2(t+t')}\Big\}\, dt\,dt'.
\end{equation}
Evaluate at $0$ to get $\frac{r}{2\pi}$ and note that it is decreasing in $N$ to obtain the first part of \eqref{85}.
To get the second statement, change variables $u=t+t'$, $v= N(t-t')$. The Jacobian is $2N$, and we get
\begin{equation*}
\frac {r^2}{4\pi N}\int_0^\infty \frac{e^{-ru}}{u}
\int_{-Nu}^{Nu}
e^{  -\frac{v^2}{2u}}\,dv\,du \le
\frac{r^2}{4\pi N}\int_0^\infty \frac{e^{-ru}}{u}
\int_{-\infty}^{\infty}
e^{  -\frac{v^2}{2u}}\,dvdu= \frac{r^2}{\sqrt{8r}N}. 
\end{equation*}
\end{proof}

\subsection{Brownian intersection local time}
\label{ss:ilsbm}
Next we consider the special case of $R=\mathbb{R}^2$, $\mathcal{B}$ the Borel sets, $\mu=$Lebesgue measure. 

\begin{proposition}\label{p:BM-alpha-moments}
Let $B$ be planar Brownian motion with drift $\mu$ run until time $t$. 
Then for every $k\ge 1$, $\alpha(B,B')$ exists as a limit in $L^k$, and we have
\begin{equation}\label{e:moment2}
E[\alpha(B,B')^k]\le   e k! \sqrt{k}(te/(2\pi))^{k}.
\end{equation}
\end{proposition}
Note that in the driftless case  \cite{LeGall}
shows $
a_1^k k!\le E\alpha(B,B')^k\le   a_2^k k!$,
so \eqref{e:moment2} is not far from the truth.
\begin{proof} 
By scaling, it suffices to show this for $t=1$. By Proposition  \ref{p:moments} and  Lemma \ref{l:Brown-comp}, \begin{equation}\label{e:moment}
\|\rho_k\|^2\le  k!^2e^{2r}\|\varphi_r/r\|^{2k}\le \frac{k!^2e^{2r}}{(2\pi r)^{k}}.
\end{equation}
Since this is finite for all $k$,  Proposition \ref{p:rhok} implies that the intersection local time $\alpha$ is the limit of $\alpha_n$ (defined in (\ref{ltdef})) in $L^k$ and $\|\rho_k\|^2=E\alpha^k$.
Setting $r=k$ and using the inequality version of Stirling's formula $k!\le e k^{1/2+k}e^{-k}$ for $k\ge 1$ gives \eqref{e:moment2}.
\end{proof}

\begin{proposition}[From BM to BB]\label{p:B-R}
Let $R$ be a Brownian bridge with covariance $\Sigma$ from $p$ to $p+\mu s$ in $\mathbb R^d$ in time $s$. Let $a \in (0,1)$ and let $F(R_{[0,a s]})$ be a nonnegative test function depending on the initial part  of $R$. Let $B$ be a Brownian motion with covariance $\Sigma$ and drift $\mu$ started at $p$. Then  
$$EF(R_{[0,a s]})\le (1-a)^{-d/2}EF(B_{[0,\alpha s]}).$$
\end{proposition}
By time reversal, a similar claim holds for the last $a$ portion of $R$.

\begin{proof}
The conditional distributions  $R_{[0,as]}$ given $R(as)$, and $B_{[0,a s]}$ given $B(as)$ are the same: both are Brownian bridges with the given endpoints. So it suffices to bound the Radon-Nikodym derivative of $R(as)$ with respect to $B(as)$. They are Gaussians with the same mean and covariance  $as(1-a)\Sigma$ and $as\Sigma$, respectively. The maximal  value of the Radon-Nikodym derivative, as a function of $s$ and $\Sigma$  does not depend on a common scale or shift, and it is $(1-a)^{-d/2}$ when $s=1$, $\Sigma=I$. 
\end{proof}

We are now ready to bound intersection local time for Brownian bridges.

\begin{proposition}\label{p:alpha-plane}  Consider planar Brownian bridge $R$ from $0$ at time $0$ to $\mu t$ at time $t$. 
Then for all $k\in \mathbb N$, 
the intersection local time $\alpha(R,R')$ exists as a limit in $L^k$ and
$$
E\alpha^k\le  4e k! \sqrt{k}(te/\pi)^{k}.
$$
\end{proposition}

\begin{proof} 
Let $R_1$, $R_2$ denote the bridge restricted to $[0,t/2]$ and $[t/2,t]$, respectively.  Let $Q$ denote the  law of $R_1$. Then $Q^{\otimes 2}$ is absolutely continuous 
with respect to $Q_{BM}^{\otimes 2}$ of Brownian motion with drift $\mu$ run until time $t$ with derivative bounded by $4$, see Proposition \ref{p:B-R}. In particular, since $\alpha_n$ converges in $L^k(Q_{BM}^{\otimes 2})$, it also converges in $L^k(Q^{\otimes 2})$, and with $\alpha=\alpha(R_1,R_1')$
$$
E_{Q^{\otimes 2}}\alpha^k\le 4E{Q_{BM}^{\otimes 2}} \alpha^k\le 4e k! \sqrt{k}(te/(4\pi))^{k}.
$$

The same bound applies to $R_2$, and Lemma \ref{l:bilinear} implies that $\alpha(R,R')=\sum_{i,j=1}^2\alpha(R_i,R'_j)$ exists. Expanding the $k$-th power, and using Proposition \ref{p:rhok}.\ref{p:alpha-misc.CS}, H\"older's inequality and symmetry, we get $$E\alpha(R,R')^k\le 4^k E\alpha(R_1,R_1')^k,$$ as claimed. 
\end{proof}

\begin{remark}\label{rem:miltdist}
The mutual intersection local time  is classically defined as the value at $0$ of the continuous density of the random measure
\begin{equation}\label{rosendef} A\mapsto \int_0^t\int_0^t \mathbf{1}_A( B(s)-B'(s')) dsds',\qquad A\in \mathscr{B}(\mathbb{R}^2),
\end{equation}
see \cite*{geman1984local}. By Lemma \ref{l:classical-milt} this coincides with our definition of $\alpha$ as a random variable on $\Qt$.
\end{remark}

Finally, we consider the case where $X$ is the occupation measure up to time $t$ of $(b(s),s)$ where $b(s)$ is a one-dimensional Brownian bridge.  The computations are the same and we can identify the limiting $\alpha$ again by its moments. Let $q(x,t)$ be the density at $x$ of the centered Gaussian with variance $t$.

\begin{proposition}\label{p:alpha-one} Let  $X$ be the occupation measure of $(b(s),s)$ where $b(s)$ is a one-dimensional Brownian bridge from $y_0=(0,0)$ to $y_*=(x,t)$. Then for every $k\in \mathbb N$,
\begin{enumerate}
\item The measure 
$E_Q[X^{\otimes k}]$ is supported on $(\mathbb R\times [0,t])^k$ and has density $$\rho_{k,X}(\mathbf y)=\frac{1}{q(y_*)}\prod_{i=1}^{k+1} q(y_{(i)}-y_{(i-1)}),$$ where $y_{k+1}=y_*$ and $y_{(i)}$ are the order statistics of $y$ ordered increasingly by the second (time) coordinate. 
\item With $s_0=0, s_{k+1}=t$ we have 
\begin{align} \notag
\|\rho_{k,X}\|_{L^2((\mathbb{R}^2)^k)}^2&
=
\frac{2^{-k/2} k!}{q(0,t)}\int_{0<s_1<\cdots<s_k<t}\prod_{i=1}^{k+1}q(0,s_i-s_{i-1})
\\&
\!\!\!\!\!\!\!=
\frac{\sqrt{t}k!}{(4\pi)^{k/2}} \int_{0<s_1<\cdots<s_k<t}\prod_{i=1}^{k+1}\frac{1}{\sqrt{s_i-s_{i-1}}}\nonumber  \\& \notag
=
\frac{\sqrt{\pi}t^{k/2}k!}{2^k\Gamma(\tfrac12(k+1))}\le c\,\frac{(tk/(2e))^{k/2}}{\sqrt{k}}.
\notag\end{align} 
\item The intersection local time $\alpha$ exists, $\alpha=\lim \alpha_n$ in $L^k(Q^{\otimes 2})$  and we have $E_\Qt[ \alpha^k] =  \|\rho_k\|_{L^2((\mathbb{R}^2)^k)}^2$.\label{p:alpha-one-exists}

\item For all $\gamma>0$ we have $E \cosh(\gamma \alpha)<\infty$.
\end{enumerate}
\end{proposition}
\begin{proof}  Fubini gives {\sl 1}. 
To get {\sl 2}, note that for $s$ fixed, the $x$ integral gives the squared $L^2$-norm of the density of $(b(s_1), \ldots, b(s_k))$. The same computation also gives the density of the self-convolution at zero, which is exactly $2^{-k/2}$ times the original density at $0$ for any Gaussian vector. Up to scaling, the integral  computes the normalizing constant of a Dirichlet$(1/2,\ldots,1/2)$ distribution. 
Proposition \ref{p:rhok} implies {\sl 3}. Writing the Taylor series of $\cosh$, using Fubini and the bounds {\sl 2} gives {\sl 4}.
\end{proof}
 \begin{remark}\label{r:milt1+1}
  As in Remark \ref{rem:miltdist}, by Lemma \ref{l:classical-milt},  $\alpha$ coincides with the classical definition, the value at $0$ of the continuous density of the random measure
\begin{equation}\label{rosendef2} A\mapsto \int_0^t \mathbf{1}_A( b(s)-b'(s)) ds,\qquad A\in \mathscr{B}(\mathbb{R}),
\end{equation}
\end{remark}

Let $L$ be a $d\times d$ matrix with $\det L\not=0$. Then $X\circ 
L(A)$ and $X\circ L^{-1}(A)$ define the pullback and pushforward of the measure $X$ by the linear transformation $L$. 

The definition implies that $\alpha(X,X')$ scales like $\|\rho_X\|_2^2$. More precisely, we have the following. The proof is left to the reader. 

\begin{lemma}[Scaling $\alpha$]\label{l:scaling}
Let $X,Y$ be a random measures on $\mathbb R^d$, so that $\alpha(X,Y')$ exists.
Then, a.s.\ the mean densities satisfy $\langle \rho_
{X\circ L},\rho_
{Y\circ L}\rangle_\mu =|\!\det L |\,\langle \rho_
{X},\rho_
{Y}\rangle_\mu$ and 
$$
\alpha(X\circ L, Y'\circ L )=|\!\det(L)|\,\alpha(X, Y' ).
$$
Now let $t>0$, $(B(s),s\in[0,1])$ be a process on $\mathbb R^d$ and  let $\hat B(s)=LB(s/t)$ on $[0,t]$. Then a.s.\ the occupation measures satisfy 
\begin{align*}\hat X(A)&=|\hat B^{-1}(A)|=t|B^{-1}L^{-1}(A)|=tX\circ L^{-1}(A), \\  
\alpha(\hat X, \hat X')&=\frac{t }{|\!\det L| }\,\alpha(X,X').
\end{align*}
\end{lemma}

\begin{proposition}\label{convofalphas}  Let $B^\nu$, be
a planar
Brownian bridge on the time interval $[0,1]$ from $(0,0)$ to $(y,1)$  with
covariance $\operatorname{diag}(1,\nu^2)$.  Then for every $y$, 
\begin{equation}
 \phi(\nu,y):=   E \alpha(B^{\nu,y}, ( B^{\nu,y})')
\end{equation}
is continuous at $\nu=0$.  Furthermore,
\begin{equation}
    \phi(0,y) =  E[ \alpha(b^y, (b^y)')^2 ]
\end{equation}
where $b^y$ is a Brownian bridge from $0$ to $y$ in time $1$. Moreover, $\phi(0,y)=\phi(0,0)$. 
\end{proposition}

\begin{proof} 
By Propositions \ref{p:rhok}, \ref{p:moments},  and \ref{p:alpha-plane},  for $\nu>0$ we have the moment formula 
$
\phi(\nu,y)=\|\rho_{2,B^{\nu,y}}\|_2^2
$
with $\rho_2$ as in \eqref{e:rho-p}. 
For $\nu>0$, it can be written using Fourier transform as 
\begin{equation}\label{FTexpression} \phi(\nu,y) = \|\rho_{2,B^{\nu,y}}\|_2^2= \int_{\mathbb R^4} \prod_{j,k=1}^2\frac{dz^j_k}{2\pi} \int_{[0,1]^4}\prod_{j,k=1}^2 dt^j_k
    \exp\{\Psi(z,t)\},\end{equation} where 
\begin{align}
\nonumber  \Psi =  -\frac{\nu^2}2 {\rm Var}(\sum_{j=1}^2 z_2^j(b(t^1_j)-{b'}(t^2_j) )&- \frac12 {\rm Var}(\sum_{j=1}^2 z_1^j (b(t^1_j)-{b'}(t^2_j) )  \\  &+ i\sum_{j=1}^2(z_2^j+ yz_1^j) (t^1_j-t^2_j).
\notag
\end{align}
(see  \cite{MR556414} p. 43).  Dropping the imaginary part of  $\Psi$ just corresponds to the same problem with the bridges going from $(0,0)$ to $(0,0)$. Since for such bridges $E\alpha^2<\infty$ by Proposition \ref{p:alpha-plane}, we see that \eqref{FTexpression} is absolutely integrable for 
  $\nu>0$. 
So we can use Fubini and perform the $z_2^j$ integrations.  We have
\begin{align*}
  {\rm Var}&(\sum_{j=1}^2 z_2^j(b(t^1_j)-{b'}(t^2_j) ) = {z}_2^T
  C{z}_2, \\ &C= \sum_{j=1}^2{\begin{pmatrix}  t_1^j(1-t_1^j)  &
   (t_1^j \wedge t_2^j)(1- t_1^j \vee t_2^j ) \\  (t_1^j \wedge t_2^j)(1- t_1^j \vee t_2^j ) &
    t_2^j(1-t_2^j) \end{pmatrix}},
\end{align*}
so the $z_2^j$ integrations give
\begin{align*}
    \int_{\mathbb R^2} \prod_{j=1}^2\frac{dz^j_1}{2\pi} \int_{[0,1]^4}\prod_{j=1}^2 &dt_j^1 dt_j^2\frac{ e^{ -\frac{\nu^{-2}}2  (t^1-t^2)^T C^{-1} (t^1-t^2)}
    }{2\pi \nu^{-2} \sqrt{\det C}}\\&\times
    e^{- \frac{1}2 {\rm Var}(\sum_{j=1}^2 z_1^j (b(t^1_j)-{b'}(t^2_j) )   + iy\sum_{j=1}^2z_1^j (t^1_j-t^2_j)
    } .
\end{align*}
Let $\nu\to 0$  to get
$$
    \int_{\mathbb R^2} \prod_{j=1}^2\frac{dz^j_1}{2\pi} \int_{[0,1]^2}\prod_{j=1}^2 dt_j
    \exp\{- \frac12 {\rm Var}(\sum_{j=1}^2 z_1^j (b(t_j)-{b'}(t_j) )\} =\phi(0,y),$$
    which we identify as the Fourier transform expression for $\|\rho_{2,b}\|_2^2$ given in Proposition \ref{p:alpha-one}. The last claim  follows from Lemma \ref{l:scaling}.
\end{proof}

\section{Exponential moments off a small set for planar Brownian motion}\label{exponentialmoments}

In this section we use the moment bounds of Section \ref{ss:moments} to bound the exponential moments of  intersection local times off some small exceptional sets. 

 For short times, exponential moment bounds are established as a consequence of usual moment bounds. For longer times, exponential moments blow up and we need to control the geometry of paths to find a high probability sets on which they are still finite. 

The main  theme in the section is to confine the pair of Brownian paths so that we are only likely to find intersections in  specific short time segments of each path. 

We need this in two different settings. The first is for fixed drift as the size of the exceptional set tends to zero, and the second is a uniform bound as the drift increases. The first is needed to define the solution of the Wick-ordered  heat equation for large times, the second is needed for the convergence to the KPZ equation.

In the fixed drift case, we need to understand that the Brownian paths are not concentrated on small sets. For this, we use a H\"older continuity and some control on the regularity of occupation measure, see \eqref{e:holder}.

In the increasing drift case, we use the effect of the drift to show that the two paths can only intersect at nearby times. The  argument uses three steps, handling increasing time scales.  Lemma \ref{l:drifted2} gives an estimate for  short times that gets better with the drift. Proposition \ref{p:tightness1} handles intermediate times, and finally Proposition \ref{p:tightness2} gets up to the time scale we need in the proof of the main theorem.

{\bf Notation for this section.} In this section, we only work on the probability spaces $(\Omega, \mathcal G, Q)$ inhabited by Brownian motions and bridges. For this reason, we briefly return to the Kolmogorov convention of suppressing the structure of the probability space. In particular, we will write $(E,P)$ instead of $(E_Q,Q)$ or even $(E_{Q^{\oplus 2}},Q^{\oplus 2})$. There is one exception: for events $A=A(B)$ given in terms of a Brownian path $B$ we write 
\begin{equation}\label{e:times-notation}
A\times A=A(B)\cap A(B')    
\end{equation}
whenever there is an independent $B'$.

\subsection{Exponential moments for a short time}

We start with establishing exponential moment bounds for short times. 

\begin{lemma}\label{l:drifted} Let $B_1,B_2'$ be independent planar Brownian motions on time interval $[0,t]$ with
 identity covariance matrix, possibly different drifts and started at possibly different locations. Then 
\begin{equation}\notag\label{e:arb-drift}
Ee^{\gamma \alpha(B_1,B_2')}\le 1+ \frac{2\pi \gamma t}{(2\pi/e-\gamma t)^2}, \qquad \gamma t \in [0,2\pi/e).
\end{equation}
For independent Brownian bridges $R_1,R_2'$ with variance 1, run until time $t$ but an arbitrary start and endpoint, 
\begin{align*}
Ee^{\gamma \alpha(R_1,R_2')}&\le \begin{cases}1+ \frac{4\pi\gamma t}{(\pi/e-\gamma t)^2}, \quad&\gamma t\in [0,\pi/e), \\  e^{10\gamma}, & \gamma t\in[0,1].
\end{cases}
\end{align*}
Moreover, there is a constant $c$
such that if the Brownian motions have a shared drift $(0,N), N\ge 0$ then  \begin{equation}\notag E [\alpha(B_1,B_2')^2] \le \frac{ct}{1+N^2t}, \quad
Ee^{\gamma \alpha(B_1,B_2')} \le  1+\frac{c\gamma t}{1+N\sqrt{t}}, \quad \gamma t\in [0,\pi/4].
\end{equation}
\end{lemma}

\begin{proof} By scaling, Lemma \ref{l:scaling}, we may assume $t=1$.
With prime denoting independent copies, by Proposition \ref{p:rhok}.\ref{p:alpha-misc.CS} we have 
$$
E\alpha(B_1,B_2')^k\le (E\alpha(B_1,B_1')^kE\alpha(B_2,B_2')^k)^{1/2}\le ek!\sqrt{k}(e/(2\pi))^k  
$$
Using Proposition \ref{p:BM-alpha-moments} and  $\sqrt{k}\le k$, we have
\begin{align*}
Ee^{\gamma \alpha(B_1,B_2')} &=\sum_{k=0}^\infty \frac{\gamma^k}{k!}E\alpha(B_1,B_2')^k\le 1+ e\sum_{k=1}^\infty {(e\gamma/(2\pi))^k}k\\ &=1+ \frac{2\pi \gamma}{(2\pi/e-\gamma)^2}.
\end{align*}
The bridge bound works the same way, just with the moment bounds of Proposition \ref{p:alpha-plane}; the exponential bounds the first bound. 
For the second bound, By Proposition  \ref{p:moments} with $k=2, r=1$ and Lemma \ref{l:Brown-comp},
\begin{equation}\label{e:secm}
E\alpha(B_1,B_1')^2\le \frac{c}{1+N^2}.
\end{equation}

For the third bound, for any $x\ge 0$ we have $e^x-1\le xe^x$. By Cauchy-Schwarz, for any nonnegative random variable $X$, we have $(Ee^X-1)^2\le EX^2Ee^{2X}$. So by what we have shown already,
\[
Ee^{\gamma \alpha}-1\le  (E[(\gamma\alpha)^2]Ee^{2\gamma\alpha})^{1/2}\le
c\gamma/(1+N).\qedhere\]
\end{proof}

\subsection{Fixed drift}

The goal of this section is to show Corollary \ref{c:large-t}: on sets of probability arbitrarily close to one, large exponential moments of the intersection local time exist. This is used in the  solution of the planar Wick-ordered  heat equation for large times.


We will control paths through two parameters. Fix $\kappa \in (0,1/2)$, and let $\bx(\eps)$ be the set of $\eps\times \eps$ closed boxes with corners in $\eps \mathbb Z^2$.
\begin{equation} \label{e:holder}\|B\|_{\kappa}=\sup_{0\le s<t\le 1}\frac{|B(t)-B(s)|}{(t-s)^\kappa}, \quad M(B,\eps,\delta)=\max_{K\in \operatorname{box}(\eps)}\sum_{i=0}^{\lfloor 1/\delta\rfloor } \one\left(B(\delta i)\in K\right).
\end{equation}
The first is the H\"older norm of the paths, the second measures the  regularity  of the occupation measure at a grid of times. When these two quantities are controlled, the two Brownian motions will not be able to spend enough time at a single place to ruin the  exponential moments  of intersection local times.
\begin{proposition}\label{p:nodrift} Let $B,B'$ be independent copies of  planar Brownian motion or Brownian bridge on the time interval $[0,1]$, variance $I$ and arbitrary drift or endpoints. Let $\eps, \delta, \gamma>0$.
Let $A$ be the event that 
$$\|B\|_{\kappa} <\eps/(5 \delta^\kappa)\qquad \mbox{  and that } \qquad M(B,\eps, \delta) \le 1/(6\sqrt{\gamma\delta}).$$
Then there exists a $c>0$ such that
\begin{align*}
E\exp \Big\{\one_{A\times A} \gamma \alpha(B,B') \Big\}  &\le E \exp\Big\{c\gamma \delta \sum_{i,i'} \one(|B(\delta i)-B'(\delta i')|<4\eps)\Big\}\\&\le \exp \{c\gamma \delta^{-1}\}.
\end{align*}
\end{proposition}

\begin{proof}
Let $S=\{(0,0),(0,1),(1,0),(1,1)\}$, let $\bx_{(0,0)}(\eps)$ be the set of boxes with lower left corner in $2\eps\mathbb Z^2$, and for $\sigma\in S$ let $\bx_\sigma(\eps)$ be the set of translates of boxes in  $\bx_{(0,0)}(\eps)$ by $\sigma\eps$. This partitions $\bx(\eps)$ into 4 subsets in a way that boxes in the same subset are separated.

Set $a^*_{i,i'}=\alpha(B_{[i,i+1]\delta},B'_{[i',i'+1]\delta}$), and consider the versions
\begin{align*}
a_{i,i'}= &a^*_{i,i'} \one \{|B(s)-B(i\delta)|\le \eps/5,s\in [i,i+1]\delta\} \\ &\quad \times \one \{|B'(s)-B'(i'\delta)|\le \eps/5,s\in [i',i'+1]\delta\}
\end{align*}
 For $K\in \bx(\eps)$ define the quantity,  asymmetric in $B$ and $B'$,
$$
Y_K=\sum_{i,i'} \one(B(i\delta)\in K) a_{i,i'}
$$
Then on $A\times A$, using only the modulus of continuity condition, we have 
$$
\alpha(B,B')=\sum_{\sigma \in S}\sum_{K\in \bx_\sigma(\eps)} Y_K.
$$
and so by Cauchy-Schwarz
\begin{equation}
\label{expl1}
E\exp \{\one_{A\times A}\gamma \alpha(B,B') \}\le \prod_{\sigma \in S} \Big(E\exp\big\{4\gamma \sum _{K\in
\bx_\sigma(\eps)} Y_K\big\}\Big)^{1/4}.
\end{equation}
Now condition on the sigma field $\mathcal G_\delta$ given by values of $B,B'$ at times in $\delta \mathbb Z$.

If $B(\delta i)\in K$ and $B(\delta j)\in K'$ for $K\not=K'\in \bx_{\sigma}(\eps)$ (with the same $\sigma$!) then by the modulus of continuity, on the event $A\times A$ the segments of $B$ on $\delta[i,i+1]$ and $\delta[j,j+1]$ cannot intersect the same segment of $B'$.

Thus, conditionally on $\mathcal G_\delta$, the collections of mutual intersection local times
$\mathcal A_K=\{a_{i,i'}: i,i'\in \mathbb Z, B({\delta}i)\in K\}$ are independent as $K$ changes. Thus
\begin{equation}
\label{expl2}
E\Big[\exp\big\{4\gamma \sum _{K\in
\bx_\sigma(\eps)} Y_K\big\}\vert \mathcal G_\delta \Big]=\prod_{K\in
\bx_\sigma(\eps)}E\Big[\exp\big\{4\gamma Y_K\big\}\vert \mathcal G_\delta \Big]
\end{equation}
Now let $L_K$ be the number of $i$ so that $B(\delta i) \in K$ and let $L^*_K$ be the number of $i'$ that $B'(\delta i')$ is in $K^*$, the union of 9 boxes in $\bx_{\eps}$ that intersect $K$.  Then $L_KL_K^*$ is an upper bound on the number of summands in $Y_K$. Using H\"older's inequality again,
\begin{equation}
\label{expl3}
E\Big[\exp\big\{4\gamma Y_K\big\}\vert \mathcal G_\delta \Big]\le \prod_{i,j:B(\delta i)\in K}E\Big[\exp\big\{4L_KL_K^*\gamma \alpha_{ij}\big\}\vert \mathcal G_\delta \Big]^{1/L_KL_K^*}
\end{equation}
On the set $A\times A$ we have  $4L_KL_K^*\gamma \delta\le 1$, (this is where we need the condition on the occupation measure). By Lemma  \ref{l:drifted}, the conditional expectation is bounded by
\begin{equation}
\label{expl4}
E\Big[\exp\big\{4L_KL_K^*\gamma \alpha_{ij}\big\}\vert \mathcal G_\delta \Big]\le \exp(cL_KL_K^*\gamma \delta).
\end{equation}
Combining (\ref{expl1}), (\ref{expl2}), (\ref{expl3}) and (\ref{expl4}), we get
$$
E\exp \{\one_{A\times A}\gamma \alpha(B,B')\}\le E\exp\Big\{ c\gamma\delta\sum_{K\in \bx(\eps)}  L_KL^*_K\Big\}
$$
Where we have
$$
\sum_{K\in \bx(\eps)} L_K L^*_K\le \sum_{i,i'} \one(|B(\delta i)-B'(\delta i')|<4\eps).
$$
The latter can be thought of as the mutual intersection local time of a pair of Gaussian random walks.
 \end{proof}

It remains to bound $M(B,\eps, \delta)$. We use a standard argument.  Throughout the rest of this section we will use the notation $P$ for the law of either a Brownian motion or a Brownian bridge with drift.

\begin{lemma}\label{l:occupation}
There is $c,c'>0$ so that the following holds. For planar Brownian motion with arbitrary drift and covariance $I$ on the interval $[0,1]$ and for all $0<\delta \le \eps^2 \le c$ and $a\ge 1$ we have
$$
P(M(B,\eps,\delta)>a)\le \frac{c'}{\delta}\exp\left\{-\frac{c\delta a}{\eps^2|\!\log \eps|}\right\}.
$$
The same result holds for any planar Brownian bridge with covariance $I$.
\end{lemma}
\begin{proof}
For planar Brownian $B$ motion with identity covariance arbitrary starting point and drift, the probability that it will not visit a fixed disk $D$ of radius $\eps$ between times $\eps^2$ and $1$ is bounded below by $c/|\!\log \eps|$, uniformly over the starting point.  Let
$$
L_j=\sum_{i=0}^{j}\one (B(\delta i)\in D), \qquad  L_*=L_{\lfloor  1/\delta \rfloor}.
$$
Let $\ell=\lceil \eps^2/\delta\rceil$ and $J_k$ be the smallest $j$ so that $L_j=\ell k$. By the strong Markov property of $B$ applied at time $\delta J_k$, we have
$$
P(L_*\ge \ell(k+1)|L_*\ge \ell k) \le  1-c/|\!\log \eps|.
$$
This inequality implies that $L_*$ is dominated by $\ell$ times a geometric random variable $G$ with success probability  $c/|\!\log \eps|$.

By this and the strong Markov property, applied at time $i\delta$, the number $M_i$ of total returns up to time $1$ to a box containing $B(i\delta)$ is dominated by $\ell G$.

By the union bound and and the tail bound for the geometric random variable, for all $a \ge 0$ we have
$$
P(\max_i M_i>a)\le \lceil \delta^{-1} \rceil P(\ell G >a)=\lceil \delta^{-1} \rceil (1-c/|\!\log \eps|)^{\lfloor a/\ell \rfloor}.
$$
Assuming $a\ge 1$ allows us to simplify this and get the desired bound for Brownian motion.

For Brownian bridge $R$ we prove the result for the time intervals $[0,1/2]$ and $[1/2,1]$ separately. There, it follows from the Brownian motion case and absolute continuity, Proposition \ref{p:B-R}.  
\end{proof}

\begin{corollary}\label{c:large-t} Let $R$ be planar Brownian bridge run until an arbitrary time with arbitrary endpoints, and let $R'$ be an independent copy.
For every $\gamma\ge 0$, $p<1$, there is an event  $A$  so that $$P(A)>p, \qquad \mbox{ and }\qquad E\exp\{\one_{A\times A} \gamma \alpha (R,R')\}<\infty.$$
\end{corollary}

\begin{proof} By scaling, it suffices to show this for bridges run until time 1. 
Fix $1/4<\kappa' <\kappa<1/2$.  As $\delta\to 0$, set  $\eps=\delta^{\kappa'}$.  By L\'evy's modulus of continuity and Lemma \ref{l:occupation},
$$
P\Big(\|B\|_{\kappa}\le \frac{\eps}{5\delta^\kappa}\Big)\to 1,\;\;P\Big(M\le \frac{1}{6\sqrt{\gamma\delta}}\Big)\ge 1-\frac{c'
}{\delta}\exp\Big\{-\frac{c\sqrt{\delta}}{\sqrt{\gamma}\eps^2|\!\log \eps|}\Big\}\to 1, $$ and for the intersection $A_\delta$ of these events $P(A_\delta)\to 1$. By Proposition \ref{p:nodrift}, 
\[
E\exp\{\one_{A_\delta\times A_\delta}\gamma \alpha (R,R')
\}<\infty.\qedhere\]
\end{proof}

\subsection{Uniform exponential moments for large drift}

We need a bound on the exponential moments which improves with the distance of the starting points. First we show such a bound for short times.

\begin{lemma}\label{l:drifted2}

Let $B,B'$ be independent planar Brownian motions started at the points $(x,y),(x',y')\in \mathbb R^2$  with identity covariance matrix and a shared drift $(0,N), N \ge 0$ and run until time $t$.
Let $0<\gamma,t$ and $\gamma t\le \pi/4$. Then, there is a constant $c$ such that
\begin{equation}\label{eq:61}
Ee^{\gamma \alpha(B,B')}\le 1+  \frac{c\gamma t}{1+N\sqrt{t}} e^{-\frac{|x-x'|^2}{8t}} \le \exp\big\{ \tfrac{c\gamma t}{1+N\sqrt{t}} e^{-\frac{|x-x'|^2}{8t}}\big\}.
\end{equation}
and
\begin{equation}\label{e:drifted2nd}
E\big[\alpha(B,B')^2\big]\le \frac{c t^2}{1+N^2t}e^{-\frac{|x-x'|^2}{8t}}.
\end{equation}
\end{lemma}
\begin{proof}
By symmetry, we may assume $(x',y')=(0,0)$ and $x\ge 0$. By Lemma \ref{l:scaling}, we may assume $t=1$. Let $\tau, \tau'$ be the first times the first coordinate of $B$, respectively $B'$,  equals $x/2$. By symmetry, $\tau'$ is an independent copy of $\tau$ and
\begin{equation}\label{eq:A}
    P(\tau\le 1)   =P(\max_{t\in[0,1]} B_1(t)\ge x/2)=P(|B_1(1)|\ge y/2)\le e^{-{x}^2/8}
\end{equation}Let $A_0=[0,\tau)$, and  $A_1=[\tau,1]$. Define $A_i'$ analogously.
By bilinearity (Lemma  \ref{l:bilinear}),
\begin{equation}\notag
\alpha(B,B') = \sum_{i,j\in\{0,1\}} \alpha(B_{A_i},B'_{A'_{j}}) \le
\alpha(B_{A_1},B') + \alpha(B,B'_{A'_1}),
\end{equation}
since $\alpha(B_{A_{0}},B'_{A'_{0}})=0$, and $\alpha$ is nonnegative. 
Cauchy-Schwarz gives
\begin{equation}\notag
E e^{\gamma \alpha(B,B')}\le  E[e^{2\gamma \alpha(B_{A_{1}},B')}]^{1/2}E[e^{2\gamma \alpha(B,B'_{A'_{1}})}]^{1/2}= Ee^{2\gamma \alpha(B_{A_{1}},B')},
\end{equation}
since the two terms are equal, by symmetry.
By conditioning on the $\sigma$-field at time  $\tau$ and using Lemma \ref{l:drifted}, 
\begin{equation}\notag
E[e^{2\gamma \alpha(B_{A_{1}},B')} -1]\le P(\tau\le 1) \frac{c\gamma}{1+N}.
\end{equation}
By the same argument, \begin{equation}\notag
E \alpha^2(B_{A_1},B') \le P(\tau\le 1) \frac{c}{1+N^2}.
\end{equation}
Together with \eqref{eq:A} these give \eqref{eq:61}
and \eqref{e:drifted2nd}.
\end{proof}
The following lemma will be needed in the proof of Proposition  \ref{p:tightness1}. It estimates the exponential moments of a smooth version of local time at zero for Gaussian random walks. 
\begin{lemma}\label{l:local-sum} Let $B$ be standard one-dimensional Brownian motion.
For every $\gamma,\lambda >0$, and integer $k\ge \gamma^2$, for the sum over integer times,
\begin{equation}\notag
E\exp \left\{ \frac{\gamma}{\sqrt{k}}\sum_{i=0}^{k-1} e^{-\lambda B^2(i)}\right\} \le 2\exp\{4\pi \gamma^2e^{\lambda}/\lambda\}.
\end{equation}
\end{lemma}
\begin{proof}
For $y\in[0,1]$, we have $e^{-y}\le 1-y/2$, so for $a\in [0,1]$,
\begin{align}\notag
&E\exp\left\{ -a\int_0^1 e^{-\lambda(B(t)-x)^2}dt\right\}
\le 1-\frac{a}{2}E\int_0^1 e^{-\lambda(B(t)-x)^2}dt.
\end{align}
By Jensen's inequality applied to both the expectation and time integral, this is at most
\begin{align}\notag
 1-\frac{a}{2} \exp\left\{-\lambda \int_0^1 E(B(t)-x)^2\, dt\right\}=1-2a_1e^{-\lambda x^2}
\le \exp\{ -2a_1e^{-\lambda x^2}\}, 
\end{align}
where $
a_1=\frac{ae^{-\lambda/2}}{4}.$
This implies that
\begin{equation}\notag
E[\exp \{-a\int_i^{i+1}  e^{-\lambda B^2(t)}dt\}\mid\mathcal F_i]\le \exp \{-2a_1e^{-\lambda B^2(i)}\}
\end{equation}
and so
\begin{equation}\notag
M_j=\exp\left\{2a_1\sum_{i=0}^{j-1} e^{-\lambda B^2(i)}-a\int_0^j e^{-\lambda B^2(t)} \,dt\right\}
\end{equation}
is a supermartingale. By Cauchy-Schwarz
\begin{equation}\notag
E\exp\left\{a_1\sum_{i=0}^{k-1}e^{-\lambda B^2(i)}\right\}\le (E M_k )^{1/2} \left(E \exp\left\{a\int_0^k e^{-\lambda B^2(t)} \,dt\right\}\right)^{1/2}
\end{equation}
with $EM_k\le 1$. Setting $a_1=\gamma/\sqrt{k}$, $a= 4\gamma e^{\lambda/2}/\sqrt{k}$,  $a'=4\gamma\sqrt{\pi/(\lambda k)}e^{\lambda/2}$,
\begin{equation}\notag
E \exp\left\{a\int_0^k e^{-\lambda B^2(t)} \,dt\right\}=E \exp\left\{ a' \int \sqrt{\lambda/\pi}e^{-\lambda x^2}\ell (k,x) dx \right\}
\end{equation}
where $\ell(k,x) = \int_0^k \delta_0( B(t)-x) dt$ is
the local time up to time $k$ at $x$.
Since $\int \sqrt{\lambda/\pi}e^{-\lambda x^2} dx=1$,
we have by Jensen's inequality,
\begin{equation}\notag
E \exp\left\{  a'\int \sqrt{\lambda/\pi}e^{-\lambda x^2}\ell (k,x) dx \right\}\le E \int \sqrt{\lambda/\pi}e^{-\lambda x^2} \exp\{  a'\,\ell (k,x) \}dx.
\end{equation}
By stochastic domination,
\begin{align*}
E\exp\{  a'\,\ell (k,x) \} & \le E \exp\{  a'\,\ell (k,0) \}= E\exp\{4\gamma\sqrt{\pi/\lambda }e^{\lambda/2}|B_1|\}\\&\le 2\exp\{8\pi \gamma^2e^{\lambda}/\lambda\}
\end{align*}
since $\ell(k,0)$ has the same distribution as  $|B(k)|$ and $\sqrt{k}|B(1)|$ and using $Ee^{|X|}\le Ee^X +Ee^{-X}$.
\end{proof}

\begin{proposition}\label{p:tightness1}
Let $B, B'$ be two independent standard planar Brownian motions with drift $(0,N)$
started at some arbitrary $z,z'\in\mathbb R^2$ on the time interval $[0,1]$.
 There is a $c<\infty$ such that for all $N,r \ge 1$, and $0\le \gamma\le \tfrac{\pi}{8r}$,
$$
E\big[e^{\gamma N \alpha(B,B')}\mathbf 1_{A_r\times A_r }\big] \le e^{c\gamma^2 },\qquad A_r:=\Big\{\sup_{t\in [0,1]}|B_2(t)-N t|\le r/2\Big\},
$$
 where $A_r\times A_r$ is defined in (\ref{e:times-notation}).
\end{proposition}
\begin{proof}
Consider overlapping time intervals $$I_j=[j\tfrac{r}{N} ,(j+1)\tfrac{r}{N} ), \qquad J_j=[(j-\tfrac12)\tfrac{r}{N} , (j+\tfrac12)\tfrac{r}{N} ).$$ On  $A_r$, the restrictions  $B_{I_j}$ and $B'_{I_{i}}$ of the paths to those intervals can intersect only when $|i-j|\le 1$. 
Even if $j=i+1$,  the intersection can only happen  at a time in $J_j$. This implies that on $A_r\times A_r$,
\begin{equation}\notag
\alpha(B,B') \le \alpha_1+ \alpha_2, \qquad \alpha_1= \sum_{j=0}^{k-1}\alpha (B_{I_j},B'_{I_j}), \qquad \alpha_2=\sum_{j=1}^{k-1} \alpha(B_{J_j},B'_{J_j}).
\end{equation}
where  $k= \lceil N /r\rceil $.
Let  $\gamma'=2 N \gamma$, so that by assumption $\gamma' r/N\le \pi/4$.  Let $Y_j$ denote the first coordinate of $B_{jr/N}-B'_{jr/N}$.
By Lemma \ref{l:drifted2},
\begin{equation}\notag
E\Big[\exp \{\gamma'  \alpha (B_{I_i},B'_{I_i})\}\,\big|\,\mathcal F_{ir/N}\Big]\le \exp\{ de^{-\frac{1}{8r}Y_i^2}\}, \qquad d=\tfrac{c
\gamma' r/N }{1+N  \sqrt{r/N}}\le \tfrac{c'\gamma \sqrt{r}}{\sqrt{N}},
\end{equation}
so the process
\begin{equation}\notag
M_j=\exp  \Big\{\sum_{i=0}^{j-1} \gamma'   \alpha (B_{I_i},B'_{I_i})-de^{-\frac{1}{8t}Y_i^2}\Big\}
\end{equation}
is a supermartingale.  By Cauchy-Schwarz,
\begin{equation}\label{problemestimate}
(Ee^{2\gamma N \alpha_1})^2 \le EM_{k} \;E\exp  \sum_{i=0}^{k-1} \tfrac{c\gamma \sqrt{r}}{\sqrt{N}} e^{-\frac{1}{8r}Y_i^2}\le 2\exp(c\gamma^2 ),
\end{equation}
where the last inequality uses $EM_k\le 1$ and Lemma \ref{l:local-sum}. The same bound holds for $\alpha_2$, and we conclude by Cauchy-Schwarz.
\end{proof}

An important shortcoming of Proposition  \ref{p:tightness1} is that the range  $[0,\pi/(8r)]$ of  exponents $\gamma$ for which it gives a bound gets worse with the truncation parameter $r$. The remedy is provided by cutting off paths with large H\"older norm, equation \eqref{e:holder}).

\begin{proposition}\label{p:tightness2}

Let $B, B'$ be two independent standard planar Brownian motions with drift $(0,N)$ on the time interval $[0,1]$. 
Alternatively, let $B,B'$ be Brownian bridges from $(0,0)$
 to $(0,N)$ on the same time interval. 
Let $0<\kappa<1/2$.
Then there exists a constant $b>0$ so that for all $\gamma>0$, $r>1$, and  $N\ge cr((\gamma r)^3+1)$  we have
\begin{eqnarray*}
&E\left[e^{\gamma N \alpha(B,B')}\mathbf 1_{H\times H}\right] \le 2 e^{b\gamma^2}, \qquad H=\{\|B\|_\kappa\le r/2\},
\end{eqnarray*}
where $H\times H$ is defined in (\ref{e:times-notation}).
\end{proposition}

\begin{proof}
We start with the Brownian motion case. 
Let $A_{r,t}$ be the event that
$ \sup_{s\in[0,t]} |B_2(s)-Ns|\le r$. Let
\begin{equation}\notag
\mathfrak{e}(t,N,r,\gamma) =\sup E\big[e^{\gamma N \alpha(B,B')}\mathbf 1_{ A_{r,t}\times A_{r,t}}\big] ,
\end{equation}
where the sup is over all starting points for $B,B'$, which still have a shared drift $(0,N)$.

Consider the time intervals $I_j=[\tfrac{j}{m},\tfrac{j+1}{m}]$, $J_j=[(j-\tfrac12)\tfrac1m,(j+\tfrac12)\tfrac1m]$. When $1/m>r/N$ and $H\times H$  holds, then $B_{I_j}$ and $B'_{I_{i}}$ have disjoint support unless $|i-j|\le 1$. Moreover, if $j=i+1$ they can only have intersecting support within the time interval $J_j$. This implies that
\begin{equation}\notag
\alpha(B,B') \le \alpha_1+ \alpha_2, \qquad \alpha_1= \sum_{j=0}^{m-1}\alpha (B_{I_j},B'_{I_j}), \qquad \alpha_2=\sum_{j=1}^{m-1} \alpha(B_{J_j},B'_{J_j}).
\end{equation}
Next, we  use a martingale argument to bound exponential moments of $\alpha_1$ and $\alpha_2$ separately. This will allow us to treat $\alpha (B_{I_j},B'_{I_j})$ for different  $j$  as if they were independent. Let  $$G_i=\Big\{\sup_{s,t\in I_i} |B(s)-B(t)-(0,N)(s-t)|\le r/m^{\kappa}\Big\}.$$
Abbreviate
\begin{equation}\notag\mathfrak{e}=\mathfrak{e}(1/m,N,rm^{-\kappa},\gamma)=\mathfrak{e}(1,N/m^{1/2},rm^{1/2-\kappa},\gamma/m^{1/2}),\end{equation}
where the equality is by scaling, Lemma \ref{l:scaling}. Then $$
R_j=\mathfrak{e}^{-j} \prod_{i=0}^{j-1}\mathbf 1_{G_i\times G_i}e^{N\gamma\alpha(B_{I_i},B'_{I_i})}
$$
is a supermartingale, since by the definition of $\mathfrak{e}$, we have
$$
E[R_{j+1}|\mathcal F_{j/m}]=R_j E\left[\one_{G_{j+1}\times G_{j+1}}e^{N\gamma \alpha(B_{I_{j+1}},B'_{I_{j+1}})}\Big|\mathcal F_{j/m}\right]/\mathfrak{e}\le  R_j.
$$
Here $\mathcal F_t$ is the natural filtration for $(B,B')$.
Since on $H$ all $G_i$ occur, we have
\begin{equation}\notag
E[e^{\gamma N\alpha_1}\mathbf 1_{H\times H}]\le \mathfrak e^m E R_m\le \mathfrak e^m.
\end{equation}
Similarly,
$E[e^{\gamma N\alpha_2}\mathbf 1_{H\times H}]\le \mathfrak  e^{m-1}$. By Cauchy-Schwarz and Proposition  \ref{p:tightness1} we get
\begin{equation}\notag
E[e^{\gamma N\alpha(B,B')}\mathbf 1_{H\times H}]\le \mathfrak e(1,N/m^{1/2},rm^{1/2-\kappa},2\gamma/m^{1/2})^{m} \le e^{4b'\gamma^2}
\end{equation}
as long as $2\gamma rm^{-\kappa}\le \pi/8$. The Brownian motion case follows with
$m=\lceil (16\gamma r/\pi)^{1/\kappa}\rceil$.

For the bridge, on $\mathbf 1_{H\times H}$, for $N$ large enough, we have $$\alpha(B_{[0,1/3]},B'_{[2/3,1]})=\alpha(B_{[2/3, 1]},B'_{[0,1/3]})=0.$$  Since   $\alpha$ is bilinear and non-negative we have
\begin{equation}\label{sumoftwo}
\alpha(R,R')\le \alpha(R_{[0, 2/3]},R'_{[0, 2/3]})+\alpha(R_{[1/3,1]},R'_{[1/3, 1]}).
\end{equation}
The two quantities have the same distribution by symmetry. 
The first two-thirds of the bridge is absolutely continuous with bounded derivative with respect to the first two-thirds of the motion,  Proposition \ref{p:B-R}. The proof concludes by Cauchy-Schwarz. 
\end{proof}

\section{Proofs of the main theorems}\label{s:ikea}

In this section we prove Theorems \ref{t:solution-intro},  \ref{t:2dpolymer-intro}, \ref{t:1dpolymer-intro} and \ref{t:process}. Theorem \ref{t:crossover} is a special case of Theorem \ref{t:process}. The hard work has been done, but the pieces need to be assembled. 

\subsection{Proof of Theorems \ref{t:solution-intro} and \ref{t:2dpolymer-intro}}
{\noindent \bf Explicit solution and uniqueness.} Definition \ref{d:Wick-ordered -she-intro} is what we rigorously mean by a solution of the Wick-ordered  heat equation. Proposition \ref{p:spde-projection} shows that the  projections $E_P[u(x,t)|\mathcal F_n]$  satisfy a finitary version of the equation. These projections are martingales. 

Proposition \ref{p:fshe-solve} shows that the finitary version of the Wick-ordered  heat equation has a unique solution $u_n(x,t)=p(x,t)Z_n(x,t)$ with 
\begin{equation}\label{e:Znrepeat}
Z_n(x,t)=E_{Q}\exp \bigg\{ \sum_{j=1}^n m_j\xi_j - \tfrac12 \sum_{j=1}^n m_j^2\bigg\}, \quad m_j=\int_0^t \basis_j(B(s))\,ds.
\end{equation}
Proposition \ref{p:solving2d} shows that if the martingale $Z_n(x,t)$ is uniformly integrable for all $x,t$, then its limit is the unique solution of the Wick-ordered  SHE. So we fix $x,t$.

In Section \ref{randomizedshifts}, we identify the martingale $Z_n$ as a randomized shift.  Proposition \ref{p:shifts-existence}.\ref{p:shifts-existence-hard} gives criteria for $Z_n$ to converge in $L^1$  in terms of the intersection exponential $\mathfrak e(m)$ (see \eqref{eqthing}). 
Here
$$\mathfrak e(m)=\lim_{n\to \infty} E_{Q^{\otimes 2}}e^{\alpha_n}, \qquad \alpha_n=\sum_{i=1}^n m_im_i'.$$ is defined in terms of  two independent copies of  $m$  in \eqref{e:Znrepeat}. Proposition \ref{p:alpha-plane} shows that $\alpha_n$ has a limit $\alpha\ge 0$ in $L^k$ for every $k>0$, the mutual intersection local time of two independent copies of $B$.  Lemma 
\ref{l:i-exp bound} shows that for $\alpha \ge 0$ we have $\mathfrak e(m)\le Ee^{\alpha}$. 

By Proposition \ref{p:shifts-existence}.\ref{p:shifts-existence-hard}, to show martingale convergence,  it suffices to show that for every $p<1$ there is set of paths $A$ with $QA>p$
so that $\mathfrak e(\one_Am)< \infty$. 

More concretely, for Brownian bridges, we need to show that there is a set of paths $A$ with probability arbitrarily close to 1 so that the mutual intersection local time satisfies $E_{Q^{\otimes 2}}\exp\{\one_{A\times A}\,\alpha\}<\infty$. 

For the set $A$, we control two properties of the paths: a H\"older norm and the maximal mass of a discretized version of the occupation measure, see \eqref{e:holder}. 

Proposition \ref{p:nodrift} shows that under such control, the exponential moments of the intersection local time exist, and Lemma \ref{l:occupation} bounds the maximal mass. Using these results, Corollary \ref{c:large-t} provides the set $A$ of paths with $QA>p$ so that $E_{Q^{\otimes 2}}\exp\{\one_{A\times A} \alpha\}<\infty$.  
Thus we have found the solution of the Wick-ordered  heat equation. 

{\noindent \bf Convergence in $L^2(P)$.} For $t<t_c$ we have $Ee^\alpha<\infty$. By Lemma \ref{l:i-exp bound}, we have $\mathfrak e(m)<\infty$, and Proposition \ref{p:shifts-existence} implies that $Z_n \to Z$ in $L^2(P)$.

{\noindent \bf Basis independence.} The solution we have found is the partition function of a randomized shift with $\alpha$ given by the mutual intersection local time. Proposition \ref{p:milt-unique} shows that when the mutual intersection local time exists, it is unique. In particular, it does not depend on the basis $\basis_j$. The partition function $Z$ is basis-independent by Remark \ref{rem:unique}.  Uniqueness in this type of construction is due to \cite{shamov} (Theorem 17 and Corollary 18), which  show that the underlying extended Gaussian covariance kernel uniquely determines the partition function $Z$. For us, the kernel is on path space $\Omega$, and is given by the mutual intersection local time $\alpha$.
Theorem \ref{t:solution-intro} follows.

{\noindent \bf Polymer measure.} The randomized shift representation has been established along the way in Proposition \ref{p:shifts-existence}.\ref{p:shifts-existence-hard}. It gives rise to a unique polymer measure by Theorem \ref{t:shamov-polymer}. This completes the proof of Theorem \ref{t:2dpolymer-intro}. Proposition \ref{p:polymer-limit} implies the claimed  polymer convergence. \qed
\medskip

\begin{remark}This proof extends to the case when the initial condition $\varsigma$ 
is a general probability measure. 
\end{remark}

\subsection{Proof of Theorem \ref{t:1dpolymer-intro}}
\label{ss:1dshe-solve}

{\noindent \bf Explicit solution and uniqueness.} By Proposition \ref{p:1dshesolve}, the conditional expectation $z_n=E_Q[z|\mathcal F_n]$ of the equation has the unique solution given by
$$
z_n(x,t)=Z_n(y,0;x,t)q(x-y,t)
$$
$$ Z_n=Z_n(0,y;t,x)=E_Q\exp\left(\sum_{j=1}^n m_j\xi_j-\frac{1}{2}m_j^2\right),
  \quad m_j=\int_0^t e_j(b(s),s)ds.
$$
and  $b(s)$ is a one-dimensional Brownian bridge between space-time points $(0,0)$ to $(x,t)$.

We need to show that this is an $L^2(P)$ martingale, that is, it has a bounded $L^2$-norm. By Proposition \ref{shift1} we have
$$
E_PZ_n^2=E_{\Qt}e^{\alpha_{n}}, \qquad \alpha_n=\sum_{i=1}^n m_jm_j'.
$$
By Proposition \ref{p:alpha-one}, $\alpha_n$ has a limit in every $L^k(\Qt)$.  The limit  $\alpha(b,b')$ is the mutual intersection local time of two independent copies of $(b(s),s)$ up to time $t$, and $E_{\Qt}\cosh \alpha<\infty.$   By Lemma \ref{l:i-exp bound}, we 
$$
E_{P}[Z_n^2]=E_{\Qt}\,e^{\alpha_n}\le E_{\Qt}\,\cosh {\alpha}.
$$
so $Z_n\to Z$ in $L^2(P)$. With this construction, taking  $\mathcal F_n$-conditional expectations as in Proposition \ref{p:1dshe-project}, we see that for every $n$ the $\mathcal F_n$-conditional expectations of the two sides  of  the equation \eqref{inteq'} of Definition \ref{d:1dshe} agree. So the limit $z$ is indeed a solution. We have thus found the unique solution in an explicit form. 

{\noindent \bf Agreement with It\^o.} The Skorokhod integral formulation in Definition \ref{d:1dshe}
of the one-dimensional Wick-ordered  stochastic heat equation coincides with the mild It\^o formulation. This is because by Remark \ref{r:Ito1+1}, the two integrals are the same. Since there is a unique solution, the solution in \cite{AKQ2} agrees with ours. 

{\noindent \bf Polymer measure and CDRP.}  Associated to the partition function $Z$ there is a random polymer measure $M$ on $\Omega$ defined by the formula \eqref{shiftkpz}, see Section \ref{ss:polymer}. We claim that it coincides with the continuum directed random polymer constructed in \cite{AKQ2}, Definition 4.1. Their construction is a random continuous function on $[0,1]$, determined by its finite dimensional distributions given the noise $\xi$. Since $\Omega$ is the space of continuous functions, it suffices to show that their finite dimensional distributions agree with ours. This is the content of the next lemma. Let $z(x,s;y,t)=q(y-x,t-s)Z(x,s;y,t)$ denote the solution of the SHE started at time $s$ from the initial condition $\delta_{x}$.  
\begin{lemma} Let $M$ be the unnormalized  polymer measure on paths $b$ from 
${\rm p}$ to ${\rm q}$, let $k\ge 1$ and 
set $(x_0,s_0)=\rm p$, $(x_{k+1},s_{k+1})=\rm q$.
Pick $s_1<\dots <s_k$ in $(s_0,s_{k+1})$. 
Let $\varphi$ be a bounded nonnegative test function on $\mathbb R^k$. Then $P$-a.s., 
$$
E_M\varphi(b(s_1),\dots ,b(s_k))=\int\,\varphi(x_1,\ldots, x_k) \prod_{i=0}^k z(x_i,s_i;x_{i+1},s_{i+1})\, dx_1\dots dx_k 
$$
\end{lemma}
\begin{proof}
To keep the notation manageable, let $$\varphi=\varphi(b(s_1),\ldots, b(s_k)), \qquad  Z_i=Z(b(s_i),s_i;b(s_{i+1}),s_{i+1}).$$ By expressing the integral as a $Q$-expectation and accounting for the heat equation $q$ factors between $z$ and $Z$, the claim translates to 
\begin{equation}\label{e:CDRP-Q}
E_M\varphi=E_Q\big[\,\varphi\,Z_0\cdots\,Z_k\,\big] \qquad P- \mbox{a.s.}
\end{equation}
Let $F_i$ be a bounded nonnegative test function on $(\Xi, \mathcal F, P)$ depending only on the noise at times in between $s_i$ and $s_{i+1}$. 
Let $F=F_0\cdots F_k$. 
Let 
$\mathcal S\subset \mathcal G$ denote the $\sigma$-field generated by $b(s_1),\ldots, b(s_k)$. By definition of the polymer measure,  Fubini, and since $\varphi$ is $\mathcal S$-measurable, we have 
\begin{align}\notag
E_PE_M[\varphi F]&=E_PE_Q[\varphi F(\xi+m)]=E_PE_QE_Q[\varphi F(\xi+m)|\mathcal S]
\\ \label{e:phiF}&=
E_Q\Big[\varphi\,E_PE_Q[F(\xi+m)|\mathcal S]\Big]. 
\end{align}
Given $\mathcal S$ and $\xi$ fixed, by the Markov property of $b$ the random variables  $F_0(\xi+m),\ldots, F_k(\xi+m)$ are independent over $Q$, so 
$$
E_Q[F(\xi+m)|\mathcal S]=\prod_{i=0}^kE_Q[F_i(\xi+m)|\mathcal S].
$$
Because the noises between times $s_i$ and $s_{i+1}$ are $P$-independent as $i$ varies, we have
$$
E_PE_Q[F(\xi+m)|\mathcal S]=\prod_{i=0}^k E_PE_Q[F_i(\xi+m)|\mathcal S]=\prod_{i=0}^k E_P[Z_iF_i],
$$
by the shift representation of the partition functions $Z_i$. Since the factors are independent over $P$,  substituting this into \eqref{e:phiF} we get
$$
E_PE_M[\varphi F]=E_QE_P[\varphi \prod_{i=0}^kZ_iF_i]=E_PE_Q[\varphi Z_0\cdots Z_kF].
$$
As this holds for any product test function $F$, the claim \eqref{e:CDRP-Q} follows. 
\end{proof}

Next, we prove  the main convergence theorem. 

\subsection{Proof of Theorem \ref{t:process}, tightness and polymer limits}
\label{ss:tconvergence}
We continue the proof started in the introduction.
 
{\noindent \bf Tightness.} We need to justify the interchange of limits in $n$ and $N$. For the rest of the proof, we will assume that all scaling parameters are standard, namely $\nu=\rho=\beta=s=1$, $x=t=0$ and only $y$ is free, since the general argument is identical by scaling, see Lemma \ref{l:scaling}.

A condition for the interchange of limits is established in Proposition \ref{propmain}.\ref{propmain-hard}. It is in terms of the self-intersection local times 
$$
\alpha_N=\lim_{n\to \infty}\sum_{i=1}^n m_{N,i}m'_{N,i}, \qquad \alpha=\lim_{n\to \infty}\sum_{i=1}^n m_{i}m'_{i}. 
$$
defined on $\Omega^2$, with prime denoting independent copies. See Section \ref{ss:intersectionlocaltime} for more explanation about the $\alpha$s. 

One condition is that the limits defining $\alpha_N$ and $\alpha$ should exist in $L^2$; this is established in Propositions \ref{p:alpha-plane} and \ref{p:alpha-one}.
Next, we need $\alpha_N\to \alpha$ in law, which, by Lemma \ref{l:alpha-c-law}, follows from $E\alpha_N^2\to E\alpha^2$. This is established by an explicit computation in Proposition \ref{convofalphas}. Finally we need a tightness condition: there is $\gamma>1$ so that given $p<1$  we have a set of paths $A_N$ for every $N$, so that  $QA_N>p$ and
$$
\limsup_{N\to\infty} \mathfrak e(\gamma \one_{A_{N}} \mathbf m_N) =\limsup_{N\to\infty}\lim_{n\to\infty}E_{Q^{\otimes^2}}\exp\sum_{i=1}^n \gamma^21_{A_N}m_i1_{A'_N}m'_i<\infty.
$$
Since $\alpha_N$ are nonnegative, by Lemma \ref{l:i-exp bound} it will be enough to show that $\limsup_N E[{e^{\gamma \alpha_N}}\one_{A\times A}]<\infty$. We take $A$ to be a set of paths with bounded H\"older-$\kappa$ norm \eqref{e:holder} for $\kappa\in (0,1/2)$. The $L^1(P)$ convergence claim then follows from Proposition \ref{p:tightness2}.

{\noindent \bf Polymer convergence.} For every $N$, the formula \eqref{shdef} gives a $P$-random measure $M_N$ on $Q$. These measures are modifications of the law of the pair $(b_{\rm sp}, b_{\rm ti})$. Similarly, the limiting measure $M$ modifies the law of this pair, but in a way that only depends on the first process. 

Proposition \ref{p:polymer-limit} implies that $M_n\to M$ in probability with respect to weak convergence of measures. In other words,
$$
(b_{\rm sp},b_{\rm ti}) \mbox{ under }M_N \;\;\stackrel{\rm law}{\Longrightarrow} \;\;(b_{\rm sp},b_{\rm ti}) \mbox{ under }M.
$$
By Theorem \ref{t:1dpolymer-intro}, under $M$, the process $b_{\rm sp}$ is the continuum directed random polymer (CDRP), and $b_{\rm ti}$ is an independent Brownian bridge. 
This implies the following laws converge in probability
\begin{align*}
\Big((N^{-1/2}&B_N(N\rho r)_1,N^{-3/2}B_N(N\rho r)_2), \;r\in[0,s]\Big) \mbox{ under $M_N$} \\& \;\;\stackrel{\rm law}{\Longrightarrow} \;\; \Big((b_{\rm sp}(t+r),t+r), \;r\in [0,s]\Big) \mbox{ under $M$}.
\end{align*}

This completes the proof of Theorem \ref{t:process}.

\begin{remark}The proof also shows that the fluctuations $b_{\rm ti}$ in the limiting time-direction converge to an unweighted Brownian bridge independent of the limiting polymer. Also, since $Z_N\to Z$ and $Z_N,Z>0$ a.s.\;by the 0-1 law, Lemma \ref{01law}, the normalized polymer measures converge as well: $M_N/Z_N\to M/Z$. 
\end{remark}

\medskip
\noindent {\bf Acknowledgments. } JQ and BV are partially supported by NSERC discovery grants. AR is partially supported by Fondecyt 1220396.

\bibliographystyle{dcu}
\bibliography{up}

\bigskip\bigskip\noindent

\bigskip

\noindent Jeremy Quastel, Department of Mathematics, University of Toronto, \\ Canada, {\tt quastel@math.toronto.edu}

\bigskip

\noindent Alejandro Ram\'irez, Departamento de Matem\'aticas,
Pontificia Universidad Cat\'olica de Chile,
{\tt aramirez@mat.uc.cl}

\bigskip

\noindent B\'alint Vir\'ag, Departments of Mathematics and Statistics, University of \\Toronto, Canada, {\tt balint@math.toronto.edu}

\end{document}